\documentclass[10pt]{elsarticle-ppn}

\makeatletter
\let\c@author\relax
\makeatother

\usepackage{fancyhdr}
\usepackage{comment}
\usepackage{marginnote}
\usepackage{booktabs}
\usepackage{multirow}
\usepackage[scr=rsfs,cal=boondox]{mathalfa}
\usepackage{soul}

\usepackage[backend=biber, citestyle=numeric-comp, bibstyle=trad-abbrv, url=false, doi=false, isbn=false]{biblatex}
\makeatletter
\def\namedlabel#1#2{\begingroup
    #2%
    \def\@currentlabel{#2}%
    \phantomsection\label{#1}\endgroup
}
\makeatother

\usepackage{amsmath}
\usepackage{mathtools}
\usepackage{amsfonts}
\usepackage{amsthm}
\usepackage{amssymb}
\usepackage{color}
\usepackage{dsfont}
\usepackage{geometry}
\usepackage{todonotes}
\usepackage{yhmath}
\usepackage[normalem]{ulem}
\newgeometry{tmargin=2.8cm, bmargin=3.5cm, lmargin=1.7cm, rmargin=1.7cm}

\usepackage{mathrsfs}
\usepackage{dutchcal}
\usepackage{hyperref}

\newcommand{\R}{{\mathbb{R}}}
\newcommand{\rn}{{\mathbb{R}^N}}

\newcommand{\rp}{{[0,\infty)}}

\newcommand{\N}{{\mathbb{N}}}

\newcommand{\wt}{\widetilde}
\newcommand{\wh}{\widehat}
\newcommand{\vp}{{\varphi}}
\newcommand{\vr}{{\varrho}}
\newcommand{\Wphi}{W^{1, 1}_{\vp}}
\newcommand{\Sphere}{{\mathbb{S}}}

\newcommand{\mres}{\mathbin{\vrule height 1.6ex depth 0pt width
0.13ex\vrule height 0.13ex depth 0pt width 1.3ex}}
\newcommand{\ve}{{\varepsilon}}
\newcommand{\vt}{{\vartheta}}
\newcommand{\vk}{{\varkappa}}
\newcommand{\mv}{{\mathsf{v}}}

\newcommand{\mw}{{\mathsf{w}}}

\newcommand{\Gu}{{\mathrm{Graph}_u}}
\newcommand{\Gr}{{\mathrm{Graph}}}
\newcommand{\constK}{K}
\def\d{{\,{\rm d}}}
\newcommand{\dx}{{\,{\rm d}x}}
\newcommand{\dy}{{\,{\rm d}y}}
\newcommand{\dt}{{\,{\rm d}t}}

\newcommand{\dhf}{{\,{\rm d}\wh{f}}}
\newcommand{\dho}{{\,{\rm d}\wh{1}}}

\newcommand{\unp}{{{u}_i^{s,\vp}}}

\newcommand{\bun}{{\overline{u}_n^s}}

\newcommand{\bunp}{{\overline{u}_n^{s,\vp}}}

\newcommand{\bunpj}{{\overline{u}_n^{\vp,j}}}
\newcommand{\odel}{\overline{\delta}}
\newcommand{\Pnd}{{\Phi_{n,\odel}}}

\newcommand{\cE}{{\mathscr{E}}}

\newcommand{\cEp}{{\mathscr{E}_\vp}} 
\newcommand{\cC}{{\mathcal{C}}}

\newcommand{\wcC}{\wt{\cC}_\delta}
\newcommand{\cH}{{\mathscr{H}}}
\newcommand{\ca}{{\mathcal{a}}}
\newcommand{\dxodt}{{\cH^{N+1}}}

\newcommand{{\SP}}{{\mathcal{Z}}}

\DeclareMathOperator{\supp}{supp}

\DeclareMathOperator{\diam}{diam}

\DeclareMathOperator*{\esssup}{ess\,sup}
\DeclareMathOperator*{\essinf}{ess\,inf}

\usepackage{refcount}

\newcounter{cte}

\definecolor{darkgreen}{rgb}{0.00, 0.50, 0.00}

\theoremstyle{plain}
\newtheorem{proposition}{Proposition}[section]
\newtheorem{lemma}[proposition]{Lemma}
\newtheorem{theorem}{Theorem}
\newtheorem{corollary}[proposition]{Corollary}
 
\newtheorem{remark}[proposition]{Remark}

\theoremstyle{definition}
  
\newtheorem{example}{Example}[section]

\numberwithin{equation}{section}

\allowdisplaybreaks

\begin{document}
\begin{frontmatter}
    
\title{Discarding  Lavrentiev's Gap\\ in  
Non-autonomous and~Non-Convex Variational Problems  
}

\author[1]{Michał Borowski}\ead{m.borowski@mimuw.edu.pl}
\author[2]{Pierre Bousquet}\ead{pierre.bousquet@math.univ-toulouse.fr}
\author[1]{Iwona Chlebicka\corref{mycorrespondingauthor}}
\cortext[mycorrespondingauthor]{Corresponding author}\ead{i.chlebicka@mimuw.edu.pl} 
\author[3]{Benjamin Lledos}\ead{benjamin.lledos@unimes.fr}
\author[1]{Błażej Miasojedow}\ead{b.miasojedow@mimuw.edu.pl}

\address[1]{Institute of Applied Mathematics and Mechanics,
University of Warsaw, ul. Banacha 2, 02-097 Warsaw, Poland
}
\address[2]{Institut de Mathématiques de Toulouse, UMR 5219, Université de Toulouse, CNRS,
UPS, F-31062 Toulouse Cedex 9, France
}
\address[3]{Laboratoire MIPA, Université de Nîmes, Site des Carmes, Place Gabriel P\'eri, 30000 Nîmes, France
}

\fntext[myfootnote]{MSC2020: 49J45 (46E30,46E40,46A80)}
\fntext[myfootnote]{M.B. and I.C. are supported by NCN grant 2019/34/E/ST1/00120. The authors express gratitude to Initiative of Excellence at University of Warsaw IDUB that enabled several fruitful meetings and discussions.}

\begin{abstract} 
We establish that the Lavrentiev gap between Sobolev and Lipschitz maps does not occur for a  scalar variational problem of the form:
\[
\textrm{to minimize}  \qquad   u \mapsto \int_\Omega f(x,u,\nabla u)\dx \,,
\]
under a Dirichlet boundary condition. Here, \(\Omega\) is a bounded Lipschitz open set in \(\rn\), \(N\geq 1\) and the function $f$ is required to be measurable with respect to the spatial variable, continuous with respect to the second one, and  convex with respect to the last variable. Under these assumptions alone, Lavrentiev gaps may occur, as illustrated by classical examples in the literature. We identify an additional natural condition on $f$ to discard such phenomena, that can be interpreted as a balance between the variations with respect to the first variable and the growth with respect to the last one. This unifies most of the structural assumptions that have been introduced so far to prevent the occurence of Lavrentiev gaps.
Remarkably, typical assumptions that are usually imposed on $f$ in this setting are dropped here: we do not require $f$ to be bounded or convex with respect to the second variable, nor impose any condition of $\Delta_2$-kind with respect to the last variable.
\end{abstract}

\end{frontmatter}
{\small\tableofcontents}

\section{Introduction}
 We are interested in excluding the Lavrentiev phenomenon, which occurs when the infimum of a variational problem over a family of regular functions is strictly greater than the infimum taken over all functions satisfying the same boundary conditions. Understanding the phenomenon in the multidimensional calculus of variations is of significant importance due to its implications not only in mathematical analysis but also in various applied fields, such as physics, engineering, and materials science, where optimization problems frequently arise, cf.~\cite{numerics1,numerics2, IC-pocket, Buttazzo-Mizel, DeATr, Buttazzo-Belloni, MinRad} and references therein. By its very nature, excluding the Lavrentiev phenomenon is closely related to the approximation theory~\cite{AmGoSe,Bor-Chl,C-b,Dal-Maso,Gossez}. Moreover, it is intrinsically connected to the regularity theory in the Calculus of Variations. Indeed, for a Sobolev minimizer to be Lipschitz continuous,  the absence of the Lavrentiev phenomenon is clearly a necessary condition. It turns out that this is also a sufficient condition for the regularity of minimizers for a large family of functionals~\cite{BaaBy, eslemi, comi, colmar, HaOk1, HaOk2, bacomi-cv}.
 
Lavrentiev~\cite{Lav} presented the first example of a one-dimensional variational problem for which the infimum on the set of absolutely continuous functions differs from the infimum on the set of Lipschitz functions. This example was later simplified by Mani\`a~\cite{Mania}, who considered the following problem:  
\[
I_M : u\in W^{1,1}(0,1)\mapsto \int_{0}^{1}(u(x)^3-x)^2 u'(x)^6\,\d x,
\]
where the admissible functions \(u\) are subject to the boundary conditions \(u(0)=0\) and \(u(1)=1\). Then, one has 
\begin{equation}\label{eq-Mania}
\inf_{\substack{u\in W^{1,1}(0,1),\\
u(0)=0, u(1)=1}} I_M(u) < \inf_{\substack{u\in W^{1,\infty}(0,1),\\ u(0)=0, u(1)=1}} I_M(u)\,.
\end{equation}
This example demonstrates that the Lavrentiev phenomenon is not related to the smoothness of the integrand, as it can occur even when the latter is a polynomial. The example also shows that convexity with respect to the last variable does not discard the Lavrentiev phenomenon. In fact, it persists for modified versions of the above problem where the integrand is uniformly convex in the last variable~\cite{BallMizel}. 
One may wonder whether an even more elementary example exists, where the Lagrangian would depend only on \((x, u')\) or \((u,u')\). In the one-dimensional case, the answer is negative, see~\cite{Lav, AlSeCa, Cerf-Mariconda}. The three variables \((x, u, u')\) have to be summoned for the  Lavrentiev phenomenon to occur.

In the higher-dimensional case, we study the Lavrentiev gap between Sobolev and Lipschitz maps for a general scalar functional of the form
\begin{equation}\label{def-E-of-u}
    E_\Omega: u \in W^{1,1}(\Omega)\mapsto \int_\Omega f(x,u,\nabla u)\dx \,,
\end{equation}
where \(\Omega\) is a bounded Lipschitz open set in \(\rn\), \(N\geq 1\).
Our standing assumptions require that \(f\) be a nonnegative Carathéodory map which is convex with respect to the last variable. For integrands  \(f\) depending just on \((x,\nabla u)\) or just on \((u, \nabla u)\), one may be tempted to conjecture for the case \(N\geq 2\) the non-occurrence of the Lavrentiev phenomenon established when \(N=1\). It turns out that the situation is much more complex in the higher-dimensional setting. For autonomous Lagrangians, namely those which do not depend on the spatial variable \(x\), this conjecture is correct. There is no Lavrentiev phenomenon in that case, as in the one-dimensional situation~\cite{Bousquet-Pisa}. On the contrary, when \(f\) depends just on \((x,\nabla u)\), Lavrentiev phenomena have been detected for many functionals~\cite{zh, eslemi, badi, bcdfm}.

Let us conjure the functional which has attracted notable attention in the past few years~\cite{Bor, bacomi-cv, comi, colmar, bcdfm, bgs-arma}, the double phase functional:
\begin{equation}\label{eq228}
u\mapsto \int_{\Omega}|\nabla u|^p + a(x)|\nabla u|^q\,\d x\,,
\end{equation}
where \(1<p<q\) and \(a\) is a  nonnegative and bounded function
. A considerable amount of work has been achieved to obtain sharp conditions on \(p, q\) and the  {behaviour}  of \(a\) that prevent the Lavrentiev phenomenon~\cite{eslemi, badi, bcdfm, Bor, bgs-arma}. Significant effort has been devoted to extending these conditions to more general Lagrangians \(f\), which still depend only on \(x\) and \(\nabla u\). As many examples suggest, the harsh changes of growth of the Lagrangian \(f(x,\nabla u)\) with respect to \(\nabla u\) when \(x\) varies are responsible for the Lavrentiev phenomenon.
 In the case of the double phase functional \eqref{eq228}, the Lagrangian has a \(p\) growth (in terms of \(\nabla u\)) at the points \(x\) where \(a\) vanishes, whereas it has a \(q\) growth where \(a\) is strictly positive. If the variations of \(a\) are not suitably small, in terms of the difference \(q-p\) or the ratio \(q/p\), then the Lavrentiev phenomenon may occur.
The competition of several regimes of \(\nabla u\)-growth associated with several subregions of \(\Omega\) is acknowledged as the only source of Lavrentiev phenomena.  From a technical point of view, however, most of the articles discarding the Lavrentiev phenomenon have found it convenient to introduce additional growth restrictions in the \(\nabla u\) variable, such as the \(\Delta_2\) condition. For later use, let us define this condition for a general function \(f:\Omega\times \R\times \rn\to \R\): There exists \(C>0\) such that for every \((x, t, \xi)\in \Omega\times \R\times  \rn\), one has
\begin{equation}\label{eqDELTA2-intro}
f(x, t, 2\xi)\leq C(f(x, t, \xi)+1).
\end{equation}
We then write \(f(x,t,\cdot)\in \Delta_2\).
Polynomials do satisfy this assumption, which, on the contrary, fail for integrands involving the exponential map. In this paper, our main result does not require such growth conditions.

For more general Lagrangians where the \(u\) dependence comes into play, it is \emph{a priori} unclear whether this competition between different \(\nabla u\)-growth regimes is solely responsible for the Lavrentiev phenomenon. Let us recall that in the one-dimensional case, the Lavrentiev phenomenon can occur only when \(f\) depends on the three variables, emphasizing the role of \(u\) in producing the gaps between the two infima, compare~\eqref{eq-Mania}. This suggests that in the higher dimensional case, new phenomena could arise when passing from Lagrangians depending just on \((x, \nabla u)\) to Lagrangians fully depending on the three variables \((x,u,\nabla u)\). 

The main goal of this article is to address this possibility. Roughly speaking, our main results state that the spatial variations of the growth with respect to \(u\) \emph{and} with respect to \(\nabla u\) should now be taken into account. This does not mean, however, that those variables \(u\) and \(\nabla u\) have to be considered on an equal footing, for otherwise unnecessary restrictions would be artificially introduced on the \(u\)-dependence. Without entering into too many details at this stage, let us mention two important aspects:
\begin{itemize}
\item The summability of \(W^{1,p}\) functions is (much) higher than the summability of their gradients, as a consequence of the Sobolev embeddings. It is thus expectable that the conditions involving the summability exponents on \(u\) and \(\nabla u\) should not be the same.
\item More importantly, \(f\) is not assumed to be convex with respect to \((u,\nabla u)\) but just with respect to \(\nabla u\).  
\end{itemize} 
Regarding the second point, we stress the fact that the convexity assumption just with respect to \(\nabla u\) is the natural condition to impose in the general framework of the direct method of existence in the Calculus of Variations. Indeed,  this assumption is essentially equivalent to the weak lower semicontinuity of the functional. In contrast, requiring the convexity of \(f\) jointly with respect to \(u\) and \(\nabla u\) is unnecessarily restrictive in this regard.

In fact, the study of the Lavrentiev phenomenon can be placed within a broader framework.
The weak sequential lower semicontinuity of the functional \(E_\Omega\) is equivalent to the identity
\[
E_\Omega(u)=
\inf\left\lbrace \liminf_{n\to +\infty} E_\Omega(u_n) : (u_n)_{n\geq 1}\subset W^{1,1}(\Omega), u_n\rightharpoonup_{W^{1,1}} u \right\rbrace.
\]
The corresponding relaxed functional on \(W^{1,\infty}(\Omega)\) is instead defined as
\[
E^{\textrm{rel}}_{\Omega}(u):=\inf\left\lbrace\liminf_{n\to +\infty} E_\Omega(u_n) : (u_n)_{n\geq 1}\subset W^{1,\infty}(\Omega), u_n\rightharpoonup_{W^{1,1}} u\right\rbrace.
\]
The problem of the integral representation of \(E^{\textrm{rel}}_{\Omega}(u)\) amounts to the existence of some function \(g:\Omega\times \R\times \rn\to \R\) such that
\begin{equation}\label{eq253}
E^{\textrm{rel}}_{\Omega}(u) = \int_{\Omega}g(x,u(x),\nabla u(x))\,dx.
\end{equation}
One may expect that \(g=f\) under suitable conditions on \(f\). As a consequence of our main results, we can assert that the assumptions on \(f\) which discard the Lavrentiev phenomenon are sufficient to establish \eqref{eq253} with \(g=f\). Actually, we have an even stronger conclusion: there is no \emph{Lavrentiev gap}, in the sense that 
for every \(u\in W^{1,1}(\Omega)\), there exists a sequence \((u_n)_{n\geq 1}\subset W^{1,\infty}(\Omega)\) which strongly converges  to \(u\) in \(W^{1,1}(\Omega)\) and such that \(\lim_{n\to +\infty}E_\Omega (u_n)=E_{\Omega}(u)\). Moreover, we can ensure that each \(u_n\) agrees with \(u\) on \(\partial \Omega\) (provided of course that the trace of \(u\) on \(\partial \Omega\) is Lipschitz continuous).  It is   expected, as the absence of the Lavrentiev phenomenon is typically inferred by excluding the Lavrentiev gap, see \cite{AmGoSe,C-b,bgs-arma,Bor-Chl,BCM,hahab,Lukas23}. \newline

\noindent{\bf Main results.} In order to present our framework, let us settle some more notation. We consider the space $W^{1,1}_\vp (\Omega)$ of those functions $u\in W^{1,1}(\Omega)$ which coincide with $\vp$ on $\partial\Omega$ and the
subset $W^{1,\infty}_\vp(\Omega)$ of those Lipschitz continuous functions \(u\) on $\Omega$ which agree with  $\vp$ on $\partial\Omega$. 
Let us define 
\begin{equation}\label{eq:spaceE}
\cE(\Omega) =\left\{u\in W^{1,1}(\Omega)\colon\ E_\Omega(u)<\infty\right\}\qquad\text{and}\qquad \cEp(\Omega) =\left\{u\in \Wphi(\Omega)\colon\ E_\Omega(u)<\infty\right\}\,.
\end{equation}
Our main results state the absence of the Lavrentiev gap between Lipschitz functions and the energy space of the functional $E_\Omega$, defined in~\eqref{def-E-of-u}, under a Dirichlet boundary condition formulated in terms of a Lipschitz function $\vp$. 
As already mentioned, this requires a condition on \(f\) which controls the growth in the last two variables \(u\) and \(\nabla u\) when the spatial variable \(x\) varies. Since this condition, in its more general unifying form (see \eqref{eqHZconv} below), looks quite uninterpretable at first glance, we prefer to begin with some specific situations where our main result can be formulated in more friendly-looking terms. Let us emphasize, however, that the following statements are already interesting in their own right: they extend and simplify many related theorems on the non-occurrence of the Lavrentiev gap.
 
In the first simplified situation that we address, we assume that $f$ is isotropic, i.e., $f(x,t,\xi)=f(x,t,|\xi|)$, where \(|\xi|\) is the Euclidean norm of \(\xi\). In that case, the condition connecting the $(t,\xi)$-growth of $f$ with the spatial variable $x$ requires that:\\
\noindent For every \(L=(L_1,L_2)\in (0,\infty)^2\), there exists a constant \(C_{L}>0\)  such that  for a.e. $x,y\in\Omega$ and every \((t,\xi)\in [-L_1,L_1]\times \rn\), for every \(\ve>0\), it holds
\begin{equation}\label{eqHZ-iso-easy}\tag{$H^{\rm iso}_0$}
 |x - y| < \ve\ \text{ and }\ |\xi| \leq L_2\ve^{-\min\big(1, \frac{N}{p}\big)}  \Longrightarrow 
f(x,t,|\xi|) \leq C_{L}  \Big(
f(y,t,|\xi|)+1\Big)\,.
\end{equation} 
Hence, if \(x\) and \(y\) are not too far apart and \(\xi\) is not too large, then \(f(x,t,\xi)\) and \(f(y,t, \xi)\) are comparable. Observe that the distance between \(x\) and \(y\) on the one hand, and the Euclidean norm of \(\xi\) on the other hand, are controlled with the same parameter \(\ve\). In particular, we do not require that the right-hand side holds for any \(\xi\), but just for those which are suitably small with respect to \(|x-y|\). To the best of our knowledge, this type of condition was first introduced by Zhikov in his seminal paper \cite{zh}  on the double phase problem, compare~\eqref{eq228}.
It should be emphasized that the second variable \(t\) and the third variable \(\xi\) are not on equal terms in \eqref{eqHZ-iso-easy}. More specifically, once \(L\) is fixed, the variable \(t\) is confined to a bounded interval, and the large values of \(t\) are thus ignored, whatever the choice of \(\ve\) may be. For isotropic Lagrangians, \eqref{eqHZ-iso-easy} is enough to discard any Lavrentiev gap. This is the content of the following statement, which can be deduced from our main result Theorem~\ref{th-main-conv}:

\begin{theorem}\label{theo:iso}
    Let \(\Omega\) be a bounded Lipschitz open set in \(\rn\), \(\vp:\rn\to \R\) be Lipschitz continuous, $p\geq 1$, and \(f:\Omega \times \R\times \rp\to \rp\) be a Carath\'eodory function, which is convex with respect to the last variable and satisfies $f(x,t,0)=0$ for a.e. $x\in\Omega$ and all $t\in\R$. 
Suppose that \eqref{eqHZ-iso-easy} holds.
Consider the functional $E_\Omega$  defined  in~\eqref{def-E-of-u}. Then, for every \(u\in \cEp(\Omega)\cap W^{1, p}(\Omega)\), there exists a sequence \((u_{n})_{n\in\N}\subset W^{1,\infty}_\vp(\Omega)\) such that  $u_n \to u$ in $W^{1, p}(\Omega)$ as $n \to \infty$ and
\begin{equation*}%
\lim_{n\to \infty}E_\Omega(u_n)=E_\Omega(u) \,.
\end{equation*}
\end{theorem}
\noindent We stress that the above result directly generalises~\cite[Theorem~2.3]{bgs-arma}, since it includes a non-trivial $u$-dependence of the integrand, which is possibly unbounded and non-convex. Moreover, there is no need here for the integrand to satisfy the $\Delta_2$-condition nor to have super-linear growth in the last variable (nor to grow faster than a fixed power function). Consequently, we retrieve classical results for variable exponents and double phase functionals~\cite{eslemi, zh}, as well as new ones for the latter~\cite{bcdfm}.  We present novel examples of functionals to which Theorem~\ref{theo:iso} applies in Section~\ref{sec:results}. We also refer to Theorem~\ref{theo:iso-ortho} below for the extended isotropic and orthotropic version of this result.

The formulation of \eqref{eqHZ-iso-easy} can be simplified by introducing the following notation:
 For a measurable function $f:\Omega\times \R\times \rn\to \R$ and a given ball $B\subset \rn$ intersecting \(\Omega\), we denote
\begin{equation}\label{f-min-f-pl}
f_B^-(t,\xi) \coloneqq \essinf_{x \in B \cap \Omega} f(x, t,\xi)\,.
\end{equation}
Then, \eqref{eqHZ-iso-easy} reads
$$
|\xi| \leq L_2\ve^{-\min\big(1, \frac{N}{p}\big)}  \Longrightarrow 
f(x,t,|\xi|) \leq C_{L}  \Big(f_{B(x,\ve)}^-(t,\xi)
+1\Big)\,.
$$
The isotropic structure of the functionals considered in Theorem~\ref{theo:iso} can be relaxed to the fully anisotropic one, i.e.,  when $f$ depends on the last variable not necessarily via its length. We present here a simple, fully anisotropic consequence of our main result, Theorem~\ref{th-main-conv}.  The new condition on \(f\) balancing the \((t,\xi)\)-growth in terms of the spatial variations now reads:\\
\noindent For every \(L=(L_1,L_2)\in (0,\infty)^2\), there exists a constant \(C_L>0\)  such that  for a.e. $x\in\Omega$ and every \((t,\xi)\in [-L_1,L_1]\times \rn\), for every \(\ve>0\), it holds
\begin{equation}\label{eqHZD2}\tag{$H_{\Delta_2}$}
f_{B(x,\ve)}^{-}(t,\xi)
\leq {L_2}{\ve^{-N}} \Longrightarrow
f(x,t,\xi) \leq C_{L }\Big( f_{B(x,\ve)}^{-}(t,\xi)+1\Big)\,.
\end{equation}  
Observe that \eqref{eqHZD2} differs from \eqref{eqHZ-iso-easy} only through the way that the length of \(t\) and \(\xi\) are measured. While in \eqref{eqHZ-iso-easy}, the quantities \(f(x,t,\xi)\) and \(f_{B(x,\ve)}^{-}(t,\xi)\) are comparable when the norm of \(\xi\) alone is not too large, in \eqref{eqHZD2} instead, this holds under a suitable control on \(f_{B(x,\ve)}^{-}(t,\xi)\). As illustrated by many examples in Section~\ref{sec:results}, the latter condition often yields sharper results, even for isotropic functionals, provided that the \(\Delta_2\) condition  \eqref{eqDELTA2-intro} is satisfied.

\begin{theorem}
    \label{theo:doubling}
Suppose \(\Omega\) is a bounded Lipschitz open set in \(\rn\) and \(\vp:\rn\to \R\) is Lipschitz  continuous. Let $f:\Omega\times\R\times\rn\to\rp$ be a Carath\'eodory function which is convex with respect to the last variable, which satisfies $f(x,t,0)=0$ and $f(x,t,\cdot)\in\Delta_2$ for a.e. $x\in\Omega$ and {all} $t\in\R$. Suppose that  \eqref{eqHZD2} holds. 
Consider the functional $E_\Omega$  defined  in~\eqref{def-E-of-u}.  Then, for every \(u\in \cEp(\Omega)\), there exists a sequence \((u_{n})_{n\in\N}\subset W^{1,\infty}_\vp(\Omega)\) such that  $u_n \to u$ in $W^{1, 1}(\Omega)$ as $n \to \infty$ and
\begin{equation*}%
\lim_{n\to \infty}E_\Omega(u_n)=E_\Omega(u) \,.
\end{equation*}
\end{theorem}We shall stress that already this special case is new, as it allows for an essentially non-convex structure of $f$. Indeed, the dependence of $f$ on the second variable is just continuous, but not bounded nor convex as e.g. in~\cite{Bor-Chl,BCM,BMT}. This makes the proof a non-trivial task, which cannot be solved by merely varying previous reasonings. Moreover, let us emphasize the following facts.
\begin{enumerate}
    \item The above statements are interesting even for \(N=1\) as the continuity of $f$ with respect to $x$ is not required. 
    \item The assumption $f(x,t,0)=0$ is given here for the clarity of the statement and might be relaxed, see the growth condition \eqref{eq:f-origin} below, which controls even possibly unbounded behaviour of $f$ in the origin.   
    \item A bunch of examples to the theorem mentioned above, covering classical functionals as well as new ones, are provided in Section~\ref{ssec:examples}. {For a quick summary see Table~\ref{tab:examples} and Table~\ref{tab:examples2}.}
\end{enumerate}
In our main result, Theorem~\ref{th-main-conv}, we introduce a fully anisotropic condition~\eqref{eqHZconv} related to~\eqref{eqHZ-iso-easy} and~\eqref{eqHZD2}. Still, unlike them, it involves a balance between  \(f\) and the greatest convex minorant of $f_B^-$ in the \(\xi\)-variable. More specifically, given an open ball \(B\) intersecting \(\Omega\) and \(t\in \R\), we denote by $\left(f^{-}_{B}(t,\cdot)\right)^{**}$ the supremum of all those convex functions \(g:\rn\to \rn\) which lie below \(f^{-}_{B}(t,\cdot)\) on \(\rn\) (since \(f\) and thus also \(f^{-}_{B}(t,\cdot)\) are nonnegative, this supremum is well-defined). We are now in a position to define the main assumption involved in the forthcoming Theorem~\ref{th-main-conv}: 
\\
\noindent Let $p \geq 1$. We assume that for every \(\mathcal{k}=(\mathcal{k}_1,\mathcal{k}_2)\in (0,\infty)^2\), there exists a constant 
$\wt\cC_\mathcal{k}>0$ such that for  a.e. $x\in{\Omega}$ and all $(t,\xi)\in (-\mathcal{k}_1,\mathcal{k}_1)\times\rn$, for every $\ve>0$,  it holds 
\begin{equation}\label{eqHZconv}\tag{$H^{\textrm{conv}}$}
{\left(f^{-}_{B(x,\ve)}(t,\cdot)\right)^{**}\left(\xi\right)+|\xi|^{\max(p,N)}
\leq {\mathcal{k}_2}{\ve^{-N}}}\qquad \Longrightarrow\qquad  f\left(x,t, \xi\right)\leq
\wt\cC_\mathcal{k} \left[\left(f^{-}_{B(x,\ve)}(t,\cdot)\right)^{**}\left(\xi\right)+1\right]\,.
\end{equation} 

The following remarks are in order.
\begin{enumerate}
\item The condition \eqref{eqHZconv} is reminiscent of many articles in the field, compare~\cite{gszg,ha-aniso,C-b}, and more  recently~\cite{BCM,BMT,Lukas23}. Although not necessarily intuitive at first, it is quite handy and allows us to handle many situations. Observe in particular that the length of \(\xi\) is now quantified both in terms of the Euclidean norm of \(\xi\)  and the function \(\left(f^{-}_{B(x,\ve)}(t,\cdot)\right)^{**}\left(\xi\right)\). As a consequence, we will prove that \eqref{eqHZconv} implies \eqref{eqHZ-iso-easy} (for isotropic Lagrangians) and \eqref{eqHZD2} (for \(\Delta_2\) Lagrangians).
\item   For every \(\ve\geq \diam \Omega\),  for every \((x,t,\xi )\in \Omega\times\R\times \R^N\), one has \(f_{B(x,\ve)}^{-}(t,\xi) = f_{B(x,\diam\Omega)}^{-}(t,\xi)\) and thus
\[
\left(f_{B(x,\ve)}^{-}(t,\cdot)\right)^{**}(\xi) = \left(f_{B(x,\diam\Omega)}^{-}(t,\cdot)\right)^{**}(\xi)\,.
\]
Consequently, in the  assumption \eqref{eqHZconv}, the parameter \(\varepsilon>0\) can be restricted to the interval \((0, \diam \Omega]\).


\item The parameter $p$ plays a role when $p > N$, in which case~\eqref{eqHZconv} allows excluding the Lavrentiev gap on the entire energy space of $E_\Omega$ if the latter is included in $W^{1, p}{(\Omega)}$, see Theorem~\ref{th-main-conv}.

\item A function $f$ of the form $f(x,t,\xi)=g(x,t, \xi)+h(x,t,\xi)$ satisfies~\eqref{eqHZconv} if both \(g\) and $h$ do. This follows from the fact that for every ball \(B\), one has \(f_{B}^{-}\geq g_{B}^- + h_{B}^-\) and thus for every \(t\in \R\),  the function $(g_{B}^-(t,\cdot))^{**}+(h_{B}^-(t,\cdot))^{**}$ is a convex minorant of $f_{B}^-(t,\cdot)$, so we can infer that $(g_{B}^-(t,\cdot))^{**}+(h_{B}^-(t,\cdot))^{**}\leq (f_{B}^-(t,\cdot))^{**}$ and the conclusion easily follows. In particular, this observation can be useful for  multi-phase functions of the form $f(x,t,\xi)=\sum_{i=1}^k\left(\tfrac{1}{k}\psi_0(t, \xi) +  a_i(x, t)\psi_i(t, \xi)\right)$. In this case,  when \(k>2\),  the function $f_{B}^{-}(t,\cdot)$ is not convex in general and the explicit expression of its greatest convex minorant is out of reach. It is much easier to check whether each term of the sum \(\vt_i(x,t,\xi)\coloneqq \frac{1}{k}\psi_0(t, \xi) +  a_i(x, t)\psi_i(t, \xi)\) satisfies~\eqref{eqHZconv}, since each \((\vt_{i})_{B}^{-}(t,\cdot)\) is convex.

\item Section~\ref{ssec:examples} presents many illustrative special cases of $f$ satisfying \eqref{eqHZconv}. A representative choice of examples is summarized in Tables~\ref{tab:examples} and~\ref{tab:examples2}.
\end{enumerate}

Our main result asserts that under \eqref{eqHZconv}, no Lavrentiev gap may occur:
\begin{theorem}\label{th-main-conv}
Let \(\Omega\) be a bounded Lipschitz open set in \(\rn\), $N\geq 1$, \(\vp:\rn\to \R\) be Lipschitz continuous,  \(p\geq 1\), and \(f:\Omega \times \R\times \rn\to \rp\) be a Carath\'eodory function which is convex with respect to the last variable {and} satisfies~\eqref{eqHZconv}. When \(N>1\), we further assume that the behaviour of $f$ in the origin is constrained, namely that there exists \(\vt\in [1,\infty]\), \(a\in L^{\vt}(\Omega)\) and \(t_0>0\) which satisfy
    \begin{equation}\label{eq:f-origin}
     0\leq f(x,t,0) \leq a(x)|t|^{p^*/\vt'} \qquad \text{for }\ |t|\geq t_0 \textrm{ and for a.e. } x\in \Omega\,.   
   \end{equation} 
Consider the functional $E_\Omega$  defined  in~\eqref{def-E-of-u}.  Then, for every \(u\in \cEp(\Omega)\cap W^{1,p}(\Omega)\), there exists a sequence \((u_{n})_{n\in\N}\subset W^{1,\infty}_\vp(\Omega)\)  such that  $u_n \to u$ in $W^{1, p}(\Omega)$ as $n \to \infty$ and
\begin{equation*}
\lim_{n\to \infty}E_\Omega(u_n)=E_\Omega(u) \,.
\end{equation*}
\end{theorem}
Here, \(p*\) is defined as \(Np/(N-p)\) if \(p<N\) while $p^*$ is any number larger or equal to \(N\) otherwise.

\begin{remark}\rm 
The above theorem includes a condition on the behaviour of $f$ when $\xi=0$, namely~\eqref{eq:f-origin}. It is only used to approximate any Sobolev map \(u\in \cEp(\Omega)\cap W^{1,p}(\Omega)\) by a sequence \((u_{n})_{n\in\N}\subset L^{\infty}(\Omega)\cap W^{1,p}_\vp(\Omega)\) in norm and energy. 
When we assume a priori that \(u\) is bounded, we do not need this assumption to conclude.
In particular, when \(N=1\), \(u\) is automatically bounded and condition~\eqref{eq:f-origin} can be dropped.
\end{remark}

Roughly speaking, the conclusion of Theorem~\ref{th-main-conv} is stable with respect to certain variations of the assumptions: 
\begin{remark}  
If $f$ is not convex with respect to the last variable, but comparable to a convex function, one can still infer the absence of the Lavrentiev gap. More precisely -- if for a Carath\'eodory function \(f:\Omega \times \R\times \rn\to \rp\),  there exist a function \(g:\Omega \times \R\times \rn\to \rp\)  satisfying the conditions of Theorem~\ref{th-main-conv}, a constant $C > 0$, and a non-negative function $\vartheta \in L^1(\Omega;\R)$ such that 
    \begin{equation}\label{eq:ass}
        {\tfrac{1}{C}g(x, t, \xi) - \vartheta(x) \leq f(x, t, \xi) \leq Cg(x, t, \xi) + \vartheta(x)\,,}
    \end{equation} then by Proposition~\ref{prop:comp}, the conclusion of  Theorem~\ref{th-main-conv} holds for $f$.
\end{remark}

It is fair to say that the direct verification of the conditions in Theorem~\ref{th-main-conv} may be difficult in practice, due to the presence of the greatest convex minorant. One may wonder whether the conclusion of Theorem~\ref{th-main-conv} remains true when  
\(\left(f^{-}_{B(x,\ve)}(t,\cdot)\right)^{**}\) in \eqref{eqHZconv} is replaced by \(f^{-}_{B(x,\ve)}(t,\cdot)\). Let us restrict for an instant to the isotropic situation when \(f\) depends on \(\xi\) only through its Euclidean norm and consider the following condition:\\
\noindent For every \(\mathcal{k}=(\mathcal{k}_1,\mathcal{k}_2)\in (0,\infty)^2\), there exists a constant 
$\cC_\mathcal{k}>0$ such that for  a.e. $x\in\Omega$ and all $(t,\xi)\in (-\mathcal{k}_1,\mathcal{k}_1)\times\rn$, for every $\ve>0$,  it holds 
\begin{equation}\label{eqHZiso}\tag{$H^{\textrm{iso}}$}
{ f^{-}_{B(x,\ve)}(t,|\xi|) +|\xi|^{\max(p,N)}
\leq {\mathcal{k}_2}{\ve^{-N}}}\qquad \Longrightarrow\qquad  f\left(x,t, |\xi|\right)\leq
\cC_\mathcal{k} \left[ f^{-}_{B(x,\ve)}(t,|\xi|)+1\right]\,.
\end{equation}   
This condition is significantly more intuitive and easier to verify than \eqref{eqHZconv}, while it substantially generalizes ~\eqref{eqHZ-iso-easy} in Theorem~\ref{theo:iso}. 
Since \(f^{-}_{B(x,\ve)}(t,|\xi|)\geq \left(f^{-}_{B(x,\ve)}(t,\cdot)\right)^{**}(|\xi|)\), the assumption \eqref{eqHZconv} implies \eqref{eqHZiso}.
Generally speaking, however, the converse implication is not true, compare Example~\ref{ex:counter} below. But under an isotropic or orthotropic regime, those two conditions are equivalent. This is the content of our next result.

\begin{theorem}\label{theo:final}Let $f : \Omega \times \rn \times \R \to [0, \infty)$ be a Carath\'eodory function, which is convex with respect to the last variable. Then, the following assertions are true.
\begin{enumerate}
    \item If $f$ is isotropic and satisfies \eqref{eqHZiso}, then it satisfies~\eqref{eqHZconv}.
    \item If $f$ admits an  orthotropic decomposition, i.e. $f:\Omega\times\R\times\rn\to\rp$ is such that 
\begin{equation}
    \label{f-is-ortho}
f(x,t,\xi)=\sum_{i=1}^N f_i(x,t,|\xi_{i}|)\,,\quad\text{where \(\xi=(\xi_1, \dots, \xi_N)\)}\,,
\end{equation} and for every $i=1, \dots, N\,,$ the function $f_i:\Omega\times\R\times\rp\to\rp$ is a Carath\'eodory function which is convex with respect to the last variable, and satisfies \eqref{eqHZiso},  then $f$ satisfies~\eqref{eqHZconv}.
\end{enumerate}
\end{theorem}

The above statement 
can be seen as a  counterpart of \cite[Theorem~1.2]{ha-aniso} for functions having isotropic or orthotropic structure. In both cases, it is shown that a balance between $f$ and $f^-_B$ on sub-level sets of $f^-_B$ given by~\eqref{eqHZiso} implies a similar balance between  $f$ and $(f^-_B)^{**}$ as in \eqref{eqHZconv}.
The main difference between Theorem~\ref{theo:final} and \cite[Theorem~1.2]{ha-aniso} is the position of the constant (we have it outside $f_{B}^-$ and $(f_{B}^-)^{**}$, respectively). Our proof is based on essentially different geometrical observations.  We point out that the result above also allows reformulating~\cite[Hypothesis (H)]{BMT} or~\cite[Assumption~1~(a4)]{Lukas23} in the isotropic or orthotropic setting by eliminating the $**$ operator.

As a simple consequence of Theorem~\ref{th-main-conv} and Theorem~\ref{theo:final}, we derive the non-occurence of Lavrentiev gaps for isotropic and orthotropic functionals:

\begin{theorem}\label{theo:iso-ortho}
Let \(\Omega\) be a bounded Lipschitz open set in \(\rn\), \(\vp:\rn\to \R\) be Lipschitz continuous,  \(p\geq 1\), \(f:\Omega \times \R\times \rn\to \rp\)  be a Carath\'eodory function which is convex with respect to the last variable and which satisfies~\eqref{eq:f-origin}. Consider the functional $E_\Omega$ defined  in~\eqref{def-E-of-u}.\\ We assume either that $f$ is isotropic and satisfies \eqref{eqHZiso},  or that  $f$ admits an orthotropic decomposition in the sense of \eqref{f-is-ortho}, with each \(f_i\) satisfying \eqref{eqHZiso}.
 Then, for every \(u\in \cEp(\Omega)\cap W^{1,p}(\Omega)\), there exists a sequence \((u_{n})_{n\in\N}\subset W^{1,\infty}_\vp(\Omega)\)  such that  $u_n \to u$ in $W^{1, p}(\Omega)$ as $n \to \infty$ and
\begin{equation*}
\lim_{n\to \infty}E_\Omega(u_n)=E_\Omega(u) \,.
\end{equation*} 
\end{theorem}

\noindent{\textbf{Methods.}} The Lagrangian \(f\) that we consider in~\eqref{def-E-of-u} depends on three variables: the spatial variable \(x\in \Omega\), the second variable \(t\) corresponding to the values of \(u(x)\), and the gradient variable \(\xi\) representing the values of \(\nabla u(x)\). It is known that the Lavrentiev gap does not arise when \(f\) is autonomous; that is, when \(f\) depends only on \(t\) and \(\xi\), see \cite{Bousquet-Pisa}. 
In contrast, the \(x\)-dependence usually produces Lavrentiev gaps, except when \(f\) depends just on $x$ and $\xi$ and satisfies a balance (or anti-jump) condition involving the spatial oscillations and the growth in the gradient variable. In that case,  one approximates any given map \(u\in W^{1,1}(\Omega)\) by a family \((u_\ve)_{\ve >0}\) obtained by convolution of \(u\) with a smooth kernel \(\vr_{\ve}\) (here we ignore the boundary condition for the moment). In order to prove the approximation in energy, namely that \(\lim_{\varepsilon\to 0}E_{\Omega}(u_\ve)=E_\Omega(u)\), one first observes that {\it the liminf inequality}: \(\liminf_{\ve\to 0}E_{\Omega}(u*\vr_\ve)\geq E_{\Omega}(u)\) easily follows from the Fatou lemma. 
The heart of the matter is {\it the limsup inequality}: \(\limsup_{\ve\to 0}E_{\Omega}(u*\vr_\ve)\leq E_{\Omega}(u)\). At this stage, one needs to exploit an anti-jump condition (\eqref{eqHZconv} for instance) which implies that: \(f(x,\nabla u_\ve)\lesssim \left(f_{B(x,\ve)}^{-}\right)^{**}(\nabla u_\ve)\), on the range where $\nabla u_\ve$ is controlled, up to innocuous additive terms. Then, one applies the Jensen inequality to \(\left(f_{B(x,\ve)}^{-}\right)^{**}\) to get
\[
f(x,\nabla u_\ve)\lesssim
\left(f_{B(x,\ve)}^{-}\right)^{**}(\nabla u_\ve) \leq \int_{\rn} \left(f_{B(x,\ve)}^{-}\right)^{**}(\nabla u(x-y))\vr_\ve(y)\dy\leq  \int_{\rn} f(x-y, \nabla u(x-y))\vr_\ve(y)\dy\,,
\]
where the last inequality follows from the definition of \(\left(f_{B(x,\ve)}^{-}\right)^{**}\). We thus obtain that the functions \(f(x,\nabla u_\ve)\) are dominated (still up to multiplicative constants and harmless additive terms) by 
\(f(x, \nabla u)*\vr_\ve\). Since \(\lim_{\ve\to 0} f(x, \nabla u)*\vr_\ve =f(x,\nabla u)\) in \(L^1\), the generalized dominated convergence theorem allows to conclude that 
\(\lim_{\ve\to 0}E_\Omega(u_\ve)=E_\Omega(u)\). 
The above calculation illustrates that the anti-jump condition is perfectly suited for approximation by convolution with a smooth kernel. 
As a matter of fact, one could even have the impression that it has been formulated precisely to make this interplay with the Jensen inequality work (for Lagrangians depending just on \(x\) and \(\xi\)). But even so, the anti-jump condition is simultaneously very efficient to discard the Lavrentiev gaps under (almost) sharp conditions on the parameters involved in the definition of \(f\) (see~\cite{eslemi}, that exploits this fact for the exponents \(p\) and \(q\) and the H\"older exponent of the coefficient \(a\) in the double phase functional \eqref{eq228},  and \cite{badi,bcdfm,Bor} for generalized versions).

The above approximation by mere convolution does not take into account the boundary condition, namely the fact that the approximating sequence should agree with \(u\) on \(\partial \Omega\). We will revisit this issue subsequently. We note here that for Lagrangians depending only on \(x\) and \(\nabla u\), this may lead to very intricate difficulties and subtle arguments to overcome them; see, e.g.,~\cite{C-b,Lukas}. 

We emphasize that this classical approximation of a Sobolev map \(u\) by convolution with a smooth kernel is doomed to fail for a Lagrangian having an essential and non-convex dependence with respect to \(t\). The obstacle is basically that the Jensen inequality, which plays so important a role in {\it the limsup inequality}, requires a convexity assumption. However, in our situation,  the Lagrangian \(f\) is assumed to be convex just with respect to \(\xi\), and not jointly convex with respect to \((t,\xi)\). Hence, even when \(f\) does not depend on \(x\), it is not possible to justify that \(f(u_\ve, \nabla u_\ve)\lesssim f(u, \nabla u)*\vr_\ve\). 
To overcome this difficulty,  as in \cite{Bousquet-Pisa}, we use the classical formulation of non-parametric variational problems in terms of parametric ones. The most prominent example of such a formulation is the area functional \(u\in W^{1,1}(\Omega)\mapsto \int_{\Omega}\sqrt{1+|\nabla u|^2}\dx\) which can be expressed as the perimeter of the subgraphs of the competing functions \(u\); that is,
\begin{equation}\label{eq506}
\int_{\Omega}\sqrt{1+|\nabla u|^2}\dx = \int_{\Omega\times \R} \d|D1_u|(t,x)\,,
\end{equation}
where \(|D1_u|\) is the total variation of the (finite) measure \(D1_u\) and \(1_u\) is the indicator function of the hypograph of \(u\), namely \(1_u(x,t)=1\) when \(t\leq u(x)\) and \(1_u(x,t)=0\) otherwise, see~\eqref{1u}. The right-hand side of \eqref{eq506} is the perimeter of the subgraph of \(u\) in the sense of Cacciopoli sets.
More generally, starting from an integrand \(f:\Omega\times \R\times \R^N\to \rp\), we can associate to the corresponding energy \(E_{\Omega}\) its convex counterpart \(\widehat{E}_{\Omega\times \R}\) which is another energy that acts on the subgraphs of the functions \(u\). 
To be more specific, while the original energy \(E_{\Omega}\) is defined on \(W^{1,1}(\Omega)\),  the new energy \(\widehat{E}_{\Omega\times \R}\) is defined on the set of those maps \(v\in L^{\infty}(\Omega\times \R)\) such that their distributional derivatives \(Dv\)  are finite \(\R^{N+1}\)-valued measures. The two energies are related by the formula
\begin{equation*}
\forall u \in W^{1,1}(\Omega) \qquad
E_{\Omega}(u)=\widehat{E}_{\Omega\times \R}(1_u)\,.
\end{equation*}
This generalizes to any function \(f\) the classical formulation \eqref{eq506} of the non-parametric minimum area problem via the minimization of the perimeter. The new energy \(\widehat{E}_{\Omega\times \R}\) involves a Lagrangian \(\widehat{f}:\Omega\times \R\times \R^{N+1}\to [0,\infty]\):
\[
\widehat{E}_{\Omega\times \R} (v)=\int_{\Omega\times \R} \dhf(x, t, Dv)\,.
\]
This Lagrangian \(\widehat{f}\) is homogeneous of degree \(1\) and convex in the last variable, which entitles one to define \(\dhf(x, t, Dv)\) as a measure whenever \(Dv\)  is a measure,  see~\eqref{eq776}. In the case of \eqref{eq506}, the Lagrangian \(\widehat{f}\) is simply the Euclidean norm in \(\R^{N+1}\).

In this new formulation, the \(t\) variable from the original functional becomes an additional spatial variable in the new functional.  The fact that the parametric formulation is a way to convexify the original one has already been exploited in the setting of \(\Gamma\)-convergence for \(BV\)-functions, see \cite{Dal-Maso},  and for the formulation of necessary conditions, see \cite{BouFra}. The key consequence for us is that the non-convex behaviour of \(f\) with respect to \(t\) becomes harmless and the convexity with respect to \(\xi\) is enough to ensure the convexity of \(\widehat{E}_{\Omega\times \R}\).

Hence, the approximate problem is transferred to this new functional \(\widehat{E}_{\Omega\times \R}\), except that the function \(v\) to be approximated is not a Sobolev map anymore. Moreover, for some reasons that will be clarified subsequently, the desired regularity of the approximating maps \(v_n\) is not the Lipschitz continuity but a \emph{cone}  condition of the following form: there exists \(C_n>0\)  such that for every \(x,y \in \Omega\) and for every \(t\in \R\), 
\begin{equation}\label{eq-Lip-vn}
v_n(x,t+C_n|x-y|)\leq v_n(y, t)\,.
\end{equation}
If \(v_n\) were the indicator function of the subgraph of a certain function \(u_n:\Omega\to \R\), this condition would mean that the graph of \(u_n\) lies below a family of cones with bounded apertures, which in turn is equivalent to the Lipschitz continuity of  \(u_n\).  {To construct such a sequence \(v_n\),   we introduce} two scales of parameters \(\ve\) and \(\delta\) and define
\[
v_{\ve, \delta}(x,t)\coloneqq v*_x\vr_\ve(x,t)+\delta\alpha(t)\,,
\]
where \(*_x\) refers to  \emph{partial} convolution just with respect to \(x\), while \(\alpha\) is a decreasing function depending just on \(t\), see the very  beginning of Section~\ref{sec:approx}. This second term is crucial to force \(t\mapsto v_n(x,t)\) to decrease uniformly with respect to \(x\), which enables us to obtain a bound \(C_n\) as in \eqref{eq-Lip-vn}.

For every sequence \((\ve_n)_n\) and \((\delta_n)_n\) decreasing to \(0\), the resulting maps \(v_n\coloneqq v_{\ve_n,\delta_n}\) converge a.e. to \(v\). Moreover, {\it the liminf inequality}, namely \(\liminf_{n\to \infty}\widehat{E}_{\Omega\times \R}(v_n)\geq \widehat{E}_{\Omega\times \R}(v)\), is an easy consequence of the Reshetnyak semicontinuity theorem for functionals of measures, see Proposition~\ref{prop-liminf}. 
Up to this point, all our arguments closely follow those developed in \cite{Bousquet-Pisa}.
The key challenge in reaching the conclusion lies in establishing {\it the limsup inequality}  of Proposition~\ref{prop-limsup} reading 
\begin{equation}\label{eq353}
\limsup_{n\to \infty}\widehat{E}_{\Omega\times \R}(v_n)\leq \widehat{E}_{\Omega\times \R}(v)\,.
\end{equation}
When \(f\) depends only on \((u,\nabla u)\), and assuming further that \(f(u,0)=0\) (which is not really restrictive), {\it the limsup inequality} can be easily derived from a  Jensen-type inequality for measures:
\begin{equation}\label{eq548}
\widehat{f}(t,Dv_{\ve,\delta}) \leq \widehat{f}(t,Dv)*_x\vr_{\ve}\,,
\end{equation}
see Remark~\ref{remark-Jensen-measure-autonomous} below. 
One derives therefrom the integral inequality
\begin{equation}\label{eq544}
\int_{\Omega\times \R}\d \widehat{f}(t,Dv_{\ve,\delta})
\leq \int_{\Omega\times \R}\d \widehat{f}(t,Dv)*_x\vr_{\ve}\,,
\end{equation}
from which {\it the limsup inequality} \eqref{eq353} easily follows.

In our situation where \(f\) depends in an essential way on the three variables, the anti-jump condition on \(f\) should play a role in the proof of {\it the limsup inequality} \eqref{eq353}, since otherwise the Lavrentiev phenomenon would never occur. As already explained, the anti-jump condition is perfectly adapted to the Jensen inequality when \(f\) only depends on \(x\) and \(\xi\) and when \(u\) is approximated by convolution with a smooth kernel. In contrast, it is not obvious how such an anti-jump condition should be exploited for the new Lagrangian \(\widehat{f}\) and the new approximating sequence \(v_{\ve_n,\delta_n}\), which is now defined from \(u\) in a more intricate way.  In particular, the inequality \eqref{eq548} is markedly false when \(f\) depends on the spatial variable \(x\) (since again no Lavrentiev gap would occur otherwise). 

In order to take into account the spatial dependence of the Lagrangian \(f\), we need to introduce a new strategy to prove {\it the limsup inequality}. 
The first task to establish \eqref{eq353} under the condition \eqref{eqHZconv} is an effective transfer of the anti-jump behaviour exhibited by the integrand  \(f\)  into pertinent information for the integrand \(\widehat{f}\) of the convexified energy $\wh E_{\Omega\times \R}$. In particular, we present in Section~\ref{ssec:Zh-for-hat-f}  estimates for the functions \(v_{\ve, \delta}\) which make it possible to exploit the anti-jump condition satisfied by \(\widehat{f}\).
This entitles us to get a suitable bound on the measure \(\widehat{f}(x,t,Dv_{\ve,\delta})\), see Lemma~\ref{lm-disintegration-nonautonomous}. Up to an additive term (depending on \(\delta\)) that is not difficult to handle, the corresponding estimate reads: for every Borel subset \(A'\subset \Omega\times \R\),
\begin{equation}\label{eq561}
\int_{A'}\d \wh{f}(x,t,Dv_{\ve, \delta}) \leq C\int_{A'}\d\left(\wh{f}(x,t,Dv)*_x \vr_{\ve}\right)\,.
\end{equation}
In contrast to \eqref{eq544}, there is a multiplicative constant \(C\) in the right-hand side (that is not necessarily equal to \(1\)). So even if we were able to pass to the limit \((\ve,\delta)\to (0,0)\), we would get (with $A'=\Omega\times \R$):
\begin{equation}\label{eq572}
\limsup_{(\ve, \delta)\to (0,0)}\int_{\Omega \times \R}\d \wh{f}(x,t,Dv_{\ve, \delta}) \leq C\int_{\Omega\times \R}\d \wh{f}(x,t,Dv)\,.
\end{equation}
One would then fail to obtain the expected {\it limsup inequality}. 

A possible way to overcome this obstacle is to rely on another naive attempt to prove \eqref{eq353}, namely applying the \emph{continuity Reshetnyak theorem} (Lemma~\ref{lm-Reschetnyak}) to the functional \(\widehat{E}_{\Omega\times \R}\),  so as to get the identity:
\[
\lim_{(\ve,\delta)\to (0,0)} \int_{\Omega\times \R} \dhf(x, t, Dv_{\ve,\delta}) = \int_{\Omega\times \R} \dhf(x, t, Dv)\,,
\]
which is exactly the desired conclusion.
However, this theorem requires among other assumptions that  \(\widehat{f}(x,t,\cdot)\) is continuous on the unit sphere of \(\R^{N+1}\). This certainly holds if \(f\) satisfies a linear growth condition, but fails to be true for a superlinear \(f\), since in that case \(\widehat{f}(x,t,(q^x,q^t))=+\infty\)  for every \( (q^x,q^t)\in \rn\times \rp\). Introducing the Lebesgue-Radon-Nikodym derivative \((q_{\ve, \delta}^x, q_{\ve, \delta}^t)\) of \(Dv_{\ve,\delta}\) with respect to its total variation \(|Dv_{\ve,\delta}|\), the measure \(\widehat{f}(x,t,Dv_{\ve, \delta})\) is simply defined as:
\[
\widehat{f}(x,t,Dv_{\ve, \delta}) = \widehat{f}(x,t,(q_{\ve, \delta}^x, q_{\ve, \delta}^t))|Dv_{\ve, \delta}|\,.
\]
It follows from the definition of \(v_{\ve,\delta}\) that \(q_{\ve,\delta}^{t}<0\) but \(\widehat{f}(x,t,(q_{\ve, \delta}^x, q_{\ve, \delta}^t))\) tends to \(+\infty\) when \(q_{\ve,\delta}^{t}\) approaches \(0\), which happens, roughly speaking, when \(\nabla u\) is very large. As a consequence,  the Reschetnyak continuity theorem seems to be of no use in our framework.

So far, we have presented two possible ways to establish {\it the limsup inequality}. They both fail for different reasons: the estimate \eqref{eq572} is an unsatisfactory limsup inequality, due to the nasty multiplicative constant \(C\), while the Reshetnyak continuity theorem only applies on subdomains of \(\Omega\) where the function \(q^{t}_{\ve,\delta}\vcentcolon=\partial_t v_{\ve, \delta}/|Dv_{\ve, \delta}|\) is bounded from above by a negative constant. It turns out that by suitably intertwining these two failing strategies, one can produce a successful one.

Roughly speaking, one introduces a new parameter \(k\geq 1\), and one splits \(\Omega\) into a good region and a bad region as follows:
\[
\int_{\Omega\times \R}\d\widehat{f}(x,t,Dv_{\ve,\delta}) = \int_{\{q^{t}_{\ve, \delta}\leq -\frac{1}{k}\}}\d\widehat{f}(x,t,Dv_{\ve,\delta})+ \int_{\{q^{t}_{\ve, \delta}> -\frac{1}{k}\}}\d\widehat{f}(x,t,Dv_{\ve,\delta})\,.
\]
For every fixed \(k\), one can apply the Reshetnyak theorem to the first term, since \(\widehat{f}\) is bounded in the good region where \(q^{t}_{\ve, \delta}\leq -1/k\). The second term is more delicate to handle because the parameter \(\ve\) arises in two different positions, namely through the two functions \(q^{t}_{\ve,\delta}\) and \(v_{\ve, \delta}\). In a nutshell, one first estimates the integrand by relying on the anti-jump condition obtained for \(\widehat{f}\) as in~\eqref{eq561}. Up to multiplicative constants and additive terms, this leads  to
\[
\int_{\{q^{t}_{\ve, \delta}> -\frac{1}{k}\}}\d\widehat{f}(x,t,Dv_{\ve,\delta})
\lesssim \int_{\{q^{t}_{\ve, \delta}> -\frac{1}{k}\}} \d\left(\widehat{f}(x,t,Dv)*_x \vr_\ve\right)\,.
\]
In order to pass to the limit when \(\ve\to 0\), we have to understand how \(q^{t}_{\ve, \delta}\) behaves when \(\ve\) varies. This requires a careful use of some classical ingredients from geometric measure theory, such as the coarea and the area formula, as well as the approximate tangent spaces to the level sets of \(u\). The fact that the bad region \(\{q^{t}_{\ve, \delta}> -\frac{1}{k}\}\) has a small measure is crucial to conclude. At a technical level, this entitles one to get rid of all the undesirable multiplicative constants when \(k\to +\infty\). 

Once {\it the limsup inequality}  \eqref{eq353} is proved, we reach a safe and familiar shore. We
can conclude with the equality \(\lim_{n\to \infty}\widehat{E}_{\Omega\times \R}(v_n)=\widehat{E}_{\Omega\times \R}(v)\) (Corollary~\ref{coro-lim-hat-E}). It remains to derive from this approximating sequence for \(v\) an approximating sequence for \(u\). This step of the proof is similar to the corresponding one in \cite{Bousquet-Pisa}: one can select some \(s\in (0,1)\) and define functions \(u_n\) on \(\Omega\) such that \(1_{u_n}\) coincides with the indicator function of the super level set  \(\{(x,t):\ v_n(x,t)>s\}\). The convergence of \(v_n\)  to \(1_u\) easily implies the convergence of \(u_n\) to \(u\).  That each \(u_n\) is Lipschitz continuous is a consequence of the condition~\eqref{eq-Lip-vn} satisfied by the maps \(v_n\).
All those arguments are provided in the proof of  Proposition~\ref{prop-inner-approximation}.

However, this is not the end of the story, since the maps \(u_n\) that we have just obtained do not necessarily agree with \(\varphi\) on \(\partial \Omega\). A non-trivial task is to modify \(u_n\) to get a new map $u_{n}^{\vp}$ that satisfies this boundary condition. A usual construction to achieve this goal is based on a partition of unity argument. Due to the absence of any growth condition of \(f\) with respect to \(\xi\), however, this approach does not work directly in our situation. Instead, we rely again on the convexified energy \(\widehat{E}_{\Omega\times \R}\) to localize the problem and reduce to the case when \(\Omega\) is the epigraph of a Lipschitz function. However, even in this setting, the strategy employed in~\cite{Bousquet-Pisa}, based on local translations, cannot be repeated due to the \(x\)-dependence.  We rely on a different approach and modify the construction of the maps \((u_n)_n\) described above by using the trick of decentered convolution, that is, the smooth kernel \(\vr_\ve\) is not taken radially symmetric but decentered with respect to the origin. This entitles us to exploit the regularity of \(\varphi\) and get first an approximating sequence which is uniformly close (but not necessarily equal) to the map \(\varphi\) on a neighborhood of \(\partial \Omega\). A final truncation argument with a cut-off function then yields an approximating sequence that agrees with \(\varphi\) on the boundary.  \newline

The above strategy works well under the anti-jump condition \eqref{eqHZconv}, which is the most general when compared to \eqref{eqHZD2} or \eqref{eqHZiso}. 
The equivalence between \eqref{eqHZconv} and \eqref{eqHZD2} when \(f\) satisfies a further \(\Delta_2\) condition  follows from a remarkable result by H\"ast\"o~\cite{ha-aniso}. Our Theorem~\ref{theo:final} provides a similar conclusion for the equivalence between \eqref{eqHZconv} and \eqref{eqHZiso}. This may seem counterintuitive, in view of the fact that the functions \(f_{B}^{-}(t,\xi)\) and \(\left(f_{B}^{-}(t,\cdot)\right)^{**}(\xi)\) are not comparable, even in the isotropic setting where \(f\), and thus also \(f_{B}^{-}(t,\cdot)\) and \(\left(f_{B}^{-}(t,\cdot)\right)^{**}(\cdot)\) only depend on the Euclidean norm of \(\xi\). 
Theorem~\ref{theo:final} is also reminiscent of  a similar result~\cite[Lemma 6.2]{bgs-arma} which is stated and proved for superlinear  functions \(f\) satisfying a suitable growth condition for large \(|\xi|\)'s and not depending on the variable \(t\). The major difference  is the way we measure the length of \(\xi\), which in Theorem~\ref{theo:final} 
involves both \(|\xi|\) and the function \(f_{B(x,\ve)}^-(t,|\xi|)\) (and not just \(|\xi|\)).
Our proof is also substantially different,  even if it relies as in \cite{bgs-arma} on several properties of the convex envelopes of one-variable functions, see Lemma~\ref{lem:at} and Lemma~\ref{lem:est-der}. 
\newline

\textbf{Organization of the paper. } In Section~\ref{sec:results}  we illustrate by numerous examples the above results on the absence of the Lavrentiev phenomenon.  Preliminary information is presented in Section~\ref{sec:prelim}.  Section~\ref{sec:twostars} translates the consequences of the main structural and anti-jump conditions to conditions satisfied on sub-level sets of the greatest convex minorant of the infimum of $f$ over small ball, namely $(f^-_B)^{**}$, see Theorem~\ref{theo:final}. Section~\ref{sec:prelim-to-main-proof} introduces the  construction of the initial approximation $(v_{\ve,\delta})$. The next two sections contain the main ingredients for the convergence of convexified energies of the initial approximate sequence, i.e., `the liminf estimate' for initial approximation is given in Section~\ref{sec:liminf} and `the limsup estimate' can be found in Section~\ref{sec:limsup-est}. In Section~\ref{sec:localization}, under {reduced} assumptions on \(f\), we establish the inner approximation (ignoring the boundary condition) in Proposition~\ref{prop-inner-approximation} and next we consider the  boundary approximation, first for special Lipschitz domains in Section~\ref{ssec:boundary-approx} and then for any Lipschitz sets in Section~\ref{ssec:proof-reduced}.  The final proof of Theorem~\ref{th-main-conv} is presented in Section~\ref{ssec:proof_conv} while the proofs of Theorems~\ref{theo:doubling},~\ref{theo:iso-ortho} are presented in Section~\ref{ssec:proofs_rest}.

\section{On classical and new functionals}\label{sec:results}

\subsection{One-dimensional examples of functionals exhibiting the Lavrentiev gap.}\label{ssec:counterexamples}
  We aim at showing why certain integrands known to be associated with some functionals with the Lavrentiev gap, do not fall into our regime.  
 One-dimensional problems are particularly illustrative 
 due to their simplicity and ease of visualization. Their analysis dates back to classical papers of Lavrentiev~\cite{Lav} and Mania~\cite{Mania}, providing the first examples of occurrence and conditions for the non-occurrence of the Lavrentiev gap. Concerning later contributions, we spotlight~\cite{Buttazzo-Mizel}, where the relaxation of functionals is considered.
 See also the survey~\cite{Buttazzo-Belloni} and the recent expository paper~\cite{Cerf-Mariconda}.

Starting with recalling probably the best known example of a functional with the Lavrentiev gap, we point out that the integrand does not satisfy any of our anti-jump conditions.
\begin{example}[Mania's example \cite{Mania}]\label{ex:mania}
    If $f:(0,1)\times\R\times\R\to \R$ is defined as $f(x,t,\xi)\coloneqq(t^3-x)^2\xi^6$, then we have, for boundary conditions $u(0) = 0$ and $u(1) =1$, that $\inf_{AC}\cE=0<\inf_{Lip}\cE<\infty$.
    However,      $f_{B(1, \ve)}^{-}(1,\ve^{-1})=0$ and $f(1-\ve,1,\ve^{-1})=\ve^{-4}$, which tends to infinity as $\ve \to 0$. Hence, the smooth function $f$ does not satisfy \eqref{eqHZconv} nor \eqref{eqHZiso}. 
\end{example}
The next example illustrates that superlinearity of \(f\) is not sufficient to exclude the Lavrentiev phenomenon.
\begin{example}[Ball \& Mizel's example~\cite{BallMizel}]\label{ex:ball-mizel}   Let $f:\R\times\R\times\R\to \R$ be given by \(f(x,t,\xi) \coloneqq (x^4-t^6)^2|\xi|^{27}+\mathcal{v}|\xi|^2\) with $\mathcal{v} > 0$. We have $f^{-}_{B(0, \ve)}(0, \ve^{-1/2}) = \mathcal{v}\ve^{-1}$, while $f(\ve, 0, \ve^{-1/2}) = \ve^{-11/2} + \mathcal{v}\ve^{-1}$, which divided by $\mathcal{v}\ve^{-1}$ grows to infinity with $\ve$ converging to $0$. Therefore, $f$ does not satisfy \eqref{eqHZconv} nor \eqref{eqHZiso}.
It is known that for the boundary conditions \(u(-1)=k_1, u(1)=k_2\)  with \(-1\leq k_1<0<k_2\leq 1\), and for \(\mathcal{v}\) sufficiently small, we have $\inf_{AC}\cE=0<\inf_{Lip}\cE<\infty$. 
\end{example}

One strength of the new results proved in the paper is the fact that we do not require continuity with respect to the first variable. However, we point out that the continuity in the second variable is needed.

\begin{example}[Cerf \&{} Mariconda's example \cite{Cerf-Mariconda}]\label{ex:cerf-mar}
    If $f:(0,1)\times\R\times\R\to \R$ is defined as $f(x,t,\xi)\coloneqq \left(\xi-\tfrac{1}{2t}\right)^2$ when \(t\not=0\) and $f(x,0,\xi)=0$, then $\inf_{AC}\cE=0$ and for every Lipschitz function $u$ that is not identically $0$, $\cE(u)=+\infty$. Note that $f_{B(x,\ve)}^{-}(t,\xi)=f(x,t,\xi)$, so $f$ does satisfy the balance condition \eqref{eqHZconv}, but it is not continuous with respect to $t$.
\end{example}

\subsection{Multidimensional examples }\label{ssec:examples}
In this section, we give examples of functionals 
\begin{equation}\label{eq:E-examples} 
    E_\Omega(u)=\int_\Omega f(x,u,\nabla u)\dx \,,
\end{equation}
for which there is no Lavrentiev phenomenon according to Theorems~\ref{theo:doubling}, \ref{th-main-conv}, and~\ref{theo:iso-ortho}. For a quick summary, see Tables~\ref{tab:examples} and~\ref{tab:examples2}. We denote the integrand in~\eqref{eq:E-examples} as $f=f(x,t,\xi)$, where $x$ stands for the spacial variable, the variable $t$ corresponds to the $u$-dependence of $f$, and the variable $\xi$ to its $\nabla u$-dependence. Let us emphasize that all examples with explicit $u$-dependence are novel, with the only remark that for $f=f(t,\xi)$ they follow from~\cite{Bousquet-Pisa} and for $f=f(x,t,\xi)$ with $(t,\xi)\mapsto f(x,t,\xi)$ being convex they are embraced by \cite{BMT}. {In this section we will use the following notation: for any real-valued function $w$ and a constant $C\in\R$, we write $C \ll w$ whenever there exists a constant $c\in\R$ such that $C<c \leq w(\cdot)$.}\newline


\noindent{\bf Examples built upon power models. } The typical examples of functionals exhibiting the Lavrentiev phenomenon are those of the variable exponent and double phase growth, i.e., when the integrand in \eqref{eq:E-examples} is given by $\wt f_{\text{v}}(x, \xi) \coloneqq |\xi|^{\wt r(x)}$ and $\wt f_{\text{d}}(x, \xi) \coloneqq  |\xi|^{p} + \wt a(x)|\xi|^q$, respectively. To exclude the Lavrentiev phenomenon for the functional involving $\wt f_{\text{v}}$, one usually requires that $1\ll \wt r  \in L^{\infty}(\Omega)$ is log-H\"older continuous ($\wt r \in \mathcal{P}^{\text{log}}$), cf.~\cite{zh}. For the functional defined by $\wt f_{\text{d}}$, the typical assumptions are $1< p< q$ and $0 \leq \wt a \in C^{0, \vk}$ with $q \leq p + \vk\max(1, p/N)$, see \cite{zh, eslemi, bgs-arma}. The latter was recently improved in~\cite{bcdfm}, allowing for $p$ and $q$ arbitrary far from each other, by considering in the place of $C^{0, \vk}$ a broader class of weights $\SP^{\vk}$, where $\vk\in (0,\infty)$ dictates a polynomial rate of vanishing decay of $\wt a$. We stress that upon these choices of parameters, the functions $\wt f_{\text{v}}$ and $\wt f_{\text{d}}$ satisfy condition~\eqref{eqHZ-iso-easy} and the corresponding energy  functionals are thus covered by Theorem~\ref{theo:iso}.

Let us prepare some notation before presenting their $u$-dependent counterparts falling into the same realm. With some abuse of typical notation, we say that a continuous function $r:\Omega\times\R\to [1,\infty)$ belongs to $\mathcal{P}^{\text{log}}(\Omega)$, if there exists a locally bounded function $C :\R \to [0, \infty)$, such that
\begin{equation}\label{eq615}
    | r(x,t) -  r(y,t)| \leq -\frac{C(t)}{\log(|x-y|)} \qquad\text{for all }\ x,y \in \Omega\,, t\in\R\,.
\end{equation}
We need also to introduce a class of functions generalizing $\SP^\vk$. Let us take a function $\omega : \rp\times \R \to \rp$ {such that}  $\omega(0,t) = 0$  and  $\omega(\cdot,t)$ is non-decreasing for every $t$. We say that a function $ a:\Omega\times\R \to [0, \infty)$ belongs to $\SP_{\omega(\cdot,\cdot)}(\Omega {\times \R})$ if $t\mapsto  a(x,t)$ is continuous for a.e. $x\in\Omega$ and there exists a locally bounded function $C:\R\to\rp$ such that
\begin{equation}\label{eq:Z-omega-t}
    a(x,t) \leq C(t)\big(a(y,t) + \omega(|x-y|,t)\big) \qquad\text{for all }\ x,y \in \Omega\,, t\in\R\,.
\end{equation}
Due to the function \(t\mapsto C(t)\) and the lack of concavity assumption  on $\omega(\cdot,t)$, the condition \eqref{eq:Z-omega-t} does not describe  the modulus of continuity of $a$ with respect to $x$. In the special case of  $\omega(s,t) = s^\vk$, the sub-class of $\SP_{\omega(\cdot,\cdot)}$  is $\SP^{\vk}$, which is meaningful for all $\vk\in (0,\infty)$. For more information about the class $\SP^{\vk}$ in the case when $a$ does not depend on $t$ we refer to~\cite{bcdfm}. 

We are now in a position to consider 
\begin{equation*}
     f_{\text{v}}(x, t, \xi) \coloneqq  |\xi|^{r(x, t)} \qquad \text{and} \qquad  f_{\text{d}}(x, t, \xi) \coloneqq|\xi|^{p} + a(x, t)|\xi|^q\,.
\end{equation*}
Let us show that $f_{\text{v}}$ verifies~\eqref{eqHZ-iso-easy} with $1=p\leq r(\cdot, \cdot)$  if $r\in\mathcal{P}^{\log}$. Given \(L_2\geq 1\) and $\ve < \exp(-1)$, let us take $|\xi| \leq L_2\ve^{-1}$. For every \(x, y\in \Omega\) with \(|x-y|<\ve\),   if \(|\xi|\geq 1\), then
\begin{equation*}
     f_{\text{v}}(x, t, \xi) = |\xi|^{  r(y, t)}  |\xi|^{ r(x, t) -   r(y, t)} \leq   f_{\text{v}}(y, t, \xi) \cdot (L_2\ve^{-1})^{-C(t)/\log(\ve)} \leq L_2^{C(t)}\exp({C(t)}) f_{\text{v}}(y, t, \xi)\,,
\end{equation*}
If instead \(|\xi|\leq 1\), we simply observe that \(f_{\text{v}}(x, t, \xi)\leq 1\).
Hence~\eqref{eqHZ-iso-easy} is proven. Additionally, $f_{\text{v}}(x, t, 0) = 0$, so by Theorem~\ref{theo:iso}, the Lavrentiev gap for the corresponding functional is excluded. We stress that the only property of $t\mapsto r(x,t)$ that is needed, besides \eqref{eq615}, is the continuity of this mapping. \\
Let us show that $f_\text{d}$ satisfies~\eqref{eqHZ-iso-easy} if  $a \in \SP^{\vk}$ with ${p\leq q} \leq p + \vk\max(1, \tfrac{p}{N})$. If $|\xi| \leq L_2\ve^{-\min(1,  {N}/{p})}$  {with \(\xi\not=0\)}, we have for every \(x, y\in \Omega\) with \(|x-y|<\ve\),
\begin{equation*}
    \frac{f_\text{d}(x, t, \xi)}{f_\text{d}(y, t, \xi)} = \frac{1 +   a(x, t)|\xi|^{q-p}}{1 +   a(y, t)|\xi|^{q-p}} \leq 1 + C(t) + C(t)|x-y|^{\vk}|\xi|^{q-p} \leq 1 + C(t) + C(t)L_2^{q-p}\,,
\end{equation*}
which means that~\eqref{eqHZ-iso-easy} holds true. Given that $f_{\text{d}}(x, t, 0) = 0$, Theorem~\ref{theo:iso} implies the absence of the Lavrentiev phenomenon in this case. Again, we stress that the only property of $t\mapsto a(x,t)$ that is needed, besides \eqref{eq:Z-omega-t}, is the continuity of this mapping. Moreover, with similar computations as for $f_\text{d}$, one can check that the function $(x,t,\xi)\mapsto |\xi|^{p(t)} +  a(x, t)|\xi|^{ q(t)}$ satisfies~\eqref{eqHZiso}  if $1 \leq p(\cdot), q(\cdot) \in C(\R)$, $a \in \SP_{\omega(\cdot,\cdot)}$ for $\omega(s,t)=s^{\vk(t)}$ where $\vk:\R\to\rp$, and $q(t) \leq p(t) +  \vk(t)\max(1, \tfrac{p(t)}{N})$ for every $t$. Moreover, for every  example given above, we can find conditions to exclude the Lavrentiev phenomenon for its orthotropic counterpart also including more phases. For instance, Theorem~\ref{theo:iso-ortho} applies to a multi-phase orthotropic function
\begin{equation*}
 f(x,t,\xi)=\sum_{i=1}^{N}\left(|\xi_i|^{p_i}+\sum_{j=1}^k a_{i,j}(x,t)|\xi_i|^{q_{i,j}}\right)\quad\text{with $a_{i,j} \in \SP^{\vk_{i,j}}$, $\vk_{i,j}>0$, and $q_{i,j} \leq p_i + \vk_{i,j}\max(1, \tfrac{p_i}{N})$}\,.
\end{equation*}
We observe that already mentioned examples embrace the best known conditions without $u$-dependence of the functional, cf.~\cite{eslemi, BCM, DeF-multi, badi, bcdfm, zh, bgs-arma}. At the same time, introducing this extra dependence does not complicate the conditions in an artificial way.\newline

{\begin{table}[ht!]
\centering
\begin{tabular}{lc}
\toprule[1.5pt]
{$f(x,t,\xi)$} & 
{{Parameters}}  \\ 
\midrule[1.5pt]
\multirow{1}{*}{$a(x,t)|\xi|^{p(x,t)}$} & \multirow{1}{*}{$1\leq p\in {\cal P}^{\log}$, $0\ll a \in L^{\infty}{(\Omega \times \R)}$, $a(x,\cdot)\in C(\R)$}\\
\midrule 
 \multirow{1}{*}{$ b(t)|\xi|^{p}\big(1+a(x,t)\log(e+|\xi|)\big)$} & \multirow{1}{*}{$1\leq p<\infty$,  $ a\in {\cal P}^{\log} $} \\
\midrule
{$b(t)\left[|\xi|^p+a(x,t)|\xi|^q\right]$}&$\ a\in \mathcal{Z}^\varkappa, \varkappa\in(0,\infty)$, $ q\leq p+\varkappa\max(1,\frac{p}{N})$  \\
\midrule
{$b(t)\big[|\xi|\log(1+|\xi|)+a(x,t)|\xi|^q\big]$}&$
a\in \mathcal{Z}^\varkappa, \varkappa\in(0,\infty)$, $q\leq 1+\varkappa\max(1,\frac{1}{N})$\\ 
\midrule
\multirow{2}{*}{{ $b(t)\left[|\xi|^{p(t)}+a(x,t)|\xi|^{q(t)}\right]$}} &{$
p, q,\vk\in C(\R;\R_+),\ a\in \mathcal{Z}_{\omega(\cdot,\cdot)},\ \omega(s,t)= s^{\vk(t)}$} \\ &{$\forall\,{t } \ q(t)\leq p(t)+\varkappa(t)\max(1,\frac{p(t)}{N})$}
\\
\midrule
{$b(t)\left[\exp(|\xi|)+a(x,t)\exp(\gamma|\xi|)\right]$} & {${|b(t)| \leq c(|t|^{N'} +1)},
a\in \mathcal{Z}^{N(\gamma-1)} {\cap L^{\infty}(\Omega\times\R)}, \gamma > 1$}
\\
\midrule
\multirow{2}{*}{$b(t)\left[|\xi|^p + a(x, t)\exp(|\xi|^q)\right]$}& \multirow{1}{*}{${|b(t)| \leq c(|t|^{N'} +1)},\ 
a\in \mathcal{Z}_{\omega(\cdot, \cdot)}{\cap L^{\infty}(\Omega\times\R)}$}\\& \multirow{1}{*}{$\omega(s) \leq \exp(s^{-\vk}),\ \vk > q\min(1, N/p)$} 
\\
\midrule
\multirow{3}{*}{$b(t)\left[\psi_0(t, |\xi|) + a(x,t)\psi_1(t, |\xi|)\right]$}
 & \multirow{1}{*}{$
\forall\,t\ \psi_0(t, \cdot), \psi_1(t, \cdot)$ -- $N$-function*, $a\in \SP_{\omega(\cdot, \cdot)}$} \\
& \multirow{1}{*}{$\forall\,{t} \ \psi_1(t, \cdot)/\psi_0(t, \cdot)$ -- non-decreasing, $\forall\,{r} \  \psi_0(\cdot, r), \psi_1(\cdot, r) \in C(\R)$} \\
& $\forall\,{L} \ \exists\,{ c } \ \forall\,{s, t} \ \ \omega(s, t) \leq c \max\left( \frac{s^{-N}}{\psi_1(t, \psi_0(t)^{-1}(Ls^{-N}))}, \frac{\psi_0(t, L^{1/N}s^{-1})}{\psi_1(t, L^{1/N}s^{-1})} \right)$ \\ 
\bottomrule[1.5pt]
\end{tabular}
\caption{Main examples of integrands for functionals $\int_\Omega f(x,u,\nabla u)\dx$ for which \eqref{eqHZiso} is satisfied and so for which we provide the absence of any Lavrentiev phenomenon by Theorem~\ref{theo:iso-ortho}. We write $f=f(x,t,\xi)$, where $x$ stands for the spacial variable, variable $t$ corresponds to $u$-dependence of $f$, and variable $\xi$ to its $\nabla u$-dependence. In all examples in the table, $b \in C(\R, (0,+\infty))$, which might be relaxed due to Remark~\ref{rem:extra-t}. The class $\SP_{\omega(\cdot,\cdot)}$  from~\eqref{eq:Z-omega-t} embraces (but is not restricted to) H\"older continuity; $\SP^\vk$, $\vk>0$ is its special case. Conditions on $\omega$ in the last example simplify under growth restrictions. *Assumption that $\psi_0,\psi_1$ are $N$-functions is given for the simplicity of the exposition.}
\label{tab:examples}
\end{table}}

\noindent{\bf Generalized Orlicz examples. } In order to include isotropic functionals of essentially non-power growth, we need to study when they verify~\eqref{eqHZiso}. The simplest example under no growth restriction is constructed by the use of an increasing,  non-negative convex function $\psi$  and reads
\begin{equation*}
     f_{\text{e}}(x, t, \xi) \coloneqq \psi(|\xi|) + a(x, t)\psi^\gamma(|\xi|)\quad \text{with }\ \gamma > 1\,.
\end{equation*}
In this case, to satisfy~\eqref{eqHZiso} one can require that $a \in \SP^{N(\gamma - 1)}$. Indeed, note that whenever $(f_{\text{e}})^{-}_{B(x, \ve)}(t, \xi) \leq L_2\ve^{-N}$  with \(\xi\not=0\), then $0<|\xi| \leq \psi^{-1}(L_2\ve^{-N})$. Therefore, for any $y \in B(x, \ve)$, we have
\begin{align*}
    \frac{f_{\text{e}}(x, t, \xi)}{f_{\text{e}}(y, t, \xi)}   
    &\leq 1 + C(t) + C(t)|x-y|^{N(\gamma-1)}\psi^{\gamma - 1}(\psi^{-1}(L_2\ve^{-N})))\leq 1 + C(t) + C(t)L_2^{\gamma-1}\,.
\end{align*}
{If $\psi$ is such that $\psi(0)=0$, or $a\in L^\infty(\Omega\times\R)$}, we can deduce the absence of any Lavrentiev gap from Theorem~\ref{theo:iso-ortho}. This example can be applied to various functions $\psi$, including slowly growing ones, e.g. $\psi(s)=s$ or $\psi(s)=s\log(1+\log(\dots (1+\log(1+s))))$, as well as fastly growing ones, e.g. when for all sufficiently large $s$ it holds $\psi(s)=\exp(s)$ or $\psi(s)=\exp(\exp(\dots \exp(s)))$. We stress that condition $a \in \SP^{N(\gamma - 1)}$ is meaningful for arbitrary $\gamma > 1$  and the only property of $t\mapsto a(x,t)$, that is needed, is the continuity of this mapping. 

 To construct a more general example, we pick two increasing, non-negative, convex functions  $\psi_0$ and $\psi_1$ on \(\rp\) such that the function $\psi_1/\psi_0$ is non-decreasing. Upon setting
\begin{equation*}
    f_{\text o}(x, t, \xi) \coloneqq \psi_0(|\xi|) + a(x, t)\psi_1(|\xi|)\,,
\end{equation*}
 one can prove that $f_{\text o}$ satisfies condition~\eqref{eqHZiso}, if $a \in \SP_{\omega(\cdot, \cdot)}\cap L^{\infty}(\Omega\times \R)$ and
\begin{equation}\label{eq:omega-cond}
    \forall\,{L > 0} \quad \exists\,{c > 0} \quad \forall\,{s > 0} \quad \omega(s) \leq c \max\left( \frac{s^{-N}}{\psi_1(\psi_0^{-1}(Ls^{-N}))}, \frac{\psi_0(L^{1/N}s^{-1})}{\psi_1(L^{1/N}s^{-1})} \right)\,,
\end{equation}
Let us point out that the assumption that $a \in L^\infty(\Omega \times \R)$ is imposed in order to guarantee~\eqref{eq:f-origin}, and is satisfied whenever $a$ does not depend on $t$.
We refer to Example~\ref{ex:f-o} for details. The condition~\eqref{eq:omega-cond} specializes to the given above conditions for functions $f_{\text v}$, $f_{\text d}$, and $f_{\text{e}}$. We note that it has also a simpler form under additional compatibility conditions on $\psi_0, \psi_1$. Namely\begin{enumerate}
    \item if $\psi_1/\psi_0 \in \Delta_2$, then~\eqref{eqHZiso} is satisfied  with $a\in\SP_{\omega(\cdot,\cdot)}$ for  $\omega(s) \leq \frac{\psi_0(s^{-1})}{\psi_1(s^{-1})}$;
    \item if $\psi_1 \circ \psi_0^{-1} \in \Delta_2$,  then~\eqref{eqHZiso} is satisfied with $a\in\SP_{\omega(\cdot,\cdot)}$ for $\omega(s) \leq \frac{s^{-N}}{\psi_1(\psi_0^{-1}(s^{-N}))}$.
\end{enumerate}

 Inspired by~\cite[Example 2.7 (2)]{Lukas23} let us consider
\begin{equation*}
    f_{\text{D}}(x, t, \xi) \coloneqq |\xi|^p + a(x, t)\exp(|\xi|^q)\,,
\end{equation*}
where $a\in \SP_{\omega(\cdot,\cdot)} \cap L^{\infty}(\Omega \times \R)$ for $\omega(s) = \exp(-s^{-\vk})$. In~\cite[Example 2.7 (2)]{Lukas23} 
the authors allow the parameter $q$ to be strictly {less} than \(1\) so that \(f_{\text{D}}\) is not convex in that case.
We demand -- via~\eqref{eq:omega-cond} -- that $\vk > q\min(1, \tfrac{N}{p})$, which means that $p\geq 1$ and $q > 0$ might be arbitrary if the function $a$ {decays fast enough near the points where it vanishes}. Moreover, unlike~\cite{Lukas23}, for the absence of the Lavrentiev gap obtained via Theorem~\ref{th-main-conv}, we allow for functionals with explicit $u$-dependence (i.e. $a=a(x,t)$), see Example~\ref{ex:Lukas} for details.

\begin{table}[ht!]
\centering
\begin{tabular}{lcc}
\toprule[1.5pt]
{$f(x,t,\xi)$} & 
{{Parameters}} & Theorem \\ 
\midrule[1.5pt]  
\multirow{1}{*}{$\sum_{i=1}^N a_i(x,t)|\xi_i|^{p_i(x,t)}$} & \multirow{1}{*}{$\forall i\quad 1\leq p_i\in {\cal P}^{\log}$, $0\ll a_i\in L^\infty{(\Omega \times \R)}$, $a_i(x,\cdot)\in C(\R)$} & Theorem \ref{theo:iso-ortho} \\  \cmidrule{1-3}  
\multirow{2}{*}{$b(t)\left(\sum_{i=1}^N |\xi_i|^{p_i}+\sum_{i=1}^N a_i(x,t)|\xi_i|^{q_i}\right)$}& \multirow{1}{*}{ $\forall\, i\ \big\{\big(p_i\leq q_i\leq p_i+\vk_i{\max\left(1, \frac{p_i}{N}\right)}$ and $ a_i\in \mathcal{Z}^{\vk_i},\,\vk_i>0\big)$,} & \multirow{2}{*}{Theorem \ref{theo:iso-ortho}}  
\\
  &
  or  $\big(q_i\leq p_i$ and $a_i\in L^\infty{(\Omega \times \R)}\big)\big\}$ & \\
\midrule
\multirow{2}{*}{$b(t)\left[\psi_0(|\xi|)+\sum_{i=1}^N a_i(x,t)\psi_i(|\xi_i|)\right]$} & \multirow{1}{*}{ 
$\forall\,i\ \psi_i$ -- $N$-function*, $\psi_i/\psi_0$ -- non-decreasing }& \multirow{2}{*}{Theorem \ref{th-main-conv}} \\&
 $a_i \in \mathcal{Z}_{\omega_i(\cdot,\cdot)},\ \omega_i$ as in~\eqref{eq:omega-cond} &\\
\midrule
\multirow{1}{*}{\ $\sum_{i=1}^N f_i(x,t, |\xi_i|)$} & \multirow{1}{*}{$\forall\,{i} \ f_i \text{ is substituted by any $f$ from Table~\ref{tab:examples}}$} & Theorem \ref{theo:iso-ortho} \\
\midrule
\multirow{1}{*}{\ $b(t)\left(\psi(|\langle {\upsilon}(x), \xi \rangle|) + |\xi|^{N/\gamma}\right)$} & \multirow{1}{*}{$\psi \in \Delta_2$ -- $N$-function*, ${\upsilon} \in C^{0, \gamma}(\Omega;\rn)$, $\gamma \in (0, 1]$} & Theorem \ref{theo:doubling} \\
 \bottomrule[1.5pt] 
\end{tabular}
\caption{Main examples of anisotropic integrands for functionals $\int_\Omega f(x,u,\nabla u)\dx$ for which we provide the absence of the Lavrentiev phenomenon. We write $f=f(x,t,\xi)$, where $x$ stands for the spacial variable, variable $t$ corresponds to $u$-dependence of $f$, and variable $\xi$ to its $\nabla u$-dependence. In all examples in the table, $b \in C(\R, (0,+\infty))$, which might be relaxed due to Remark~\ref{rem:extra-t}. The class $\SP_{\omega(\cdot,\cdot)}$  from~\eqref{eq:Z-omega-t} embraces (but is not restricted to) H\"older continuity; $\SP^\vk$, $\vk>0$ is its special case. *Assumption that certain functions are $N$-function is given for the simplicity of the exposition.  }
\label{tab:examples2}
\end{table}
 
 \bigskip

{\textbf{Fully anisotropic examples.}} Inspired by \cite{BaaBy,DeFilippisRegularity, BaaByOh,DeF-multi} we present an anisotropic Orlicz multi-phase example. Let us consider radially increasing,  convex  functions $(\psi_j)_{j=0}^k$, such that each $\psi_j$ with \(1\leq j \leq k\) grows essentially faster than $\psi_0$ at infinity:
\begin{equation*}
    f (x, t, \xi) \coloneqq \psi_0(\xi) + \sum_{j=1}^k a_j(x, t)\psi_j(\xi)=\sum_{j=1}^k\left(\tfrac{1}{k}\psi_0(\xi) +  a_j(x, t)\psi_j(\xi)\right)=:\sum_{j=1}^k\vt_j(x,t,\xi)\,.
\end{equation*}

As already observed in item {\it (iii)} before Theorem~\ref{th-main-conv}, to justify that this function satisfies \eqref{eqHZconv}, it is enough that this fact holds for each $\vt_j$. 
Assuming further that the functions \(\psi_j\) are isotropic, and relying on Theorem~\ref{theo:final} below, we only need to require that
condition~\eqref{eqHZiso} is satisfied for every ${\vartheta}_j$ which is the case when $a_j\in \SP_{\omega_j(\cdot,\cdot)}$ for  $\omega_j$ satisfying~\eqref{eq:omega-cond} with $\psi_j$ in the place of $\psi_1$. This condition applied to isotropic Lagrangian fully covers the scope of \cite[Theorem 3.1]{BaaBy} and extends it in three directions. We  allow for explicit $u$-dependence of the considered functional and do not need to assume that $\psi_j \in \Delta_2$ for any $j$. Moreover, in~\cite{BaaBy} the imposed compatibility condition forces the closeness of phases expressed as $\limsup_{|\xi|\to\infty}\frac{\psi_j(x,|\xi|)}{|\xi|\psi_0(x,|\xi|)}<\infty$ for every $j$. Upon the condition above, the function $\omega_j$ does not play a role of the modulus of continuity, because we do not require its concavity, so no closeness of phases is needed. 

Let us present another fully anisotropic example using the fact that the function $|\langle x, \xi \rangle|$ does not admit an orthotropic decompostion. We can consider Lagrangians of the form
\begin{equation*}
    f_{\text{a}}(x, \xi) \coloneqq \psi(|\langle {\upsilon}(x), \xi \rangle|) + |\xi|^{N/\gamma}\,,
\end{equation*}
where $\psi \in \Delta_2$ is an increasing,  convex  function, ${\upsilon} : \rn \to \rn$ is in $C^{0, \gamma}$, $\gamma \in (0, 1]$. Then, the integrand $f_{\text{a}}$ satisfies assumptions of Theorem~\ref{theo:doubling}, see Example~\ref{ex:f-a} for details.

\begin{remark}\label{rem:extra-t}
    \rm In every above example,
 we can freely multiply the Lagrangian \(f\) by an extra $t$-dependent continuous function $b:\R\to (0,\infty)$ still keeping the  balance conditions satisfied.  In the case when ~\eqref{eqHZ-iso-easy} is satisfied, the function $b$ could also vanish. A similar remark holds for a more restrictive variant of \eqref{eqHZconv} when the condition  $\left(f^{-}_{B(x,\ve)}(t,\cdot)\right)^{**}\left(\xi\right)+|\xi|^{\max(p,N)}
\leq {\mathcal{k}_2}{\ve^{-N}}$ is replaced by $\ |\xi|^{\max(p,N)}
\leq {\mathcal{k}_2}{\ve^{-N}}$. Then if $f$ satisfies such a  modified \eqref{eqHZconv}, the same condition holds for $b(t)f(x,t,\xi)$ and Theorem~\ref{th-main-conv} can be applied to the latter.
\end{remark}

\section{Preliminaries}\label{sec:prelim}
\subsection{Notation} \label{ssec:notation}

In the sequel $\Omega\subset\rn$ is a fixed bounded open set with Lipschitz boundary. Moreover,  \(\Lambda\) denotes a bounded open set which  contains the closure of $\Omega$. We also formulate some arguments with an arbitrary  open set \(\Lambda'\Subset \Lambda\) (the latter meaning that \(\overline{\Lambda'}\subset \Lambda\)). Throughout the paper we assume that $f:\Omega\times\R\times\rn\to\rp$ or $f:\Lambda\times\R\times\rn\to\rp$ is a Carath\'eodory function, i.e. it is measurable with respect to the first variable and continuous with respect to the second and the third variables. Moreover, $f$ is always required to be convex with respect to the last variable. For a function $u$ defined on \(\Lambda\) which shall be clear from the context, we will introduce two Borel subsets \(\Lambda_+\) and \(\Lambda_0\) of \(\Lambda\) such that \(\Lambda=\Lambda_+\cup \Lambda_0\) and
\begin{equation}
    \label{def-Lambda+}
 |\nabla u|>0 \textrm{ a.e. on } \Lambda_+ \textrm{ and } |\nabla u|=0 \textrm{ a.e. on } \Lambda_0
\,.\end{equation}

For a set $U\subset\R^d$, we denote by
{$C^{0}_c(U)$ (resp. $C_c^\infty(U)$) the set of continuous (resp. smooth) functions with compact support in $U$, while  $W^{1,p}_\vp (\Omega)$, $p\geq 1$ stands for the set of weakly differentiable functions in $\Omega$ with $p$-integrable weak gradients and trace $\vp$ on $\partial\Omega$. Finally,  $W^{1,\infty}_\vp(\Omega)$  is the set of Lipschitz functions in $\Omega$ agreeing with $\vp$ on $\partial\Omega$.
}

 For every \(r\in \R\), we denote by \(r_+\) the positive part of \(r\), namely \(r_+=\max(r,0)\). 

%
\noindent The partial derivative of a Sobolev function \(w:\R^N\to \R\) with respect to a unit vector \(e\in \R^N\) is denoted by \(\partial_e w\).  For a convex function $w: [0, \infty) \to [0, \infty)$, we denote by $D^+w(s)$, $s \geq 0$, the right-hand side derivative of $w$ in point~$s$.

\noindent For any measurable map \(w:U\to \R\), $U\subset\R^d$, we define its graph 
\[
\Gr_w\coloneqq\{(x,t)\in U\times\R:\ t=w(x)\}\,.
\]
For a given Borel function \(w: U\times \R\to \R\), which is non-increasing and left-continuous with respect to the second variable, we define the {\it generalized inverse} with respect to the second variable by 
\begin{equation}\label{def-gen-inv}
w^{-1}(x,s)\coloneqq \inf \{t\in \R :\ w(x,t)\leq s\}\in [-\infty,+\infty] \qquad\text{for every }\ x\in U \,.    
\end{equation}
In particular, if for a given \(x\in U\), there is no \(t\in \R\) such that \(w(x,t)\leq s\), then \(w^{-1}(x,s)=+\infty\). If instead, \(w(x,t)\leq s\) for every \(t\in \R\), then \(w(x,t)=-\infty\).

\noindent For a function $w:\R^d\to\R$ bounded from below by an affine function,  we define its {\it greatest convex minorant} $w^{**}:\R^d\to\R$ as the supremum of all convex  functions which are not larger  than $w$ in the whole $\R^d$. It can  also  be obtained by applying the Young conjugation operation $*$ twice, see e.g.~\cite[Corollary~2.1.42]{C-b}. Note however that this latter fact will not be used in the sequel.
\noindent We say that a function \(w:\rn \to \R\)  is {\it superlinear} if
\[
\lim_{|\xi|\to \infty}\frac{w(\xi)}{|\xi|}=\infty\,.
\]
We say that \(f:\Omega\times \R\times \R^N\to [0,\infty)\) is superlinear if there exists a function \(w:\R^N\to \R\) as above such that for a.e. \(x\in \Omega\), for every \((t,\xi)\in \R\times \R^N\),
\begin{equation}\label{eq-superlinearity-sv}
f(x,t,\xi)\geq w(\xi).
\end{equation}
We say that a function $w:[0,\infty)\to[0,\infty)$ is an {\it $N$--function} if it is convex, continuous,  and such that $w(0)=0$, $\lim_{t \to 0}{w(t)}/{t}=0$ and $\lim_{t \to \infty}{w(t)}/{t}=\infty$. \\
 We say that a Carath\'{e}odory function $w:\Omega\times\R\times\rn\to\rp$ satisfies the $\Delta_2$ condition (denoted $w(x,t,\cdot)\in\Delta_2$) if there exists a constant $c>0$ independent of $x$ and $t$ such that it holds 
\begin{equation*}
 w(x,t,2\xi)\leq c(w(x,t,\xi)+1)\qquad\text{for a.e. $x\in\Omega$ and all } (t,\xi)\in\R\times\rn\,.
\end{equation*}
\noindent For a standard regularizing kernel $\vr\in C_c^\infty(B(0,1))$ and $\ve>0$ we denote \begin{equation}\label{def-vr-ve}
\vr_{\ve}(x)\coloneqq\ve^{-N}\vr(x/\ve)\,.
\end{equation}
Note that \(\vr_{\ve}\in C^{\infty}_c(B(0,\ve))\) for every \(\ve>0\). Recall that given a Borel function \(v:\rn\times \R \to \R\), the convolution with respect to \(x\) is defined as follows: for every \((x,t)\in \rn\times \R\), we set
\[
v*_x\vr_{\ve}(x,t)\coloneqq\int_{\rn}v(x-y,t)\vr_{\ve}(y)\dy \,.
\]
For an \(\R^{N+1}\)-valued Borel measure \(\mv\) on \(\Lambda\times \R\), we define the measure \(\mv*_x\vr_{\ve}\) by setting for every Borel set \(A' \subset \rn\times \R\):
\begin{equation}\label{eq-def-conv-partial-bis}
\mv*_x\vr_{\ve}(A')\coloneqq\int_{\Lambda\times \R}\left(\int_{\rn}\chi_{A'}(y,t)\vr_{\ve}(y-x)\dy  \right)\d\mv(x,t)\,,
\end{equation}
where \(\chi_{A'}\) denotes the indicator function of \(A'\).
Hence, for every bounded Borel function \(h:\rn\times \R \to \R\), it holds
\begin{equation}\label{eq-def-conv-partial}
\int_{\rn\times \R} h(x,t)\d (\mv*_x\vr_{\ve})(x,t)=\int_{\Lambda\times \R}\left(\int_{\rn}h(y,t)\vr_{\ve}(y-x)\dy  \right)\d \mv(x,t)\,.
\end{equation}

Another important  tool is the disintegration of finite measures defined on \(\Omega\times \R\). To detail this technique, we closely follow \cite[Section 2.5]{AmFuPa}. Let \(\mu\) be a positive Radon measure on \(\R\) and  \((\nu_t)_{t\in \R}\) a family of \(\R^{N+1}\)-valued  measures  on \(\Omega\) such that the function \(t\mapsto \nu_t(B)\) is \(\mu\)-measurable for any Borel set \(B\subset \Omega\). 
We further assume that 
\begin{equation}
t\mapsto |\nu_t|(\Omega) \textrm{ belongs to  } L^{1}(\R,\mu)\,.
\end{equation}
We  denote by \(\nu_t\otimes \mu\) the \(\R^{N+1}\)-valued  measure on \(\Omega\times \R\) defined for every Borel set  \(B\subset \Omega\times \R\) by
\begin{equation}
(\nu_t\otimes \mu)(B):=\int_{\R}\left(\int_{\Omega}\chi_B(x,t)\d\nu_t(x) \right)\d\mu(t)\,.
\end{equation}
It then follows that 
\begin{equation}\label{eq578}
\int_{\Omega\times \R} g(x,t)\d(\nu_t\otimes \mu)(x,t)=\int_{\R}\left( \int_{\Omega}g(x,t)\d\nu_t(x)\right)\d\mu(t)\,,
\end{equation}
 for every bounded Borel map \(g:\Omega\times \R\to \rp\).
The above definitions also make sense when \(\nu_t\) is a positive finite measure. Then \eqref{eq578} holds true for  any non-negative Borel function \(g\). 
Conversely, given a finite \(\R^{N+1}\)-valued  measure  on \(\Omega\times \R\), it can be written as a product of the form \(\nu_t\otimes \mu\) as above, where 
\(\mu\) is finite and \(|\nu_t|(\Omega)=1\) for \(\mu\) a.e. \(t\in \R\), see \cite[Theorem 2.28]{AmFuPa}.

We will make use of a non-decreasing sequence of continuous decreasing functions \(\theta_k:\R\to [0,1]\) for \(k\geq 1\) such that \(\theta_k\) is supported in \((-\infty,-\frac{1}{k})\) and for $t\in\R$ it holds
\begin{equation}
    \label{theta_k}   \lim_{k\to \infty}\theta_k(t)= \chi_{(-\infty,0)}(t)\,.
\end{equation}
\subsection{Precise representatives }
\label{sec:GMT}
In this paragraph, we closely follow \cite{MSZ}.
Given  \(q\in \N\), a set $E\subset\R^N$ is said to be  \emph{countably \(\cH^q\)-rectifiable} if there exist a family \((E_k)_{k\in \N}\) of subsets in  \(\R^q\) and for every \(k\in \N\), a Lipschitz map \(f_k:E_k\to \R^N\) such that $\cH^q\big( E\backslash \bigcup\limits_{k\in\N} f_k(E_k)\big)=0$.

If \(\Omega\subset \R^N\) is an open set and \(u\in L^{1}_{\textrm{loc}}(\Omega)\), then a representative \(\widetilde{u}\)  of \(u\) is said to be a \emph{precise representative} if 
\[
\widetilde{u}(x)\coloneqq\lim_{r\to 0}\frac{1}{|B_r|}\int_{B(x,r)} u(y)\dy
\]
at all \(x\) where the limit exists. For every \(u\in W^{1,1}(\Omega)\), a precise representative is unique up to a \(\cH^{N-1}\)-negligible set.

\begin{proposition}\cite[Theorems~1.1 and~1.2]{MSZ}
Let \(u\in W^{1,1}(\Omega)\) be precisely represented. Then,  the level set \(u^{-1}(t)\)  is countably \(\cH^{N-1}\)-rectifiable for a.e. \(t\in\R\), the graph \(\Gu=\{(x,u(x)):x\in \Omega\}\) of \(u\) is countably \(\cH^N\)-rectifiable and moreover, for every measurable set \(E\subset \Omega\), it holds 
\begin{equation}\label{eq-coarea-Sobolev}
    \int_{E}|\nabla u(x)|\dx = \int_{\R}\cH^{N-1}(E\cap u^{-1}(t))\dt\,,
\end{equation}
\begin{equation}\label{eq-area-Sobolev}
\int_{E}\sqrt{1+|\nabla u(x)|^2}\dx=\cH^N(\Gu\cap (E\times \R))\,. 
\end{equation}
\end{proposition}
From the \emph{coarea formula} \eqref{eq-coarea-Sobolev}, we deduce by a standard argument that
for every non-negative Borel measurable function $g:\Omega\times \R\to \rp$, we have
\begin{equation}\label{eq:hdH-hdx-coarea}
\int_\Omega g(x,u(x))|\nabla u(x)|\dx=\int_\R\int_{u^{-1}(t)} g(y,t)\d\cH^{N-1}(y)\dt.
\end{equation}
Similarly, from the \emph{area formula} \eqref{eq-area-Sobolev}, we have
\begin{equation}\label{eq:hdH-hdx-area}
\int_\Omega g(x,u(x))\sqrt{1+|\nabla u(x)|^2}\dx=\int_{\Omega\times \R} g(x,t)\d\cH^{N}\mres \Gu.
\end{equation}


%

%
{For every $u \in W^{1, 1}(\Omega)$, it follows from the area formula~\eqref{eq-area-Sobolev} that if \(u\) is precisely represented, its graph map satisfies the Lusin condition, i.e., for every Lebesgue measurable set $E \subset \Omega$,
\begin{equation}\label{eq-Lusin}
    |E|=0\quad \Longrightarrow\quad \cH^{N}\left( \{(x, u(x)) : x \in E\} \right)=0\,.
\end{equation}
}
%
For every \(t\in \R\) such that \(u^{-1}(t)\) is countably \(\cH^{N-1}\)-rectifiable, for \(\cH^{N-1}\) a.e. \(x\in u^{-1}(t)\), there exists an approximate tangent space \(T_xu^{-1}(t)\), which means that
\begin{equation}\label{eq-approximate-tangent-space}
\lim_{r\to 0}\frac{1}{r^{N-1}} \int_{u^{-1}(t)}\phi\left(\frac{y-x}{r}\right) \d\cH^{N-1}(y) =\int_{T_xu^{-1}(t)} \phi(y)\d\cH^{N-1}(y)\,,  \end{equation}
for every \(\phi\in C^{0}_c(\R^N)\). For a proof of this result, holding for any countably rectifiable set, see e.g. \cite[Theorem~11.6]{Simon}.

\section{Between balance conditions \texorpdfstring{\eqref{eqHZiso} and \eqref{eqHZconv}}{}}
\label{sec:twostars} 

This section is devoted to the proof of Theorem~\ref{theo:final}
. The reasoning is based on several geometrical observations of the properties of convex minorants, which we now present.

\begin{lemma}\label{lem:at}
    Let ${w} : [0, \infty) \to [0, \infty)$ be a continuous and non-decreasing function. For every $t \geq 0$, there exists $a_t \in [0, t]$ such that     ${w}^{**}$  is affine on $[a_t, t]$  and ${w}(a_t) = {w}^{**}(a_t)$.
\end{lemma}
\begin{proof}     The result is clear if $w^{**}$ is a constant function, as then $w^{**}$ is constantly equal to $w(0)$. Note also that as $w$ is non-decreasing, also $w^{**}$ is non-decreasing. Additionally, by monotonicity of the derivative, if $w^{**}$ is not constant, then it is unbounded. The result is also true if $w^{**}(t) = w(t)$, as then one can take $a_t = t$.

    Let us then assume that $w^{**}(t) < w(t)$ and that $w^{**}$ is unbounded. We take $a_t \leq t$ to be the minimal number such that $w^{**}$ is affine on $[a_t, t]$. It suffices to show that $w^{**}(a_t) = w(a_t)$. The result holds if $a_t = 0$. Hence, let us assume that $a_t > 0$.
    
    Suppose by contradiction that $w^{**}(a_t) < w(a_t)$. Let us take any $s < a_t$ such that {$w^{**}(a_t) < w(s)$}. As $w^{**}$ is unbounded, there exists $s_2 > a_t$ such that $w^{**}(s_2) = w(s)$. Let us define
    
    \begin{equation*}       \wt{w}(\tau) = \begin{cases}            w^{**}(\tau) &\text{ for $\tau \not\in [s, s_2]$,}\\
            w^{**}(s) + (\tau - s)\frac{w^{**}(s_2) - w^{**}(s)}{s_2 - s} &\text{ for $\tau \in [s, s_2]$.}
        \end{cases}
    \end{equation*}
    
    Note that $w^{**} \leq \wt{w}$ and $\wt w$ is convex as the maximum of $w^{**}$ and an affine function. It is also true that $\wt w$ is a minorant of $w$. Indeed, if $\tau \not\in [s, s_2]$, then $\wt w(\tau) = w^{**}(\tau) \leq w(\tau)$. On the other hand, if $\tau \in [s, s_2]$, then
    
    \begin{equation*}      \wt w(\tau) \leq \wt w(s_2) = w^{**}(s_2) = w(s) \leq w(\tau)\,.    \end{equation*}
    
    As $w^{**} \leq \wt w$, and $w^{**}$ is the greatest convex minorant of $w$, we have $w^{**} = \wt w$. Therefore, $w^{**}$ is affine on $[s, s_2]$. This however means that $w^{**}$ is affine on $[s, t]$, where $s < a_t$, which contradicts the definition of $a_t$. Therefore, we have that $w^{**}(a_t) = w(a_t)$.
\end{proof}
The following lemma  gives a lower bound on the derivative of ${w}^{**}$ whenever ${w}$ is the essential infimum of convex functions.
\begin{lemma}\label{lem:est-der}
    Let $B$ be a non-empty Borel subset of $\R^N$ and  $\big\{{w}_y\big\}_{y \in B}$ be a family of non-negative and non-decreasing convex functions on $\rp$. Let us denote ${w} \coloneqq \essinf_{y\in B} {w}_y$. Then for all $s \geq 0$, it holds that 
    $$D^+{w}^{**}(s) \geq \essinf_{y \in B} D^+{w}_y(s).$$
\end{lemma}
\begin{proof}
    Suppose the contrary, i.e. that $D^+{w}^{**}(s) < \essinf_{y \in B} D^+{w}_y(s) \eqqcolon \mathcal{d}_s$ for some $s$. Let us define  
    \begin{equation*}
        \wt{w}(\tau) \coloneqq \begin{cases}
            {w}^{**}(\tau) &\text{ for $\tau \in [0, s]\,$,}\\
            {w}^{**}(s) + \mathcal{d}_s(\tau - s) &\text{ for $\tau > s\,$.} 
        \end{cases}
    \end{equation*}
    Observe that $\wt{w}$ is convex by the monotonicity of its derivative. It is also clear that $\wt{w}(\tau) \leq {w}(\tau)$ for $\tau \leq s$. Moreover, for $\tau > s$ and a.e.  $y \in B$, we have
    \begin{equation*}
        \wt{w}(\tau) = {w}^{**}(s) + \mathcal{d}_s(\tau - s) \leq {w}_y(s) + D^+{w}_y(s)(\tau - s) \leq {w}_y(\tau)\,.
    \end{equation*}
    As the last inequality is true for a.e. $y \in B$, we get $\wt{w} \leq {w}$. Therefore, $\wt{w}$ is a convex minorant of ${w}$. However, as $\mathcal{d}_s > D^+{w}^{**}(s)$, we have ${w}^{**}(\tau) < \wt{w}(\tau)$ for some $\tau > s$, which is a contradiction.
\end{proof}

We are now in a position to prove Theorem~\ref{theo:final}.
\begin{proof}[Proof of Theorem~\ref{theo:final}]
    We  start with proving assertion $(i)$. Observe that we can assume that $f(x, t, 0) = 0$ for a.e. \(x\in \Omega\) and every \(t\in \R\). Indeed, if it is not true, we can replace $f$ by the function $g(x, t, \xi) = (f(x, t, \xi) - f(x, t, 0))_+$. By Lemma~\ref{lm-reduction-2-0}, the function $g$ satisfies~\eqref{eqHZconv} if and only $f$ does, and the same is true for~\eqref{eqHZiso}. Therefore, proving that $g$ satisfies~\eqref{eqHZconv} yields the desired result. 

In the statement of  \eqref{eqHZiso}, there is no loss of generality in replacing  \(\mathcal{k}= (\mathcal{k}_1, \mathcal{k}_2)\in \rp^2\) by \(\mathcal{k}= (\mathcal{k}_1, \mathcal{k}_2)\in (\N_*)^2\). Using that a countable union of negligible sets is negligible, we infer that there exists a negligible set in \(\Omega\) such that \eqref{eqHZiso} holds for every \(x\) in its complement and for every \(\mathcal{k}\in \rp^2\).    
Let us fix such an $x \in \Omega$ and  $\mathcal{k} = (\mathcal{k}_1, \mathcal{k}_2)$, $\ve > 0$, $t \in [-\mathcal{k}_1, \mathcal{k}_1]$. For a.e.  $y \in B(x, \ve)\cap \Omega$, let us denote
    \begin{equation*}
        {w}_y(s) \coloneqq f(y, t, s)\,, \quad {w}(s) \coloneqq f^-_{B(x, \ve)}(t, s)\,.
    \end{equation*}
For a.e.  $y$,  since   ${w}_y$ is non-negative and vanishes at \(0\), it achieves its global minimum in $0$. As ${w}_y$ is convex, this implies that it is non-decreasing. Therefore, ${w}$ is also non-decreasing as the essential infimum of non-decreasing functions. Consequently, ${w}^{**}$ is non-decreasing as well. Let us also notice that for a.e.  $y$ and any $s \geq 0$, by convexity of ${w}_y$, it holds
    \begin{equation*}
        D^+{w}_y(s) \leq {w}_y(s+1) - {w}_y(s) \leq {w}_y(s+1)\,,
    \end{equation*}
    which means that the family $\{{w}_y\}_{y \in B(x, \ve)\cap \Omega}$ is uniformly locally Lipschitz, as by Lemma~\ref{lem-bdd-f}, the function $f$ is bounded on bounded sets. Therefore, ${w}$ is continuous.
    
    Let us now take $s\in \rp$ such that ${w}^{**}(s) + s^{\max(p,N)} \leq \mathcal{k}_2\ve^{-N}$. As ${w}$ is non-decreasing and continuous, by Lemma~\ref{lem:at}, there exists $a_s \in [0, s]$ such that ${w}^{**}$  is affine on $[a_s, s]$ and $w(a_s) = {w}^{**}(a_s)$.
    For every  \(y\in B(x,\ve)\), we have \(B(x,\ve)\subset B(y,2\ve)\), and thus 
    \begin{equation}\label{eq951}
        f^-_{B(y, 2\ve)}(t, a_s)\leq f_{B(x,\ve)}^{-}(t,a_s)=w(a_s)=w^{**}(a_s).
    \end{equation}
    Since \(w^{**}\) is non-decreasing, this implies that
    \[
    f^-_{B(y, 2\ve)}(t, a_s)+a_{s}^{\max(p,N)} \leq w^{**}(s)+s^{\max(p,N)}\leq \mathcal{k}_2\ve^{-N}=(2^N\mathcal{k}_2)(2\ve)^{-N}\,.
    \]
    By~\eqref{eqHZiso}, there exists \(\cC_{\widehat{\mathcal{k}}}\geq 1\), with \(\widehat{\mathcal{k}}\coloneqq (\mathcal{k}_1,2^N\mathcal{k}_2)\), such that for a.e. \(y\in B(x,\ve)\cap \Omega\), it holds
    \begin{equation}\label{eq:seqas}
        {w}_y(a_s) \leq \cC_{\widehat{\mathcal{k}}}\left((f^-_{B(y, 2\ve)}(t, \cdot))^{**}(a_s)+1 \right)\leq \cC_{\widehat{\mathcal{k}}}\left({w}^{**}(a_s) + 1\right)\,,
    \end{equation}
    where the last inequality follows from~\eqref{eq951}. In particular, the above inequality holds when \(y=x\).
Assume that $a_s < s$ and let \(s'\in (a_s,s)\). Using that ${w}^{**}$ is affine on $[a_s, s]$ together with ~\eqref{eq:seqas}, we get
\[
    {w}^{**}(s') = {w}^{**}(a_s) + D^+{w}^{**}\left(s'\right)(s' - a_s) {\geq} \tfrac{1}{\cC_{\widehat{\mathcal{k}}}}{w}_y(a_s) - 1 + D^+{w}^{**}\left(s'\right)(s' - a_s)\,.
\]    
Given \(\delta>0\), it follows from Lemma~\ref{lem:est-der} that for every \(y\) in a non-negligible subset of \(B(x,\ve)\cap \Omega\), it holds
    \begin{equation}\label{eq:dev-inq}
         D^+{w}^{**}\left(s'\right) \geq D^+{w}_y(s')-\delta\,.
    \end{equation}
      Using additionally that $\cC_{\widehat{\mathcal{k}}} \geq 1$, we obtain
    \begin{align*}
        {w}^{**}(s') &\geq \tfrac{1}{\cC_{\widehat{\mathcal{k}}}}{w}_{y}(a_s) - 1 + D^+{w}_y(s')(s' - a_s) - \delta(s' - a_s)\\
        &\geq \tfrac{1}{\cC_{\widehat{\mathcal{k}}}}\left({w}_{y}(a_s) + D^+{w}_y(s')(s' - a_s)\right) - 1 - \delta(s' - a_s)\,.
    \end{align*}
    Then, by the monotonicity of $D^+{w}_y$ and the fact that ${w}_y \geq {w}$ for a.e. \(y\), we have
    \begin{align*}
        {w}^{**}(s') {\geq}&   \tfrac{1}{\cC_{\widehat{\mathcal{k}}}}\left({w}_{y}(a_s) + \int_{a_s}^{s'} D^+{w}_y(\tau) \d \tau\right) - 1 - \delta(s' - a_s)\\
        &= \tfrac{1}{\cC_{\widehat{\mathcal{k}}}}{w}_{y}(s') - 1 - \delta(s' - a_s) \overset{}{\geq} \tfrac{1}{\cC_{\widehat{\mathcal{k}}}}{w}(s') - 1 - \delta(s' - a_s)\,.
    \end{align*}
    Letting $\delta$ to $0$ and then \(s'\) to \(s\), we get ${w}^{**}(s) \geq \tfrac{1}{\cC_{\widehat{\mathcal{k}}}}{w}(s) - 1$, or equivalently, 
    \begin{equation}\label{eq:help2}
        {w}(s) \leq \cC_{\widehat{\mathcal{k}}}\left({w}^{**}(s) + 1\right)\,.
    \end{equation}
    In particular, ${w}(s) + s^{\max(p, N)} \leq \cC_{\widehat{\mathcal{k}}}\mathcal{k}_2\ve^{-N} + \cC_{\widehat{\mathcal{k}}}$. Without loss of generality, we may assume that $\ve \leq \text{diam}(\Omega)$ and that $\mathcal{k}_2$ is sufficiently large to ensure that $\mathcal{k}_2\text{diam}(\Omega)^{-N} \geq 1$. Therefore, we have ${w}(s) + s^{\max(p, N)} \leq 2\cC_{\widehat{\mathcal{k}}}\mathcal{k}_2\ve^{-N}$. Denoting $\wt{\mathcal{k}} =(\mathcal{k}_1, 2\cC_{\widehat{\mathcal{k}}} \mathcal{k}_2)$, we get from~\eqref{eqHZiso} that ${w}_x(s) \leq \cC_{\wt{\mathcal{k}}}\left({w}(s) + 1\right)$, which by~\eqref{eq:help2} yields
    \begin{equation}\label{eq:fin}
        {w}_x(s) \leq \cC_{\wt{\mathcal{k}}}\left(\cC_{\widehat{\mathcal{k}}}\left({w}^{**}(s) + 1\right) + 1\right) \leq \left(\cC_{\wt{\mathcal{k}}} \cC_{\widehat{\mathcal{k}}} + \cC_{\wt{\mathcal{k}}} \right)\left({w}^{**}(s) + 1\right)\,.
    \end{equation}
    Using~\eqref{eq:seqas} in case of $a_s = s$, and~\eqref{eq:fin} otherwise, we get
    \begin{equation*}
        {w}^{**}(s) + s^{\max(p, N)} \leq \mathcal{k}_2\ve^{-N} \Rightarrow {w}_x(s) \leq \left(\cC_{\wt{\mathcal{k}}} \cC_{\widehat{\mathcal{k}}} + \cC_{\wt{\mathcal{k}}} + \cC_{\widehat{\mathcal{k}}} \right)\left({w}^{**}(s) + 1\right)\,,
    \end{equation*}
    which     is~\eqref{eqHZconv} for $f$ with  $\wt \cC_\mathcal{k} = \left(\cC_{\wt{\mathcal{k}}} \cC_{\widehat{\mathcal{k}}} + \cC_{\wt{\mathcal{k}}} + \cC_{\widehat{\mathcal{k}}} \right)$. \newline
    
    Let us now prove assertion $(ii)$. Let us fix $i \in \{1, 2, \dots, N\}$. As $f_i$ satisfies~\eqref{eqHZiso}, the function $(x, t, \xi) \mapsto f_i(x, t, |\xi|)$ satisfies~\eqref{eqHZconv}. It easily follows  that the  function $(x, t, \xi) \mapsto f_i(x, t, |\xi_i|)$ also satisfies~\eqref{eqHZconv}. Hence, $f$ is a sum of functions satisfying~\eqref{eqHZconv}, which implies that $f$ satisfies~\eqref{eqHZconv}.

\end{proof}

\section{Preliminaries to the proof of Theorem~\ref{th-main-conv}}\label{sec:prelim-to-main-proof}

\subsection{\texorpdfstring{The function $1_u$}{The hypograph function}}
In the following, we systematically choose precise representatives  in \(L^{1}(\Omega)\). For such a representative \(u:\Omega\to \R\), we define the map $v: \Omega\times \R \to[0,1]$ by the following formula
\begin{equation}\label{1u}
    v(x,t)\coloneqq{1}_u(x,t)=\begin{cases}
        1\,, & \textrm{ if } t\leq  u(x)\,,\\
        0\,, & \text{ otherwise}\,.
    \end{cases}
\end{equation}

In order to calculate the distributional derivative of \(v\), we first establish the following technical fact.
\begin{lemma}
\label{lemma-alt-Stokes-Sobolev}
Let \(u\in W^{1,1}(\Omega)\),  \(\Phi\in C^{\infty}_c(\Omega\times \R)\), \(\theta\in L^{\infty}(\R)\), and  \(e\in \mathbb{S}^{N-1}\). Then
\[
\int_{\Omega}\theta\circ u(x)\Phi(x,u(x))\partial_e u(x)\dx=-\int_{\R}\theta(t)\left(\int_{[u\geq t]}\partial_e\Phi(x,t)\dx\right)\dt.
\]
\end{lemma}
\noindent Here, \(\partial_e u\) is the directional derivative of \(u\) in the direction \(e\), namely \(\partial_e u=\langle \nabla u, e\rangle\).
\begin{proof}
 Let us define the Lipschitz continuous function $g: \Omega \times \R \to\R$ via
 \[
 g(x,s)\coloneqq \int_{-\infty}^{s}\Phi(x,t)\theta(t)\dt\,.
 \]
Since \(\Phi\) is compactly supported in \(\Omega\times \R\), there exists a compact subset \(K\Subset \Omega\) such that \(g(x,s)=0\) for every \((x,s)\in (\Omega\setminus K)\times \R\). By the chain rule, the function \(G:x\mapsto g(x,u(x))\) belongs to \(W^{1,1}_0(\Omega)\) and for a.e. \(x\in \Omega\), 
\[
\partial_e G(x)=\partial_e g(x,u(x))+\partial_s g(x,u(x))\partial_e u(x)
=\int_{-\infty}^{u(x)}\partial_e\Phi(x,t)\theta(t)\dt + 
\Phi(x,u(x))\theta(u(x))\partial_e u(x)\,.
\]
Integrating the above identity on \(\Omega\), one gets
\[
\int_{\Omega}\partial_e G(x)\dx =\int_{\Omega}\left(\int_{-\infty}^{u(x)}\partial_e\Phi(x,t)\theta(t)\dt\right)\dx + 
\int_{\Omega}\Phi(x,u(x))\theta(u(x))\partial_e u(x)\dx\,.
\]
The left-hand side vanishes by the Stokes formula. In the right-hand side, we use the Fubini theorem to write the first integral as 
\[
\int_{\R}\theta(t)\left(\int_{[u\geq t]} \partial_e\Phi(x,t)\dx\right)\dt\,.
\]
This completes the proof. 
\end{proof}

Using the above lemma, we can determine the distributional derivative of the function \(v=1_u\).
\begin{lemma}
\label{lm-deriv-v}
Let \(u\in W^{1,1}(\Omega)\) and \(v=1_u\).
The distributional derivative $Dv$ of $v$ is the $\R^{N+1}$-valued measure given by
\begin{equation}
    \label{def-Dv}
    Dv=\frac{(\nabla u,-1)}{\sqrt{1+|\nabla u|^2}} \cH^N\mres \Gu\,.
\end{equation}
Moreover, for every Borel set  \(A\subset \Omega\) such that \(|\nabla u(x)|>0\) for a.e. \(x\in A\), one has:
\begin{equation}
\label{Dv-other-formulation}
Dv\llcorner(A\times \R) = \left(\chi_{A}\frac{(\nabla u,-1)}{|\nabla u|}\cH^{N-1}\llcorner u^{-1}(t)\right)\otimes \cH^1\,.
\end{equation}
\end{lemma}
It follows from \eqref{lm-deriv-v} that the total variation of $D1_u$ is the measure $|D1_u|=\cH^N\mres \Gu$, the Radon--Nikod\'ym derivative of $D1_u$ with respect to its total variation is $\frac{(\nabla u,-1)}{\sqrt{1+|\nabla u|^2}}$.
 The above statement is well-known, in particular in the setting of \(BV\) functions, see \cite[Theorem 4.1.5.2]{GMS-1}. For the convenience of the reader, we provide an elementary proof which does not rely on the theory of BV functions.

\begin{proof}[Proof of Lemma 5.2]
We have first to prove that for every \(\phi=(\phi^1, \dots, \phi^{N+1})\in C^{\infty}_c(\Omega\times \R;\R^{N+1})\) it holds
\begin{equation*}
\int_{\Omega\times \R}1_u(x,t)\,\textrm{div}\,\phi (x,t)\dx\dt = -\int_{\Gu}\left(\sum_{i=1}^{N} \phi^i(x,t) \partial_i u(x) - \phi^{N+1}(x,t)\right)\frac{1}{\sqrt{1+|\nabla u(x)|^2}}\d\cH^{N}(x,t)\,.
\end{equation*}
By applying the formula \eqref{eq:hdH-hdx-area} to the right-hand side, this is equivalent to
\begin{equation}\label{eq588}
\int_{\Omega\times \R}1_u(x,t)\,\textrm{div}\,\phi (x,t)\dx\dt = -\int_{\Omega}\left(\sum_{i=1}^{N} \phi^i(x,u(x)) \partial_i u(x) -  \phi^{N+1}(x,u(x))\right)\dx\,.
\end{equation}

For every \(1\leq i \leq N\), we apply Lemma~\ref{lemma-alt-Stokes-Sobolev} with \(\Phi\coloneqq\phi^i\), \(e\) being the \(i^{\mathrm{th}}\) vector of the canonical basis of \(\R^N\), and \(\theta\equiv 1\). This gives
\[
\int_{\R}\int_{[u\geq t]}\partial_i\phi^i(x,t)\dx\dt=-\int_{\Omega}\phi^i(x,u(x))\partial_i u(x)\dx\,.
\]
Summing over \(i=1,\dots,N\), one gets
\begin{equation}\label{eq600}
\int_{\Omega\times \R}1_u(x,t) \sum_{i=1}^{N} \partial_i\phi^i(x,t)\dx\dt  = -\int_{\Omega} \sum_{i=1}^{N} \phi^i(x,u(x))\partial_i u(x)\dx\,.
\end{equation}
For the last partial derivative \(\partial_t\phi^{N+1}\), we simply use the Fubini theorem to write
\begin{equation}\label{eq604}
\int_{\Omega\times \R}1_u(x,t)  \partial_{t}\phi^{N+1}(x,t)\dx\dt
=\int_{\Omega} \left(\int_{-\infty}^{u(x)}\partial_{t}\phi^{N+1}(x,t)\dt\right)\dx = \int_{\Omega}\phi^{N+1}(x,u(x))\dx\,.
\end{equation}
Adding \eqref{eq600} and \eqref{eq604}, we get the identity \eqref{eq588}.

We proceed with the proof of \eqref{Dv-other-formulation}. Let \(A\subset \Omega\) as in the statement and \(g:A\times \R\to \R\) a bounded Borel map.
By \eqref{def-Dv} and the area formula, it holds
\begin{flalign*}
\int_{A\times \R} g(x,t)\d Dv(x,t)
&=\int_{\Gu \cap (A\times \R)} g(x,t)\frac{(\nabla u,-1)}{\sqrt{1+|\nabla u|^2}}\d\cH^N\\
&=\int_{A}g(x,u(x))(\nabla u(x), -1)\dx.
\end{flalign*}
Applying next the coarea formula, this yields
\begin{flalign*}
\int_{A\times \R} g(x,t)\d Dv(x,t)&= \int_{\R}\int_{A\cap u^{-1}(t)} g(z,t)\frac{(\nabla u(z),-1)}{|\nabla u(z)|}\d\cH^{N-1}(z)\dt \,,
\end{flalign*}
from which \eqref{Dv-other-formulation} follows.
\end{proof}

\subsection{The convex extension \texorpdfstring{$\wh E$ of the energy $E$}{of the energy}}
Given a Carathéodory function $f:\Omega\times\R\times\rn\to \rp$ which is convex with respect to the last variable,
the energy \(E_{\Omega}\) (see \eqref{def-E-of-u}) does not need to be convex on \(W^{1,1}(\Omega)\). Therefore, we associate to \(E_{\Omega}\) a
 new energy $\wh E_{\Omega\times\R}$ defined on the set of those  functions \(v\in L^{\infty}_{loc}(\Omega\times \R)\) such that the distributional derivative \(Dv\) is equal to a finite \(\R^{N+1}\)-valued measure. 
The functional $\wh E_{\Omega\times\R}$ is constructed in such a way that for every \(u\in W^{1,1}(\Omega)\),
\(
E_{\Omega}(u)=\wh E_{\Omega\times\R}(1_u)
\).
As we shall see, $\wh E_{\Omega\times\R}$ is convex, in contrast to \(E_{\Omega}\).

In order to define $\wh E_{\Omega\times\R}$, we  need to introduce the  $\wh{\ }$-operation, which goes back  to \cite{Dal-Maso} in the context of integral representation of {$\Gamma$}-limits of variational integrals. 
For any non-negative convex function $h : \R^{N} \to \rp$, we define its recession function \(h^{\infty}:\rn \to [0,\infty]\) as
\[ 
h^{\infty}(\xi)\coloneqq\lim_{\lambda\to \infty }\tfrac{h( \lambda\xi)}{\lambda}\,. 
\]
Remember that \(h^{\infty}\) is a convex positively one-homogeneous function. 
If  one further assumes that \(h\) is superlinear, then \(h^\infty(\xi)=\infty\) except  when \(\xi=0\) for which  \(h^\infty(0)=0\). 

Using the notation \(q=(q^x,q^t)\) for every \(q\in \R^{N+1}=\R^{N}\times \R\), we denote by $\wh{h}:\R^{N}\times \R \to[0,\infty]$ the function
\begin{equation}\label{def-hat}
    \wh{h}(q)=\wh{h}(q^x,q^t)\coloneqq 
    \begin{cases} -q^t h(-\tfrac{q^x}{q^t})\,, &\text{if $q^t < 0$}\,,\\
    h^{\infty}(q^x)\,, &\text{if $q^t = 0$}\,,\\
    +\infty\,, &\text{if $q^t > 0$}\,.
    \end{cases}
\end{equation}
It follows from the convexity of \(h\) that \(\wh{h}\) is convex on \(\R^{N+1}\) and positively homogeneous of degree \(1\), see e.g. \cite[Lemma~8.1]{Bousquet-Pisa}. In particular, for $h\equiv{1}$, we get for every \(q=(q^x,q^t)\in \R^N\times \R\) with \(q^t\leq 0\),
\begin{equation}\label{eq:def-1-hat}
\wh{1}(q^x, q^t)=|q^t|\,.
\end{equation}

Applying the operation $\ \wh{ }\ $ to \(h=f(x,t,\cdot)\) for a.e. \(x\in \Omega\) and every \(t\in \R\), we obtain a map \(q\mapsto \wh{f(x,t,\cdot)}(q)\). In order to simplify the notation, we write
\[
\wh f(x,t,q)\coloneqq\wh{f(x,t,\cdot)}(q)\qquad\text{for } \quad (x,t,q)\in \Omega\times \R\times \R^{N+1}\,.
\]
Hence, if  \(f(x,t,\cdot)\) is superlinear for a.e. \(x\in \Omega\)  and every \(t\in \R\), then the resulting map  $\wh{f}: \Omega\times \R\times \rn \times \R \to [0,\infty]$ is 
\begin{equation}\label{eq776}
\wh{f}(x,t,q^x,q^t)=
\begin{cases}
-q^t f(x,t,\frac{q^x}{-q^t})\,, & \textrm{ if } q^t<0\,,\\
\infty\,, \quad &\textrm{ if  $q^t> 0$ or $q^t=0$ and $q^x\not=0$}\,,\\
0\,, \quad &\textrm{ if  $(q^x,q^t)=(0,0)$}\,.
\end{cases}
\end{equation}
For every \(v\in L^{\infty}_{loc}(\Omega\times \R)\) such that the distributional derivative \(Dv\) is a finite measure on \(\Omega\times \R\), we define the measure
\[
\wh{f}(x,t,Dv)\coloneqq\wh{f}\left(x,t,\frac{Dv}{|Dv|}\right)|Dv|\,,
\]
where $|Dv|$ is the total variation of $Dv$ and $\frac{Dv}{|Dv|}$ denotes the Radon--Nikod\'ym derivative of $Dv$ with respect to $|Dv|$. For the definition of convex functions of measures, see e.g. \cite[Section 2.6]{AmFuPa}. We then consider the auxiliary energy
\begin{equation}\label{def-hat-E}
    \wh E_{\Omega \times \R} (v) \coloneqq \int_{\Omega \times \R} \dhf(x, t, Dv) 
    = \int_{\Omega \times \R} \wh f\left(x, t, \frac{Dv}{|Dv|}(x, t)\right)\d |Dv|(x, t)\,.
\end{equation} 

Given \(u\in W^{1,1}(\Omega)\) and setting \(v=1_u\), it follows from \eqref{def-Dv} and the homogeneity of \(\wh{f}\) with respect to the last variable  that for every Borel function $h : \Omega \times \R \to \rp$,  
\[
\int_{\Omega \times \R} h(x, t) \dhf(x, t, Dv)
=\int_{\Omega \times \R} h(x, t) \wh{f}(x, t, (\nabla u,-1))\frac{1}{\sqrt{1+|\nabla u|^2}}\d\cH^{N}\mres \Gu\,.
\]
The definition of \(\wh{f}\) given in \eqref{eq776} and the area formula \eqref{eq:hdH-hdx-area} then yield
\begin{equation}\label{eq:f-f-hat-general}
    \int_{\Omega \times \R} h(x, t) \dhf(x, t, Dv) = \int_{\Omega} h(x, u(x))f(x, u(x), \nabla u(x))\dx\,.
\end{equation}
In particular\,, taking \(h\equiv 1\), we obtain that $E_\Omega(u) = \wh E_{\Omega \times \R} (1_u)$ for every $u \in W^{1,1}(\Omega)$, where $\wh E_{\Omega \times \R} $ is given by~\eqref{def-hat-E}.

\subsection{Construction of the inner approximating sequence  in a reduced setting}\label{sec:approx}
In this section, we work under seemingly more restrictive assumptions than in the main results.  By now we consider an open set \(\Lambda\) in \(\rn\) such that \(\Omega\Subset \Lambda\), while the Lagrangian satisfies the following, technical version of our main structural assumption.
\begin{description}
\item[\namedlabel{f-red}{($f^{\rm red}$)}]  Assume  that $f:\Lambda\times\R\times\rn\to\rp$ is a Carath\'eodory function which is convex and superlinear with respect to the last variable and $f(x,t,0)=0$ for a.e. $x\in\Lambda$ and all $t\in\R$.\end{description}
We explain in Section~\ref{sec:proofs} how the extra assumptions can be removed. Our goal in the reduced setting reads as follows.

\begin{proposition}[Absence of the Lavrentiev phenomenon in the reduced setting] \label{prop-main-reduced}
Let $\Lambda$ be an open set on $\rn$, let  \(\Omega\Subset\Lambda\) be a bounded Lipschitz open set, \(\vp:\rn\to \R\) be Lipschitz continuous,  \(p\geq 1\), and \(f:\Lambda \times \R\times \rn\to \rp\) satisfy Assumptions \ref{f-red} and \eqref{eqHZconv}. Consider the functional $E_\Omega$ be defined as in~\eqref{def-E-of-u}. Then, for every \(u\in \cEp(\Omega)\cap W^{1,p}(\Omega)\cap L^\infty(\Omega)\) there exists a sequence \((\wt u_{n})_{n\in\N}\subset W^{1,\infty}_\vp(\Omega)\) such that $\wt u_n \to u$   as $n \to \infty$ in $L^{1}(\Omega)$, $\sup_n\|\wt u_n\|_{L^\infty}<\infty$ and
\begin{equation*} 
\lim_{n\to \infty}E_\Omega(\wt u_n)=E_\Omega(u) \,.
\end{equation*}
\end{proposition} 
We emphasize that in the above statement, the assumption \eqref{eqHZconv} is required to hold on \(\Lambda\) and not just on \(\Omega\).

Let us present the construction of the initial approximation and establish its basic properties. We start from a map \(u\in\cE(\Lambda)\cap W^{1,p}(\Lambda)\cap L^{\infty}(\Lambda)\) for some \(p\geq 1\). 
 Let \begin{equation}
     \label{def-M}
 M\coloneqq\|u\|_{L^{\infty}(\Lambda)}\,.
 \end{equation}
We introduce an absolutely continuous positive function \begin{equation}
    \label{def:alpha}
\alpha:\R\to (0,\infty)
\end{equation} such that \(\lim_{t\to +\infty}\alpha(t)=0\) and for every \(\mathcal{k}_1>0\), there exists \(c_{\mathcal{k}_1}>0\)  such that 
\begin{equation}\label{alpha'prop}
    \alpha'(t)\leq -c_{\mathcal{k}_1}\qquad\text{for a.e. \(t\in (-\mathcal{k}_1,\mathcal{k}_1)\)}\,.
\end{equation}
 Ultimately, \(\alpha\) will be subject to an additional condition that we proceed to formulate. Remember first that \(\Lambda_+\) is a Borel subset of \(\Lambda\) such that \(|\nabla u|>0\) a.e. on \(\Lambda_+\) and \(|\nabla u|=0\) a.e. on \(\Lambda\setminus \Lambda_+\).  We then observe  that for a.e. \(t\in \R\), for \(\cH^{N-1}\) a.e. \(z\in u^{-1}(t)\), we have \(\nabla u(z)\in \Lambda_+\). This is a consequence of the coarea formula, which also implies that
 \[
\int_{\R}\left(\int_{u^{-1}(t)} \frac{f(z,t, \nabla u(z))+|\nabla u(z)|^p}{|\nabla u(z)|}\d\cH^{N-1}(z) \right)\dt =\int_{\Lambda_+}f(x,u(x), \nabla u(x) )+|\nabla u(x)|^p\dx\,. 
 \]  
We deduce therefrom that the inner integral in the left-hand side is a summable function of \(t\in \R\) and we are thus entitled to require that
 for a.e. \(t\in \R\), 
\begin{equation}\label{eq1168}
-\alpha'(t)\geq \int_{u^{-1}(t)} \frac{f(z,u(z), \nabla u(z))+|\nabla u(z)|^p}{|\nabla u(z)|}\d\cH^{N-1}(z)\,.
\end{equation}

In order to construct the approximate sequence, we employ the function \(v={1}_u\) defined in~\eqref{1u} and then, for every \(\delta \geq 0\), we set \(v_{0,\delta}:\Lambda\times \R\to \R\) by \begin{equation}
    \label{v0delta}
v_{0, \delta}(x,t)\coloneqq v(x,t)+\delta\alpha(t)\,.
\end{equation}
 Recall also $\vr_\ve$ from~\eqref{def-vr-ve} and note that 
 since $\Omega\Subset\Lambda$, one can find \(\ve_0>0\) such that \(\Omega+B{(0,\ve_0)}\subset \Lambda\). For every \(\ve \in (0,\ve_0)\),   we  consider the map \(v_{\ve,\delta}:\Omega\times \R\to \R\) defined as
\begin{equation}
    \label{def-v-eps-delt}
v_{\ve,\delta}(x,t)\coloneqq(v_{0,\delta}*_x\vr_\ve)(x,t)=(v*_x\vr_\ve)(x,t) + \delta \alpha(t)\,.
\end{equation} 
Note that if \(\delta\in (0,\infty)\) and $\ve\in [0, \ve_0)$, then $v_{\ve, \delta}$ is decreasing with respect to the second variable. We shall denote by $q_{\ve, \delta}$ the Radon--Nikod\'ym derivative of $Dv_{\ve, \delta}$ with respect to its total variation, i.e., $q_{\ve, \delta} = (q_{\ve, \delta}^x, q_{\ve, \delta}^t)$ and
\begin{equation}\label{eq:def-q}
q_{\ve, \delta} \coloneqq\frac{Dv_{\ve, \delta}}{|Dv_{\ve, \delta}|} \,.
\end{equation}
By standard properties of convolution, the family \((v_{\varepsilon, \delta})_{\varepsilon}\) converges to \(v_{0,\delta}\) in \(L^{1}_{\rm loc}(\Omega\times \R)\) when \(\ve\to 0\).  Moreover,  \(Dv_{\ve,\delta}=Dv_{0,\delta}*_x\vr_\ve\) for every \(\ve\in (0,\ve_0), \delta\in (0,\infty)\), and the family \((Dv_{\ve,\delta})_\ve\) converges weakly-$*$ to \(Dv_{0,\delta}\)  when \(\ve\to 0\). 
From  \cite[Corollary~1.60]{AmFuPa}, this implies that
\begin{equation}\label{eq1647}
\liminf_{\ve\to 0} |Dv_{\ve, \delta}|(\Omega\times \R)\geq |Dv_{0,\delta}|(\Omega\times \R)\,.
\end{equation}
Actually, the converse inequality also holds.

\begin{lemma}\label{lm-lim-DveOm}
Let $u \in \cE(\Lambda) \cap L^{\infty}(\Lambda)$, $\delta\in [0,\infty)$, \(\ve\in [0,\ve_0)\), and $v_{\ve, \delta}$ be defined in~\eqref{def-v-eps-delt}. For every Borel set \(A\subset \Omega\), 
\begin{equation}\label{eq-estim-sup-dved}
|Dv_{\ve,\delta}|(A\times \R)\leq 
\int_{\Lambda\times \R}\left( \int_{A}\vr_{\varepsilon}(y-x)\dy\right) \d|Dv_{0,\delta}|(x,t)\,,
\end{equation}
where the inner integral in the right-hand side reduces to \(\chi_A(x)\) when \(\ve=0\). Moreover,
\begin{equation}\label{eq-cvgce-dved}
\limsup_{\ve\to 0} |Dv_{\ve, \delta}|(A\times \R)\leq |Dv_{0,\delta}|(A\times \R)=\int_{A}\sqrt{1+|\nabla u|^{2}}\dx + \delta {|A|}\int_{\R}|\alpha'|\dt\,.    
\end{equation}
\end{lemma}
\begin{proof}
The measures \(\cH^{N}\llcorner \Gu\) and \(\cH^{N+1}\) are mutually singular.
By the definition of \(v_{0,\delta}\) and Lemma~\ref{lm-deriv-v}, we thus have
\[
|Dv_{0,\delta}|(A \times \R)=|Dv|(A\times \R)+\delta |\alpha'|{\cH^{N+1}}(A\times \R)=\cH^{N}(\Gu \cap (A\times \R))+\delta\int_{A\times \R}|\alpha'|\dx \dt\,.
\]
Using the area formula for the first term and the fact that \(|\alpha'|\) depends only on \(t\) for the second one, this gives the second equality in \eqref{eq-cvgce-dved}; that is,
\begin{equation}\label{eq1604}
|Dv_{0,\delta}|(A\times \R)=\int_{A}\sqrt{1+|\nabla u|^{2}}\dx + \delta {|A|}\int_{\R}|\alpha'|\dt\,.    
\end{equation}
We deduce therefrom that
\begin{equation}\label{eq1641}
|A|=0\quad\Longrightarrow\quad |Dv_{0,\delta}|(A\times \R)=0.    
\end{equation}

For the first inequality in \eqref{eq-cvgce-dved}, we start from the fact that \(|Dv_{\ve, \delta}|\leq |Dv_{0, \delta}|*_x\varrho_{\varepsilon}\) for every \(\ve\in (0, \ve_0)\).
Together with \eqref{eq-def-conv-partial-bis}, this implies \eqref{eq-estim-sup-dved} from which we deduce that
\begin{equation}\label{eq1652}
\limsup_{\ve\to 0} |Dv_{\ve, \delta}|(A\times \R)
\leq \limsup_{\ve\to 0}
\int_{\Lambda\times \R}\left( \int_{A}\vr_{\varepsilon}(y-x)\dy\right) \d|Dv_{0,\delta}|(x,t)\,.
\end{equation}
The inner integral in the right-hand side is equal to \(\chi_A*\widetilde{\vr}_{\ve}(x)\) where \(\widetilde{\vr}(y)=\vr(-y)\), for every \(y\in \rn\). We observe that   \(\chi_A*\widetilde{\vr}_{\ve}(x)\) is bounded from above by \(1\) and converges to \(\chi_A(x)\) for every Lebesgue point \(x\) of \(\chi_A\). From \eqref{eq1641}, we deduce that for \(|Dv_{0,\delta}|\) a.e. \((x,t)\in \Omega\times \R\), the point \(x\) is a Lebesgue point of \(\chi_A\) and thus
\[
\lim_{\ve\to 0}\chi_A*\widetilde{\vr}_{\ve}(x)=\chi_A(x).
\]
Moreover, applying \eqref{eq1604} with \(A=\Omega\), we get  that \(|Dv_{0,\delta}|(\Omega\times  \R)<\infty\). We can thus apply the dominated convergence theorem in the right-hand side of \eqref{eq1652} to obtain 
the first inequality in \eqref{eq-cvgce-dved}.

\end{proof}
From \eqref{eq1647} and \eqref{eq-cvgce-dved} with \(A=\Omega\), it follows that
\begin{equation}\label{eq16755}
\lim_{\ve\to 0} |Dv_{\ve, \delta}|(\Omega\times \R)= |Dv_{0,\delta}|(\Omega\times \R).   
\end{equation}

As a by-product of Lemma~\ref{lm-lim-DveOm} (see \eqref{eq1641} for \(\ve=0\) and \eqref{eq-estim-sup-dved} for \(\ve>0\)), we also deduce what follows.
\begin{remark}\label{lemma-Negligibility}
For every Borel set \(A\subset \Omega\) such that \({|A|}=0\), one has \(|Dv_{\ve, \delta}|(A\times \R)=0\) for  \(\ve\in [0, \ve_0),\, \delta \in \rp\).    
\end{remark}

The fact that \(f(x,t,0)=0\) in~\ref{f-red} entitles one to restrict the measure \(\widehat{f}(x,t,Dv_{\ve,\delta})\) to suitable subsets of \(\Omega\times \R\):

\begin{lemma}
\label{lm-whf-dv0d-dv}
Suppose  $f$ satisfies~\ref{f-red}.
 Let $u \in \cE(\Lambda) \cap L^{\infty}(\Lambda), v_{\ve, \delta}$ defined as in~\eqref{def-v-eps-delt},  $M$ given by~\eqref{def-M} and let us recall the decomposition $\Lambda = \Lambda_+\cup \Lambda_0$ from~\eqref{def-Lambda+}.   Then, for every \(\delta \in \rp\),
 \begin{equation}\label{eq1717prim}
\widehat{f}(\cdot,t,Dv_{0,\delta})= \widehat{f}(\cdot,t, Dv)=\widehat{f}(\cdot,t,Dv_{0,\delta}\llcorner (\Lambda\times [-M,M])).
 \end{equation}
 Moreover,
 \begin{equation}
 \label{eq-whf-DV-ell}
     \widehat{f}(\cdot,t,Dv)= \left(\frac{1}{|\nabla u|}f(\cdot,t,\nabla u)\left(\cH^{N-1}\llcorner (u^{-1}(t)\cap \Lambda_+)\right)\right)\otimes {\cH^1}.
 \end{equation}
  Finally, for every \(\ve\in (0, \ve_0)\),
  \begin{equation}\label{eq1717}
 \widehat{f}(\cdot,t,Dv_{\ve,\delta}) = \widehat{f}(\cdot,t, Dv_{\ve,\delta}\llcorner (\Omega\times [-M,M]))\,.
 \end{equation}
 
\end{lemma}
\begin{proof}
Since \(Dv_{0,\delta}\) is the sum of the two mutually singular measures \(Dv\) and \((0, \delta \alpha'){\cH^{N+1}}\), one has by \cite[Proposition~2.37]{AmFuPa},
\[
\widehat{f}(\cdot,t,Dv_{0,\delta})=\widehat{f}(\cdot, t, Dv) +\widehat{f}(\cdot, t, (0,\delta \alpha')){\cH^{N+1}}\,.
\]
Using that \(f(x,t,0)=0\) for a.e. \(x\in \Lambda\) and every \(t\in \R\), we get 
\begin{equation}\label{eq1333}
\forall q^t\in (-\infty,0] \qquad \widehat{f}(x, t, (0,q^t)))=0\,.
\end{equation}
In view of the fact that \(\alpha'(t)<0\) for a.e. \(t\in \R\), this implies that \(\widehat{f}(x,t,(0, \delta \alpha'(t)))=0\) for \(\cH^{N+1}\)-a.e. \((x,t)\in \Lambda\times \R\) and the first equality in \eqref{eq1717prim} follows. The second one is a consequence of the fact that \(Dv=Dv\llcorner (\Lambda\times [-M,M])\), which in turn follows from the fact that \(\Gu\subset \Lambda\times [-M,M]\).

We proceed with the proof of ~\eqref{eq-whf-DV-ell}. Using the notation introduced in~\eqref{eq:def-q}, one has for \(|Dv|\)-a.e.  \((x,t)\in \Lambda\times \R\), 
\[q_{0,0}(x,t)=\frac{Dv}{|Dv|}(x,t)=\frac{(\nabla u(x) ,-1)}{\sqrt{1+|\nabla u|^2}}\,.
\]
Since \(\nabla u(x)=0\) for a.e. \(x\in \Lambda_0\), we have \(q_{0,0}^{x}(x,t)=0\) \(|Dv|\)-a.e. on \(\Lambda_0\times \R\) by Remark~\ref{lemma-Negligibility}, which also implies by \eqref{eq1333} that \(\wh{f}(x,t,q_{(0,0)}(x,t))=0\) for \(|Dv|\)-a.e. \((x,t)\in \Lambda_0\times \R\). 
Hence, 
\[
\widehat{f}(\cdot,t,Dv)=\widehat{f}(\cdot, t, Dv\llcorner (\Lambda_+\times \R))\,.
\]
Inserting \eqref{Dv-other-formulation} (with \(\Lambda\) instead of \(\Omega\) and \(A=\Lambda_+\)) in the above identity,  we get
\[
\widehat{f}(\cdot,t,Dv)=\widehat{f}\left(\cdot, t, \chi_{\Lambda_+}\frac{(\nabla u,-1)}{|\nabla u|}\right) \left(\cH^{N-1}\llcorner u^{-1}(t)\right)\otimes {\cH^1}\,,
\]
and \eqref{eq-whf-DV-ell} then follows from the definition of \(\widehat{f}\) in terms of \(f\).

By definition of partial convolution, for every \(\ve\in (0,\ve_0)\) and  every Borel \(A'\subset \Omega\times \R\),
\[
Dv_{\ve,0}(A')=\int_{\Lambda\times \R}\left( \int_{\rn}\chi_{A'}(y,t)\vr_{\ve}(y-x)\dy\right)\d Dv(x,t)=\int_{\Lambda\times [-M,M]}\left( \int_{\rn}\chi_{A'}(y,t)\vr_{\ve}(y-x)\dy\right)\d Dv(x,t)\,.
\]
This implies ~\eqref{eq1717}.

\end{proof}

As a consequence of the above lemma, we can deduce that \(|Dv_{\ve,\delta}|\)-negligible sets are also \(\widehat{f}(\cdot,t,Dv)*_x\vr_{\ve}\)-negligible.

\begin{lemma}
\label{lm-abs-convol-f-ch}
Suppose  $f$ satisfies~\ref{f-red}.
 Let $u \in \cE(\Lambda) \cap L^{\infty}(\Lambda)$ and let us recall $v_{\ve, \delta}$ from~\eqref{def-v-eps-delt}. For every \(\delta\in (0,\infty), \ve \in (0,\ve_0)\), {and every Borel set $A' \subset \Omega \times \R$} it holds that
\begin{equation}\label{whf-convol-other-formulation}
\widehat{f}(\cdot,t,Dv)*_x\vr_{\ve}{(A')}={\int_{A'}}\int_{\Lambda_+\cap u^{-1}(t)}\frac{1}{|\nabla u(x)|}f(x,u(x),\nabla u(x))\vr_\ve(y-x)\d\cH^{N-1}(x) {\d \cH^{N+1}(y, t)}\,.    
\end{equation}
 Moreover, the measure \(\widehat{f}(\cdot,t,Dv)*_x\vr_{\ve}\) is absolutely continuous with respect to \(|Dv_{\ve,\delta}|\).
\end{lemma}
\begin{proof}
From the definition of partial convolution, 
for every Borel set  \(A'\subset \Omega\times\R\),
\[
\widehat{f}(\cdot,t,Dv)*_x\vr_\ve(A')=\int_{\Lambda\times \R}\left(\int_{\rn}\chi_{A'}(y,t)\vr_{\ve}(y-x)\dy \right)\d \widehat{f}(x,t,Dv).
\]
By \eqref{eq-whf-DV-ell}, this gives
\[
\widehat{f}(\cdot,t,Dv)*_x\vr_\ve(A')= \int_{\R} \int_{\Lambda_+\cap u^{-1}(t)} \frac{1}{|\nabla u(x)|}f(x,u(x), \nabla u(x))\left(\int_{\Omega}\chi_{A'}(y,t)\vr_{\ve}(y-x)\dy\right)\d\cH^{N-1}(x)\dt
\]
and \eqref{whf-convol-other-formulation} then follows from the Fubini theorem. 

To prove the last assertion, let \(A'\subset \Omega\times \R\) be a Borel subset such that \(|Dv_{\ve, \delta}|(A')=0\). Then, one exploits the fact that \(|\partial_t v_{\ve, \delta}|\leq |Dv_{\ve, \delta}|\) to deduce that \(|\partial_t v_{\ve, \delta}|(A')=0\). We then observe that \(\partial_t v_{\ve, \delta}\) is the sum of two non-positive measures:
\[
\partial_tv_{\ve, \delta} = \partial_t v*_x\vr_\ve + \delta \alpha' {\cH^{N+1}}\,.
\]
This implies that 
\[
\int_{A'}\delta \alpha'(t)\dx\dt=0\,.
\]
Since \(\alpha'(t)<0\) for a.e.  \(t\in \R\), it follows that \({|A'|}=0\), {which by~\eqref{whf-convol-other-formulation} implies $\widehat{f}(\cdot,t,Dv)*_x\vr_\ve(A') = 0$, as desired.}
\end{proof}

\section{The liminf estimate}\label{sec:liminf}

The proofs of the liminf and the limsup estimates strongly rely on the Reschetnyak continuity theorem, which requires boundedness and continuity properties of the integrand.  Let us notice that by Lemma~\ref{lem-bdd-f} any function satisfying~\eqref{eqHZconv} is bounded on bounded sets. We exploit this fact in the following lemma. 

\begin{lemma}\label{lm-Reschetnyak}
 Suppose  $f$ satisfies~\ref{f-red}.
 Let $u \in \cE(\Lambda) \cap L^{\infty}(\Lambda)$ and let us recall $v_{\ve, \delta}$ and $q_{\ve, \delta}$ given by~\eqref{def-v-eps-delt} and~\eqref{eq:def-q}. Let \(\theta:\R\to [0,1]\) be a continuous function which vanishes on \([c,\infty)\) for some \(c<0\). Then, for every \(\delta\in (0,\infty)\), it holds
 \[
 \lim_{\ve\to 0}\int_{\Omega\times \R}\theta(q^{t}_{\ve,\delta})\d \widehat{f}(x,t,Dv_{\ve,\delta})= \int_{\Omega\times \R}\theta(q^{t}_{0,\delta})\d \widehat{f}(x,t,Dv_{0,\delta}).
 \]
\end{lemma}
\begin{proof}
Given \(\tau >0\), there exists \(\iota_\tau>0\) such that for every Borel set \(A\subset \Lambda\),
\begin{equation}\label{eq-1706}
 {|A|}\leq \iota_\tau \Longrightarrow \int_{A}\sqrt{1+|\nabla u|^2}\dx + \delta  {|A|}\int_{\R}|\alpha'|\dt    \leq \tau.
\end{equation}
Let \(\eta:\R\to [0,1]\) be a continuous compactly supported function such that \(\eta\equiv 1\) on \([-M,M]\). Let us define the function $h:  \Omega\times \R\times \mathbb{S}^N \to \R$ via
\[
h(x,t,(q^x,q^t))\coloneqq \eta(t)\theta(q^t) \widehat{f}(x,t,(q^x,q^t))\,.
\]
Observe that \(h(x,t,(q^x,q^t))=0\) when \(q^t\geq c\) while for \(q^t\leq c\),
\[
h(x,t,(q^x,q^t))=-q^t\eta(t)\theta(q^t)f\left(x,t,\frac{q^x}{-q^t}\right).
\]
We deduce from Lemma~\ref{lem-bdd-f} that \(h\) is bounded on \(\Omega\times \R\times \mathbb{S}^N\). Moreover, for a.e. \(x\in \Omega\), the map \((t,q)\mapsto h(x,t,q)\) is uniformly continuous. By the Scorza--Dragoni theorem, there exists a compact set \(K_\tau \subset \Omega\) such that \({|\Omega\setminus K_\tau|}\leq \iota_\tau\) and \(h|_{K_\tau\times \R\times \mathbb{S}^N}\)  is continuous. By the Tietze theorem, there exists a continuous extension \(\overline{h}:\Omega\times \R\times \mathbb{S}^N\to \rp\) such that \(\|\overline{h}\|_{L^{\infty}}\leq \|h\|_{L^{\infty}}\).

Since \(Dv_{\ve,\delta}\)  weakly-$*$ converges to \(Dv_{0,\delta}\) when \(\ve\to 0\) and using also  \eqref{eq16755}, we can rely on the Reschetnyak continuity theorem (see e.g. \cite[Theorem~2.39]{AmFuPa}) to obtain
\[
\lim_{\ve\to 0}\int_{\Omega\times \R}\overline{h}\left(x,t,\frac{Dv_{\ve,\delta}}{|Dv_{\ve,\delta}|}\right)\d |Dv_{\ve,\delta}|=\int_{\Omega\times \R}\overline{h}\left(x,t,\frac{Dv_{0,\delta}}{|Dv_{0,\delta}|}\right)\d |Dv_{0,\delta}|\,.
\]
Since \(h\) and \(\overline{h}\) only differs on \((\Omega\setminus K_\tau)\times \R\times \mathbb{S}^N\), we obtain that:
\begin{multline*}
\limsup_{\ve\to 0}\left|\int_{\Omega\times \R}h\left(x,t,\frac{Dv_{\ve,\delta}}{|Dv_{\ve,\delta}|}\right)\d |Dv_{\ve,\delta}|-\int_{\Omega\times \R}h\left(x,t,\frac{Dv_{0,\delta}}{|Dv_{0,\delta}|}\right)\d |Dv_{0,\delta}| \right|\\
\leq (\|h\|_{L^{\infty}}+\|\overline{h}\|_{L^{\infty}}) \limsup_{\ve\to 0}\left(|Dv_{\ve,\delta}|+|Dv_{0,\delta}|\right)\left((\Omega\setminus K_\tau)\times \R\right).
\end{multline*}
Using that \(\|\overline{h}\|_{L^{\infty}}\leq \|h\|_{L^{\infty}}\) and also Lemma~\ref{lm-lim-DveOm}, this gives
\begin{align*}
\limsup_{\ve\to 0}\left|\int_{\Omega\times \R}h\left(x,t,\frac{Dv_{\ve,\delta}}{|Dv_{\ve,\delta}|}\right)\d |Dv_{\ve,\delta}|-\right.&\left.\int_{\Omega\times \R}h\left(x,t,\frac{Dv_{0,\delta}}{|Dv_{0,\delta}|}\right)\d |Dv_{0,\delta}| \right|
\leq 4\|h\|_{L^{\infty}}|Dv_{0,\delta}|((\Omega\setminus K_\tau)\times \R)\\
&= 4\|h\|_{L^{\infty}}\left(\int_{\Omega\setminus K_\tau}\sqrt{1+|\nabla u|^2}\dx + \delta {|\Omega\setminus K_\tau|}\int_{\R}|\alpha'|\dt\right)\leq 4\|h\|_{L^{\infty}}\tau\,,
\end{align*}
where the last inequality follows from \eqref{eq-1706} applied to \(A=\Omega\setminus K_\tau\).
In the left-hand side, we rely on the definition of \(h\), the fact that \(\eta\equiv 1\) on \([-M,M]\) as well as Lemma~\ref{lm-whf-dv0d-dv} to obtain:
\[
\limsup_{\ve\to 0}\left|\int_{\Omega\times \R}\theta(q^{t}_{\ve,\delta})\d \widehat{f}(x,t,Dv_{\ve,\delta})-\int_{\Omega\times \R}\theta(q^{t}_{0,\delta})\d \widehat{f}(x,t,Dv_{0,\delta}) \right|
\leq 4\|h\|_{L^{\infty}}\tau.
\]
Since \(\tau\) is arbitrary, this implies the desired result.
\end{proof}

We can easily derive from Lemma~\ref{lm-Reschetnyak} the liminf estimate for \(\widehat{E}_{\Omega\times \R}\).

\begin{proposition}[The liminf estimate]\label{prop-liminf} Suppose  $f$ satisfies~\ref{f-red}. Let  $v_{\ve, \delta}$ be given by~\eqref{def-v-eps-delt} for $u\in \cE(\Lambda)\cap L^\infty(\Lambda)$.Then for every \(\delta>0\), it holds
\begin{equation}\label{eq1362}
\liminf_{\ve\to 0} \int_{{\Omega}\times \R}\dhf(x,t,Dv_{\ve,\delta}(x,t))  \geq \int_{{\Omega}\times \R}\dhf(x,t,Dv(x,t))\, .
\end{equation}
\end{proposition}
\begin{proof}
By Lemma~\ref{lm-Reschetnyak} with \(\theta=\theta_k\)  (given by~\eqref{theta_k}) and the fact that \(\theta_k\leq 1\), we have:
\[
\liminf_{\ve\to 0}\int_{\Omega\times \R}\d \widehat{f}(x,t,Dv_{\ve,\delta})\geq  \int_{\Omega\times \R}\theta_k(q^{t}_{0,\delta})\d \widehat{f}(x,t,Dv_{0,\delta})\,.
\]
Using that \(q^{t}_{0,\delta}(x,t)<0\) for \(|Dv_{0,\delta}|\) a.e. \((x,t)\in \Omega\times \R\), we deduce from the dominated  convergence theorem 
that the right-hand side  converges, as $k \to \infty$, to \(\int_{\Omega\times \R}\d \widehat{f}(x,t,Dv_{0,\delta})\), which in turn is equal to the right-hand side of~\eqref{eq1362} by Lemma~\ref{lm-whf-dv0d-dv}.
\end{proof}

The corresponding limsup estimate is much more delicate to prove and is the object of the next section.

\section{The limsup estimate}\label{sec:limsup-est}
We rely on the notation introduced in Section~\ref{sec:approx}.
The goal of this section is the following limsup estimate, which will be proven in Section~\ref{ssec:limsup-proof}. With this aim, Sections~\ref{ssec:Zh-for-hat-f}-\ref{ssec:key-est} are focused on translating  the anti-jump conditions to the  convexified energy.

\begin{proposition}[The limsup estimate]\label{prop-limsup} Suppose  $f$ satisfies \ref{f-red}. Let  $v_{\ve, \delta}$ be given by~\eqref{def-v-eps-delt} for $u \in \cE(\Lambda)\cap L^\infty(\Lambda)$ and $\alpha$ satisfying \eqref{alpha'prop} and \eqref{eq1168}.  Then for every  $\delta>0$, it holds
\[
\limsup_{\ve\to 0} \int_{{\Omega}\times \R}\dhf(x,t, Dv_{\ve,\delta}) \leq \int_{{\Omega}\times \R}\dhf(x,t,Dv)\,.
\]    
\end{proposition} 

\subsection{\texorpdfstring{Consequences for \(\wh{f}\) of~\eqref{eqHZconv}   imposed on $f$}{Consequences of the balance condition for the extension}}\label{ssec:Zh-for-hat-f}

In this section, we formulate the natural implication of the anti-jump condition \eqref{eqHZconv} for functions  governing $\wh E$. We denote by \(F^\ve: \Lambda\times \R\times \R^{N+1}\to [0,\infty]\) the map
\begin{equation}\label{eq:def-Feps}
F^\ve(x,t,q)\coloneqq \widehat{\bigg(\big(f^{-}_{B(x,\ve)}(t,\cdot)\big)^{**}\bigg)}(q)\,.
\end{equation}
Since \(f(x,t,0)=0\) for a.e. \(x\in \Lambda\) and every \(t\in \R\), one has \(f_{B(x,\ve)}^-(t,0)=0\) for every \((x,t)\in \Lambda\times \R\). Hence, using that  \(0\leq (f_{B(x,\ve)}^{-}(t,\cdot))^{**}\leq f_{B(x,\ve)}^-\), we deduce that \( (f_{B(x,\ve)}^{-}(t,\cdot))^{**}(0)=0\) and thus,
\begin{equation}\label{eq1404}
\forall (x,t,q^t)\in \Lambda\times \R\times (-\infty,0]\quad \text{it holds}  \quad F^{\ve}(x,t,(0,q^t))=0\,.
\end{equation}
Since \(f\) is superlinear in the sense of \eqref{eq-superlinearity-sv}, the map \((x,t,\xi)\mapsto f_{B(x, \ve)}^{-}(t,\xi)\) is superlinear as well, and so is the map \((x,t,\xi)\mapsto  (f^{-}_{B(x,\ve)}(t,\cdot))^{**}(\xi)\). We deduce therefrom that
\begin{equation}\label{eq1438}
\forall (q^x,q^t)\in (\R^N\times [0,\infty))\setminus \{(0,0)\}\quad \text{it holds}  \quad  F^{\ve}(x,t,(q^x,q^t))=+\infty\,.
\end{equation}

We will show that if $f$ satisfies condition \eqref{eqHZconv}, then $F^\ve$ satisfies that for every \(\mathcal{k}=(\mathcal{k}_1, \mathcal{k}_2)\in (0,\infty)^2\), for a.e. \(x\in \Lambda\), for every \(t\in [-\mathcal{k}_1, \mathcal{k}_1]\), for every \(q\in \mathbb{S}^N\) with \(q^t<0\),
\begin{equation}\label{eqHZ-hat}
\left(\frac{|q^x|}{|q^t|}\right)^{\max(p,N)} + \frac{1}{|q^t|}F^{\ve}(x,t,q)   \leq \frac{\mathcal{k}_2}{\ve^{N}} \Longrightarrow \wh{f}(x,t,q)\leq \cC_{\mathcal{k}}(F^\ve(x,t,q)+ |q^t|)\,, 
\end{equation}
where  \(\cC_{\mathcal{k}}>0\) is the constant in \eqref{eqHZconv}, which depends only on $\mathcal{k}$.

\begin{lemma}
    \label{lm-from-H_L-to-hat}If $f$ satisfies Assumptions~\ref{f-red} and \eqref{eqHZconv} on $\Lambda$, and  $F^\ve$ is given by~\eqref{eq:def-Feps}, then  \eqref{eqHZ-hat} holds.
\end{lemma}
\begin{proof}
Due to  \eqref{eqHZconv}  with \(\xi=\frac{q^x}{|q^t|}\), for a.e. \(x\in \Lambda\), for every \(t\in (-\mathcal{k}_1, \mathcal{k}_1)\), for every \(q\in \mathbb{S}^N\) with \(q^t\not=0\), 
it holds 
\begin{equation*}
\left(\frac{|q^x|}{|q^t|}\right)^{\max(p,N)} + {\left(f^{-}_{B(x,\ve)}(t,\cdot)\right)^{**}\left(\tfrac{q^x}{|q^t|}\right)
\leq \frac{\mathcal{k}_2}{\ve^{N}}}\qquad \Longrightarrow\qquad |q^t|f\left(x,t, \tfrac{q^x}{|q^t|}\right)\leq
\cC_{\mathcal{k}} |q^t|\left[\left(f^{-}_{B(x,\ve)}(t,\cdot)\right)^{**}\left(\tfrac{q^x}{|q^t|}\right)+1\right]\,,
\end{equation*} 
which implies the desired result when \(q^t<0\). For \(q^t>0\) or \(q^t=0\) and \(q^x\not=0\), the right-hand side of~\eqref{eqHZ-hat} is \(+\infty\). Finally, when \((q^x,q^t)=(0,0)\), both sides of~\eqref{eqHZ-hat} vanish. 
\end{proof}

A substantial part of the proof of Proposition~\ref{prop-limsup} amounts to proving that the map \(q_{\ve,\delta}=(q_{\ve,\delta}^x, q_{\ve,\delta}^t)\) defined in \eqref{eq:def-q} satisfies the estimate required in the left-hand side of \eqref{eqHZ-hat}. A key ingredient is the Jensen inequality, that we apply here to functionals defined on subsets of measures.

\begin{lemma}\label{lm-estimge} Let  $f$ satisfy~\ref{f-red}. Let $F^\ve$ be given by~\eqref{eq:def-Feps} and let $v_{\ve, \delta}$ be given by~\eqref{def-v-eps-delt} with $u \in \cE(\Lambda)$. Then for  every \(\delta\in (0,\infty)\), and \(\ve\in (0, \ve_0)\), for every \(|Dv_{\ve,\delta}|\)-measurable set \(A'\subset \Omega\times \R\),
\[
\int_{A'}\d F^{\ve}(x,t,Dv_{\ve, \delta})\leq \int_{A'}\d 
\left(\wh{f}(\cdot, t, Dv)*_x\vr_{\ve}\right)\,.
\]
\end{lemma}
\begin{proof}
Observe that the measure \(\wh{f}(\cdot, t, Dv)*_x\vr_{\ve}\), which is absolutely continuous with respect to \(|Dv_{\ve,\delta}|\) by Lemma~\ref{lm-abs-convol-f-ch}, has a natural extension to the \(\sigma\)-algebra of \(|Dv_{\ve,\delta}|\)-measurable sets, so the above inequality makes sense. 

We first disintegrate the measure $Dv$ by writing 
$Dv = \mv_t \otimes \mw$, where \(\mw\) is a positive finite measure on \(\R\) and \(\mv_t\) is an \(\R^{N+1}\)-valued  measure on \(\Lambda\) with \(|\mv_t(\Lambda)|=1\).   Then, 
$$
Dv_{\ve, \delta} = \left((\mv_t * \vr_{\ve})\cH^N) \otimes \mw+(0, \delta \alpha'(t)\right)\dxodt\,.
$$
By subadditivity and homogeneity of $F^{\ve}$, we have
\begin{equation}\label{eq1446}
F^{\ve}(\cdot,t,Dv_{\ve, \delta}) \leq 
(F^{\ve}\left(\cdot,t,\mv_t * \vr_{\ve})\cH^N\right)\otimes \mw + \delta F^{\ve}\left(\cdot,t,(0, \alpha'(t))\right)\dxodt =\left(F^{\ve}(\cdot,t,\mv_t * \vr_{\ve})\cH^N\right)\otimes \mw\,,
\end{equation}
where the last equality relies on \eqref{eq1404}.
Using the convexity and the homogeneity of \(F^{\ve}(x,t,\cdot)\),  the Jensen inequality implies that for every \(x\in \Omega\),
\[
F^{\ve}(x,t, \mv_t * \vr_{\ve}(x)) \leq \int_{\Lambda} \vr_{\ve}(x - y)\d F^{\ve}\left(x,t, \mv_t(y)\right).
\]
For a.e. \(y\in B(x,\ve)\) and for every  \((t,q)\in \R\times \mathbb{S}^N\), it holds that \(F^{\ve}(x,t,q)\leq \widehat{f}(y,t,q)\). This inequality remains true for \(|\mv_t|\otimes \mw\) a.e. \((y,t)\in B(x,\ve)\times \R\) and every \(q\in \mathbb{S}^N\), in view of the fact that   \(|\mv_t|\otimes \mw (A\times \R)=|Dv|(A\times \R)=0\)  for every Borel set \(A\subset \Lambda\) such that \({|A|}=0\), see  Remark~\ref{lemma-Negligibility}. It follows that for \(\mw\) a.e. \(t\in \R\) and every \(x\in \Omega\),
\[
    F^{\ve}(x,t, \mv_t * \vr_{\ve}(x)) \leq \int_{\Lambda} \vr_{\ve}(x - y)\d\wh{f}\left(y, t, \mv_t(y)\right) = \wh{f}(\cdot, t, \mv_t) * \vr_{\ve}(x)\,.
\]
Inserting this estimate in \eqref{eq1446}, one gets
\begin{align*}
F^{\ve}(\cdot,t,Dv_{\ve, \delta}) &\leq  \left(\wh{f}(\cdot, t, \mv_t) * \vr_{\ve}\cH^N\right)\otimes \mw = \left(\wh{f}(\cdot, t, \mv_t )\otimes \mw\right)*_x\vr_{\ve}\\
&= \wh{f}(\cdot,t,\mv_t\otimes \mw)*_x\vr_{\ve}= \wh{f}(\cdot, t, Dv)*_x\vr_{\ve}\,.
\end{align*}
\end{proof}

\begin{remark}\rm
\label{remark-Jensen-measure-autonomous}
When \(f(x,t,\xi)=g(t,\xi)\) for some continuous function \(g:\R\times \R^N\to \rp\) which is convex with respect to the second variable and satisfies \(g(t,0)=0\) for every \(t\in \R\), the corresponding function \(F^\ve\) is simply \(\wh{g}\). The conclusion of Lemma~\ref{lm-estimge} then reads
\[
\widehat{g}(t,Dv_{\ve, \delta})\leq \widehat{g}(t,Dv)*_x\vr_\ve.
\]
We thus recover the Jensen inequality for measures formulated in  \cite[Proposition~3.11]{Bousquet-Pisa}.
\end{remark}

\begin{lemma}\label{lm-estim_Fe}
Let $f$ satisfy~\ref{f-red}  and let $u \in \cE(\Lambda) \cap L^{\infty}(\Lambda)$. For every \(\ve\in (0,\ve_0)\) and \(\delta \in (0,1)\), let $F^\ve$ be given by~\eqref{eq:def-Feps} and let $v_{\ve, \delta}$ be given by~\eqref{def-v-eps-delt} with an absolutely continuous bounded function $\alpha:\R\to(0,\infty)$ satisfying \eqref{alpha'prop} and \eqref{eq1168}. Then there exists \(\mathsf{c}>0\) which depends only on \(\|\vr\|_{L^\infty}\), such that for \(|Dv_{\ve, \delta}|\)-a.e. \((x,t)\in \Omega\times \R\) we have
\[
F^{\ve}(x,t,q_{\ve, \delta})\leq \frac{\mathsf{c}}{\delta \ve^N}(-q^{t}_{\ve, \delta})\,.
\] 
\end{lemma}
\begin{proof}
Let \(A\) be a Borel set in \(\Omega\times \R\). 
By Lemma \ref{lm-abs-convol-f-ch}, the estimate  \(\|\vr_{\ve}\|_{L^\infty}\leq {\|\vr\|_{L^\infty}}{\ve^{-N}}\) and~\eqref{eq1168}, one gets
\[
\int_{A}\dhf(\cdot,t,Dv)*_x\vr_{\ve}\leq \frac{\|\vr\|_{L^\infty}}{\ve^N} \int_{\rn} \int_{\R} 
\chi_A(y,t)
\int_{\Lambda_+\cap u^{-1}(t)}\frac{f(z,t,\nabla u(z))}{|\nabla u(z)|}\d\cH^{N-1}(z)\dt\dy\leq \frac{\|\vr\|_{L^\infty}}{\ve^N} \int_{A}-\alpha'(t)\dt\dy\,,
\]
where  \(\Lambda_+\) is defined in~\eqref{def-Lambda+}.
Then we use that
\begin{equation}\label{eq1890}
\delta\alpha'{\cH^{N+1}} = \partial_t v_{\ve,\delta} -\partial_t v*_x \vr_{\ve} \geq \partial_t v_{\ve,\delta}.
\end{equation}
This gives
\begin{equation}\label{eq-1492}
\int_{A}\dhf(\cdot,t,Dv)*_x\vr_{\ve}\leq \frac{\|\vr\|_{L^\infty}}{\delta\ve^N} \int_{A}-\d\partial_tv_{\ve, \delta}=\frac{\|\vr\|_{L^\infty}}{\delta \ve^N}
\int_{A}(-q^{t}_{\ve, \delta})\d|Dv_{\ve,\delta}|\,.
\end{equation}
In view of Lemma \ref{lm-estimge}, it holds that
\[
\int_{A} F^{\ve}(x,t,q_{\ve, \delta})\d|Dv_{\ve,\delta}|=\int_{A}\d F^{\ve}(x,t,Dv_{\ve, \delta})\leq \frac{\|\vr\|_{L^\infty}}{\delta \ve^N}
\int_{A}(-q^{t}_{\ve, \delta})\d|Dv_{\ve,\delta}|\,.
\]
Since \(A\) is arbitrary, this proves the desired result.
\end{proof}

Lemma~\ref{lm-estim_Fe} provides a suitable estimate for the second term in the left-hand side of \eqref{eqHZ-hat}. The estimate of the first term is given in the following statement. 

\begin{lemma}\label{lm-estim-rat}  
Let $u \in \cE(\Lambda)\cap L^\infty(\Lambda)$. For every \(\ve\in (0,\ve_0)\) and \(\delta \in (0,1)\), let $v_{\ve, \delta}$ be given by~\eqref{def-v-eps-delt} with an absolutely continuous bounded function $\alpha:\R\to(0,\infty)$ satisfying \eqref{alpha'prop} and \eqref{eq1168}.  Then there exists \(\mathsf{c}'>0\) which depends only on \(\|\vr\|_{L^\infty}\), \(\|\nabla\vr\|_{L^1}\), and  $c_M$ from \eqref{alpha'prop} with \(M=\|u\|_{L^{\infty}(\Lambda)}\), such that for \(|Dv_{\ve, \delta}|\)-a.e. \((x,t)\in \Omega\times \R\) we have
\[ 
\left(\frac{|q^{x}_{\ve, \delta}(x,t) |}{|q^{t}_{\ve, \delta}(x,t) |}\right)^{\max(p,N)} \leq \frac{\mathsf{c}'}{\delta^N\ve^{N}}\,.
\] 
\end{lemma}
\begin{proof}
We first prove that for \(|Dv_{\ve, \delta}|\)-a.e. \((x,t)\in \Omega\times \R\), one has \begin{equation}\label{(qx/qt)p-est}\left(\frac{|q^{x}_{\ve, \delta}|}{|q^{t}_{\ve, \delta}|}\right)^p(x,t)\leq \frac{C}{\delta \ve^N}\qquad\text{for some }\ C>0\,.\end{equation} 
Let us introduce the map \(g(\xi)\coloneqq|\xi|^p\). 
Since \(q^t_{\ve,\delta}(x,t)<0\) for \(|Dv_{\ve, \delta}|\)-a.e. \((x,t)\), this is equivalent to
the fact that for every Borel set \(A'\subset \Omega\times \R\), it holds
\[ \int_{A'} -q^t_{\ve,\delta}g\left(\frac{q^x_{\ve,\delta}}{-q^t_{\ve,\delta}}\right)\d |Dv_{\ve, \delta}|\leq \frac{C}{\delta \ve^N}\int_{A'}-q^t_{\ve,\delta}\d |Dv_{\ve,\delta}|\,.
\]
In turn, by the very definition of $\wh g$ in~\eqref{def-hat}, it suffices to show that
\begin{equation}\label{eq1009}
\int_{A'}\d\wh{g}(Dv_{\ve, \delta})
\leq \frac{C}{\delta\ve^N}\int_{A'}-q^t_{\ve,\delta}\d |Dv_{\ve,\delta}|\,.
\end{equation}
By Remark~\ref{remark-Jensen-measure-autonomous}, the left-hand side above satisfies
\[
\int_{A'}\d \wh{g}(Dv_{\ve, \delta})\leq \int_{A'}\d \wh{g}(Dv)*_x \vr_{\ve}.
\]
Since the function \((x,t,\xi)\mapsto g(\xi)\) is a Carathéodory  non-negative  function which is convex with respect to the last variable and~\eqref{eq1168} is satisfied, we obtain similarly to \eqref{eq-1492} the following countepart for \(g\):
\[
\int_{A'}\d \wh{g}(Dv)*_x \vr_{\ve}\leq \frac{\|\vr\|_{L^\infty}}{\delta \ve^N}
\int_{A'}(-q^{t}_{\ve, \delta})\d|Dv_{\ve,\delta}|\,.
\]
Hence, \eqref{eq1009} is a consequence of the last two displays and \eqref{(qx/qt)p-est} is justified.\newline

Let us now prove that for \(|Dv_{\ve, \delta}|\)-a.e. \((x,t)\in \Omega\times \R\), one has \begin{equation}\label{(qx/qt)N-est}
    \left(\frac{|q^{x}_{\ve, \delta}|}{|q^{t}_{\ve, \delta}|}\right)(x,t)\leq \frac{C}{\delta \ve}\qquad\text{for some }\ C>0\,.
\end{equation}
For any \(e\in \Sphere^{N-1}\), we introduce the function $h_e:\Omega\times\R\times\rn\to\R$ given by $h_e(x,t,\xi)= \langle \xi, e \rangle$. Note that $h$ is linear with respect to \(\xi\) and satisfies \(h_e(x,t,0)=0\). We then rely on  Lemma~\ref{lm-abs-convol-f-ch} applied to \(h_e\) instead of \(f\) and Remark~\ref{remark-Jensen-measure-autonomous} to get for every Borel set \(A'\subset \Omega\times \R\):
\begin{equation}
\int_{A'}\d \wh{h}_e(Dv_{\ve, \delta})
\leq \int_{\Omega}\int_{\R}  \chi_{A'}(y,t) \int_{\Lambda_+\cap u^{-1}(t)}\vr_{\ve}(y-z)\frac{\partial_e u(z)}{|\nabla u(z)|}\d\cH^{N-1}(z)\dt \dy\label{eq1056}\,.
\end{equation}
Fix \(y\in \Omega\). Then, the function \(\vr_{\varepsilon}(y-\cdot)\) is compactly supported in \(\Lambda\).
Applying Lemma~\ref{lemma-alt-Stokes-Sobolev} on \(\Lambda\) instead of \(\Omega\) with the functions 
\[\Phi(x,t)\coloneqq \vr_{\varepsilon}(y-x)\eta(t)\quad\text{and} \quad  \theta (t)\coloneqq  \chi_{A'\cap (\Lambda\times [-M,M])}(y, t)\,, 
\]  
where \(\eta\) is any function in \(C^{\infty}_c(\R)\) such that \(\eta\equiv 1\) on \([-M,M]\) (recall that \(M=\|u\|_{L^{\infty}}\)), 
one gets
\[
\int_{\Lambda}\chi_{A'}(y,u(x))\vr_\varepsilon(y-x)\partial_e u(x)\dx=\int_{-M}^{M}\chi_{A'}(y,t)\left(\int_{[u\geq t]}\partial_e\vr_\varepsilon(y-x)\dx\right)\dt\,.
\]
In the left-hand side, one can restrict the domain of integration to \(\Lambda_+\) instead of \(\Lambda\), and then get by the coarea formula, 
\[
\int_{\R}\chi_{A'}(y,t)\left(\int_{u^{-1}(t)\cap \Lambda_+}\vr_\varepsilon(y-z)\frac{\partial_e u(z)}{|\nabla u(z)|}\d\cH^{N-1}(z)\right)\dt=\int_{-M}^{M}\chi_{A'}(y,t)\left(\int_{[u\geq t]}\partial_e\vr_\varepsilon(y-x)\dx\right)\dt\,.
\]

Integrating with respect to \(y\in \Omega\) and inserting the resulting estimate into \eqref{eq1056}, one gets
\begin{align}
\int_{A'}\d \wh{h}_e(Dv_{\ve, \delta})
&\leq \int_{\Omega}\int_{-M}^{M}\chi_{A'}(y,t)\left(\int_{\R^N}|\partial_e\vr_\varepsilon(y-x)|\dx\right)\dt\dy\nonumber\\
&\leq \tfrac{1}{\ve}\|\nabla \vr\|_{L^{1}(\rn)} |A'\cap (\Omega\times (-M,M))|\,.\label{eq1071}
\end{align}
By  \eqref{alpha'prop}, we know that \(\essinf_{t\in [-M,M]}- \alpha'(t)\geq c_M>0\). Taking into account that $\delta\alpha'{\cH^{N+1}} \geq \partial_t v_{\ve,\delta}$ in the sense of non-positive measures, we find that
\begin{align*}
\int_{A'}\d \wh{h}_e(Dv_{\ve, \delta})
&\leq \frac{\|\nabla \vr\|_{L^{1}(\rn)}}{c_M\delta\ve}\int_{A'\cap (\Omega \times (-M,M))}-\d \partial_t v_{\ve, \delta}\leq \frac{\|\nabla \vr\|_{L^{1}(\rn)}}{c_M\delta\ve}\int_{A'}-\d \partial_t v_{\ve, \delta}\,.
\end{align*}
This is equivalent to the fact that for \(|Dv_{\ve, \delta}|\) a.e. \((x,t)\in \Omega \times \R\), it holds
\[
\langle q_{\ve,\delta}^x , e \rangle \leq \frac{\|\nabla \vr\|_{L^{1}(\rn)}}{c_M {\delta}\ve} |q_{\ve,\delta}^t|\,.
\]
Since \(e\) is arbitrary, this proves~\eqref{(qx/qt)N-est}. We then have both~\eqref{(qx/qt)p-est} and~\eqref{(qx/qt)N-est}, which completes the proof.
\end{proof}

We are now in position to apply \eqref{eqHZ-hat} and derive its consequences to estimate the measure \(\wh{f}(x,t,Dv_{\ve, \delta})\) from above. 
\begin{lemma}\label{lm-disintegration-nonautonomous}
Suppose  $f$ satisfies Assumptions~\ref{f-red} and \eqref{eqHZconv}. For every \(\ve\in (0,\ve_0)\) and \(\delta \in (0,1)\), let $v_{\ve, \delta}$ be given by~\eqref{def-v-eps-delt} with $u\in\cE(\Lambda) \cap L^{\infty}(\Lambda)$ and $\alpha$ satisfying \eqref{alpha'prop} and \eqref{eq1168}. 
Let  $\mathsf{c}$ and $\mathsf{c}'$ be given by Lemma~\ref{lm-estim_Fe} and Lemma~\ref{lm-estim-rat} respectively and let \(\cC_{\mathcal{k_\delta}}>0\) be given by \eqref{eqHZ-hat} for $\mathcal{k}_{\delta}=(M,(\mathsf{c}+\mathsf{c}')/{\delta^{N}})$ with \(M=\|u\|_{L^{\infty}(\Lambda)}\). Then for every  \(|Dv_{\ve, \delta}|\)-measurable set \(A'\subset \Omega\times \R\), it holds that
\begin{equation}\label{eq1500}
\int_{A'}\d \wh{f}(x,t,Dv_{\ve, \delta}) \leq \cC_{\mathcal{k}_{\delta}}\int_{A'}\d \left(\wh{f}(x,t,Dv)*_x \vr_{\ve}+\wh{1}(Dv_{\ve, \delta})\right)\,.
\end{equation}
\end{lemma}
\begin{proof}
Recall the definition of $F^\ve$ from~\eqref{eq:def-Feps}\,, and the definition of $q_{\ve, \delta}$ from~\eqref{eq:def-q}. By Lemma \ref{lm-estim_Fe} and Lemma~\ref{lm-estim-rat}, for \(|Dv_{\ve, \delta}|\)-a.e. \((x,t)\in \Omega\times \R\), we have
\[
\left(\frac{|q^{x}_{\ve,\delta}|}{|q^{t}_{\ve, \delta}|}\right)^{\max(p,N)}+\frac{1}{|q^{t}_{\ve, \delta}|}F^{\ve}(x,t,q_{\ve, \delta})\leq \frac{\mathsf{c}+\mathsf{c}'}{{\delta^N} \ve^N}\,.
\]
%
In view of \eqref{eqHZ-hat} with \(\mathcal{k}_\delta=(M,(\mathsf{c}+\mathsf{c}')/{\delta^N})\) and also Remark~\ref{lemma-Negligibility}, this implies that there exists \(\cC_{\mathcal{k}_\delta}>0\) such that for \(|Dv_{\ve,\delta}|\) a.e. \((x,t)\in \Omega\times [-M,M]\),
\[
\wh{f}(x,t,q_{\ve, \delta})\leq \cC_{\mathcal{k}_{\delta}}\left( F^\ve(x,t,q_{\ve, \delta})+ |q^{t}_{\ve, \delta}|\right)\,.
\]
Since \(q^{x}_{\ve, \delta}=0\) outside \(\Omega\times [-M,M]\) and \(\wh{f}(x,t,(0,q^t))=0\) for every \(q^t\leq 0\), the above inequality is actually true \(|Dv_{\ve,\delta}|\) a.e. \((x,t)\in \Omega\times \R\). 
In terms of measures, this yields
\[
\wh{f}(x,t,Dv_{\ve, \delta})\leq \cC_{\mathcal{k}_{\delta}}\left( F^\ve(x,t,Dv_{\ve, \delta})+ |q^{t}_{\ve, \delta}||Dv_{\ve, \delta}|\right)\,.
\]
By Lemma \ref{lm-estimge},
\begin{equation*}
\wh{f}(x,t,Dv_{\ve, \delta})\leq \cC_{\mathcal{k}_{\delta}} \left(\wh{f}(\cdot, t, Dv)*_x\vr_{\ve}+ |q^{t}_{\ve, \delta}||Dv_{\ve, \delta}|\right)\,.
\end{equation*}
Since \(q^{t}_{\ve, \delta}\leq 0\), the definition of \(\wh{1}\) implies that  \( |q^{t}_{\ve, \delta}||Dv_{\ve, \delta}|=\wh{1}(Dv_{\ve, \delta})\), which completes the proof of \eqref{eq1500}. 
\end{proof}

\subsection{A key estimate}\label{ssec:key-est}
 
To prove Proposition~\ref{prop-limsup}, we split the domain \(\Omega\) into two regions. In the good one, corresponding to the set where \(\nabla u\) is small, we rely on the Reschetnyak continuity theorem, or more specifically, on its formulation given in  Lemma~\ref{lm-Reschetnyak}. In the bad region where \(\nabla u\) is large, we use the following estimate, which is based on the consequences of the Jensen inequality that we have derived in the previous section. The decomposition of \(\Omega\) in two parts is conveyed through a function  \(\beta\) which is subsequently chosen to vanish in the good region (see the proof of Proposition~\ref{prop-limsup} in Section~\ref{ssec:limsup-proof}).

\begin{proposition}\label{prop-limsup-hard} 
Suppose that $f$ satisfies Assumptions~\ref{f-red} and \eqref{eqHZconv}.
Let \(\beta:\R\to \rp\) be a~continuous, bounded, and non-decreasing function. Let $v_{\ve, \delta}$ and $q_{\ve, \delta}=(q_{\ve, \delta}^x,q_{\ve, \delta}^t)$ be given by~\eqref{def-v-eps-delt} and~\eqref{eq:def-q}, respectively, with $u \in \cE(\Lambda)\cap L^\infty(\Lambda)$ and $\alpha$ satisfying \eqref{alpha'prop} and \eqref{eq1168}.  Then for every  \(\delta>0\), there exists \(\wcC>0\) such that
\begin{multline}\label{373}
\limsup_{\ve\to 0} \int_{\Omega\times \R} \beta(q^{t}_{\ve,\delta}(x,t))\wh{f}(x,t, q_{\ve,\delta}(x,t))\d |Dv_{\ve,\delta}|\\
   \leq   \wcC\left[\int_{\Omega}\beta\left(\frac{-1}{\sqrt{|\nabla u(x)|^2+1}}\right)f(x,u(x), \nabla u(x))\dx 
+\int_{\Omega\times \R}\beta(q^{t}_{0, \delta}(x,t))\dho (Dv_{0, \delta})\right]. 
\end{multline}      
\end{proposition}

\begin{proof} We proceed step by step. \newline
     
\noindent{\bf Step 1. Decomposition of $Dv$ and $Dv_{\ve,\delta}$.} 
 Let us recall the decomposition \(\Lambda=\Lambda_+\cup \Lambda_0\), with  \(|\nabla u|>0\) a.e. on \(\Lambda_+\) and \(\nabla u=0\) a.e. on \(\Lambda_0\).
  We define $\ell: \Lambda\to \R^{N+1}$ by the following formula
\begin{equation}
    \label{def:ell}
\ell (x)\coloneqq \chi_{\Lambda_+}(x)\frac{(\nabla u(x),-1)}{|\nabla u(x)|}\,.
\end{equation}
We will show  that there exist  non-negative finite measures \(\gamma_t\) on \(\Lambda\), a non-negative function \(h\in L^{1}(\R)\), and a~non-negative finite  measure \(\mu_0\) on \(\R\) such that \(\dt\) and \(\mu_0\) are mutually singular,   and
\begin{equation}
    \label{Dv-decomposition}
Dv(x,t)= (\ell(x)\cH^{N-1}\mres  u^{-1}(t)){\otimes \cH^1} -(0, h(t)\gamma_t){\otimes \cH^1}  -(0,\gamma_t)\otimes \mu_0\,.
\end{equation}
By \eqref{Dv-other-formulation} applied with \(\Lambda\) instead of \(\Omega\) and \(A=\Lambda_+\), one has:
\begin{equation}\label{eq1561}
Dv\mres  (\Lambda_+\times \R) = (\ell\cH^{N-1}\mres  u^{-1}(t)){\otimes \cH^1}\,.
\end{equation}

On the other hand,   for every bounded Borel map \(g:\Lambda_0\times \R\to \R\), we have by \eqref{def-Dv} and the area formula, 
\[
\int_{\Lambda_0\times \R}g(x,t)\d Dv(x,t)=\int_{\Lambda_0}(0,-g(x,u(x)))\dx\, .
\]
This proves that \((Dv\mres  (\Lambda_0\times \R))^x =0\) and \((Dv\mres  (\Lambda_0\times \R))^t\) is a non-positive finite measure. By disintegration,  there exist  non-negative probability measures  \(\gamma_t\) on \(\Lambda\) and a non-negative finite measure \(\overline{\mu}_0\) on \(\R\) such that
\[
Dv\mres  (\Lambda_0\times \R) = -(0,\gamma_t)\otimes \overline{\mu}_0\,.
\]
We decompose \(\overline{\mu}_0=h(t)\dt +\mu_0\) where \(h\in L^{1}(\R)\) is non-negative, while \(\mu_{0}\) and \(\dt\) are mutually singular.  
Hence,
\begin{align*}
Dv&=Dv\mres  (\Lambda_+\times \R)+ Dv\mres  (\Lambda_0\times \R) =(\ell\cH^{N-1}\mres  u^{-1}(t)){\otimes \cH^1} - (0, h(t)\gamma_t){\otimes \cH^1}  -(0,\gamma_t)\otimes \mu_{0}\,.
\end{align*}
This completes the proof of \eqref{Dv-decomposition}.
It follows from the latter that for every \(\delta\in (0,1)\),
\begin{equation}\label{eq-decomp-Dv0d}
    Dv_{0,\delta}=\Big[\left(\ell\cH^{N-1}\mres  u^{-1}(t)\right)-(0, h(t)\gamma_t) +(0,\delta\alpha'(t)){\cH^N} \Big]{\otimes \cH^1} -(0, \gamma_t)\otimes \mu_0\,.
\end{equation}
Then, by partial convolution with \(\vr_\ve\), one obtains that
\begin{equation}\label{eq-decomp-Dved}
    Dv_{\ve,\delta}(x,t)=\big[\zeta_\ve(x,t)-(0, h(t)\gamma_t*\vr_{\ve}(x)) +(0,\delta\alpha'(t)) \big]{\cH^{N+1}} -(0, \gamma_t*\vr_\ve(x)){\cH^N}\otimes \mu_0\,,
\end{equation}
where \(\zeta_\ve:\R^N\times \R \to \R^{N+1}\) is a Borel function  such that  for a.e. \(t\in \R\) and  for every \(x\in \R^N\), it holds
\begin{equation}
    \label{def:zeta-eps}
\zeta_\ve(x,t)=
\left(\ell\cH^{N-1}\mres  u^{-1}(t)\right)*\vr_{\ve}(x)=\int_{u^{-1}(t)}\vr_\ve(x-z)\ell(z)\d\cH^{N-1}(z)\,.
\end{equation}

\noindent{\bf Step 2. Continuity of $\ell$ given by~\eqref{def:ell}. } 
By the coarea formula, the map \(\ell\) is summable on \(u^{-1}(t)\) for a.e. \(t\in \R\). We proceed to show that for  \(|Dv|\)-a.e. \((x,t)\in \Lambda_+\times \R\)
and for every continuous function \(\phi\in C^{0}_c(\rn)\), it holds
\begin{equation}
    \label{eqE}
\lim_{\ve\to 0} \frac{1}{\ve^{N-1}}\int_{u^{-1}(t)} \ell(z)\,\phi\left(\frac{z-x}{\ve}\right)\d\cH^{N-1}(z) = \ell(x) \int_{T_x u^{-1}(t)}\phi(y)\d\cH^{N-1}(y)\,.
\end{equation}
Remember that the approximate tangent space \(T_x u^{-1}(t)\) to the level set \(u^{-1}(t)\) exists at \(\cH^{N-1}\)-a.e. \(x\in u^{-1}(t)\), see Section~\ref{sec:GMT}. This means that for every \(\phi\in C^{0}_c(\rn)\), it holds
\begin{equation}\label{eq898}
\lim_{\ve\to 0}\frac{1}{\ve^{N-1}} \int_{\Lambda}\phi\left(\frac{y-x}{\ve}\right) \d \left(\cH^{N-1}\mres  u^{-1}(t)\right)(y) - \int_{T_x u^{-1}(t)} \phi(y)\d\cH^{N-1}(y)=0\,.
\end{equation}
We fix \(t\in \R\) such that \(u^{-1}(t)\) is a countably \(\cH^{N-1}\) rectifiable set, \(\cH^{N-1}(u^{-1}(t))<\infty\), and \(\ell\) is summable on \(u^{-1}(t)\).

The measure \(\cH^{N-1}\mres u^{-1}(t)\) is finite and Borel regular (see e.g.~\cite[Theorem~1 of Section~2.1 and Theorem~3 of Section~1.1]{Evans-Gariepy}). Since  \(\ell|_{u^{-1}(t)}\)  is summable   with respect to this measure, then \(\cH^{N-1}\)-a.e. \(x\in u^{-1}(t)\) is a Lebesgue point of \(\ell|_{u^{-1}(t)}\) (see e.g.~\cite[Corollary 1 of Section 1.7]{Evans-Gariepy}):
\begin{equation}\label{eq1830}
\lim_{\ve\to 0}\frac{1}{\cH^{N-1}(u^{-1}(t)\cap B(x,\ve))}\int_{u^{-1}(t)\cap B(x,\ve)}|\ell-\ell(x)|\d\cH^{N-1}=0\,.
\end{equation}
Using that \(u^{-1}(t)\) is \(\cH^{N-1}\) measurable, one has (see \cite[Section 2.3]{Evans-Gariepy}),
\[
\limsup_{\ve\to 0} \frac{\cH^{N-1}(u^{-1}(t)\cap B(x,\ve))}{\ve^{N-1}}<\infty\,,
\]
for \(\cH^{N-1}\)-a.e. \(x\in \R^N\). Together with \eqref{eq1830}, this implies \begin{equation}\label{eq-Lebesgue-point}
\lim_{\ve\to 0}\frac{1}{\ve^{N-1}}\int_{u^{-1}(t)\cap B(x,\ve)}|\ell-\ell(x)|\d\cH^{N-1}=0\,.
\end{equation}
For every  \(x\in u^{-1}(t)\) such that the above identity holds and \(T_xu^{-1}(t)\) exists, we make the following estimate:
\begin{align*}
\left|\frac{1}{\ve^{N-1}}\right. \int_{u^{-1}(t)}  \ell(z)\phi\left(\frac{z-x}{\ve}\right)&\d\cH^{N-1}(z)
\left.- \ell(x) \int_{T_{x}u^{-1}(t)}\phi(z)\d\cH^{N-1}(z)\right|\\
&\leq \frac{1}{\ve^{N-1}}\int_{u^{-1}(t)} |\ell(z)-\ell(x)|\phi\left(\frac{z-x}{\ve}\right)\d\cH^{N-1}(z)
\\
&\quad +|\ell(x)|\left|\frac{1}{\ve^{N-1}}\int_{u^{-1}(t)} \phi\left(\frac{z-x}{\ve}\right)\d\cH^{N-1}(z)-\int_{T_{x}u^{-1}(t)}\phi(z)\d\cH^{N-1}\right|\eqqcolon \textrm{I}_\ve+\textrm{II}_\ve\,.
\end{align*}
Note that 
\[
\mathrm{I}_\ve\leq\frac{\|\phi\|_{L^{\infty}}}{\ve^{N-1}}\int_{B(x,{\ve R})\cap u^{-1}(t)} |\ell(z)-\ell(x)|\d\cH^{N-1}(z)\,,
\]
where we choose \(R>0\) such that \(\supp \phi \subset B(0,R)\), so $\mathrm{I}_\ve\to 0$ as $\ve\to 0$ by using \eqref{eq-Lebesgue-point}. Moreover, \(\textrm{II}_\ve\to 0\) when \(\ve\to 0\) by \eqref{eq898}.  It follows that \eqref{eqE} holds for a.e. \(t\in \R\) and for \({\cH}^{N-1}\)-a.e. \(x\in u^{-1}(t)\). 
Since by \eqref{eq1561} one has \(|Dv|=(|\ell|\cH^{N-1}\llcorner u^{-1}(t))\otimes {\cH^1}\) on \(\Lambda_+\times \R\),  the identity \eqref{eqE} also holds for \(|Dv|\)-a.e. \((x,t)\in \Lambda_+\times \R\). 
\newline

\noindent{\bf Step 3. Completion of the proof of \eqref{373} } 
Fix \(\delta\in (0,\infty)\) and \(\ve\in (0, \ve_0)\). The map \((x,t)\in \Omega\times \R\mapsto \beta(q_{\ve,\delta}^t(x,t))\) is \(|Dv_{\ve,\delta}|\)-measurable and bounded. 
From Lemma \ref{lm-disintegration-nonautonomous}, we deduce that
\begin{equation}
\label{eq1737}
\frac{1}{ \cC_{\mathcal{k}_\delta}}\int_{{\Omega}\times \R}\beta(q^{t}_{\ve, \delta}(x,t))\wh{f}(x,t, q_{\ve, \delta}(x,t))\d |Dv_{\ve, \delta}|   
\leq \int_{{\Omega}\times \R}\beta(q^{t}_{\ve, \delta}(x,t))\d \widehat{f}(x,t,Dv)*_x\vr_{\varepsilon}+ \int_{{\Omega}\times \R}\beta(q^{t}_{\ve, \delta}(x,t))\d \widehat{1}(Dv_{\ve, \delta})\,.    
\end{equation}
For the last term,
we claim that
\begin{equation}
\label{eq-addendum-adelta-2}
\lim_{\ve\to 0}\int_{\Omega\times \R} \beta(q^{t}_{\ve, \delta}(x,t)) \dho (Dv_{\ve, \delta})
= \int_{\Omega \times \R}\beta(q^{t}_{0, \delta})\dho (Dv_{0,\delta})\,.
\end{equation}
Indeed, since \(q^{t}_{\ve,\delta}<0\) a.e. for every \(\ve, \delta\geq 0\), one has \(\widehat{1}(Dv_{\ve,\delta})=|\partial_t v_{\ve,\delta}|\).  Moreover, we know by \eqref{eq16755}, that \(|Dv_{\ve, \delta}|(\Omega\times \R)\) converges to \(|Dv_{0,\delta}|(\Omega\times \R)\). Hence, \eqref{eq-addendum-adelta-2} follows from the Reshetnyak Continuity Theorem \cite[Theorem 2.39]{AmFuPa} applied to the continuous bounded map \(q\in \mathbb{S}^N\mapsto \beta(q^t)|q^t|\).

By the decomposition of \(Dv\) described in Step 1, the measures \(\mu_0\) and {$\cH^1$} are mutually singular on \(\R\). Hence, there exists a Borel subset \(I\subset \R\) such that \(\mu_0(I)=0\) and \( {|\R\setminus I|}=0\).  
For \(|Dv_{\ve, \delta}|\) a.e. \((x,t)\in \Omega\times I\),  the decomposition \eqref{eq-decomp-Dved} implies that
\[
q_{\ve, \delta}(x,t)= \frac{Dv_{\ve,\delta}(x,t)}{|Dv_{\ve,\delta}(x,t)|} = \frac{\zeta_\ve(x,t)+(0, -h(t)\gamma_t*\vr_{\ve}(x)+\delta \alpha'(t))}{|\zeta_\ve(x,t)+(0, -h(t)\gamma_t*\vr_{\ve}(x)+\delta \alpha'(t))|}\,.
\]
A simple calculation shows that for every \(\xi\in\rn\) and $a,b<0$, one has
\begin{equation*}
\frac{a+b}{|(\xi,a+b)|}\leq \frac{a}{|(\xi,a)|}\,.
\end{equation*}
We apply this remark with \((\xi,a)=\zeta_\ve(x,t)\) and \(b=-h(t)\gamma_t*\vr_{\ve}(x)+\delta \alpha'(t)\) to get
\begin{equation*}
    \label{estim-qt}
q_{\ve, \delta}^t(x,t)\leq \frac{\zeta_\ve^t(x,t)}{|\zeta_\ve(x,t)|}\,,
\end{equation*}
where the right-hand side has to be understood as \(0\) when \(\zeta_\ve(x,t)=0\).
Since \(\beta\) is non-decreasing, this gives for \(|Dv_{\ve, \delta}|\) a.e. \((x,t)\in \Omega\times I\),
\begin{equation}\label{eq-1749}
\beta\left(q_{\ve, \delta}^t(x,t)\right)\leq \beta\left(\frac{\zeta_\ve^t(x,t)}{|\zeta_\ve(x,t)|}\right)\,.
\end{equation}
As already observed, \(\widehat{f}(x,t,Dv)*_x\vr_{\varepsilon}\) is absolutely continuous with respect to \(\cH^{N+1}\). In particular, \(\widehat{f}(x,t,Dv)*_x\vr_{\varepsilon}(\Omega\times (\R\setminus I))=0\) and thus
\begin{align*}
\int_{{\Omega}\times \R}\beta(q^{t}_{\ve, \delta}(x,t))\d \widehat{f}(x,t,Dv)*_x\vr_{\varepsilon}
&=\int_{{\Omega}\times I}\beta(q^{t}_{\ve, \delta}(x,t))\d \widehat{f}(x,t,Dv)*_x\vr_{\varepsilon}
&\leq \int_{{\Omega}\times I}\beta\left(\frac{\zeta_\ve^t(x,t)}{|\zeta_\ve(x,t)|}\right)\d \widehat{f}(x,t,Dv)*_x\vr_{\varepsilon}\,,
\end{align*}
where the last inequality follows from \eqref{eq-1749}. Since the integrand in the right-hand side is non-negative, one can replace \(I\) by \(\R\). 
Together with \eqref{eq-def-conv-partial}, this gives
\begin{equation}\label{eq1814}
\int_{{\Omega}\times \R}\beta(q^{t}_{\ve, \delta}(x,t))\d \widehat{f}(x,t,Dv)*_x\vr_{\varepsilon}
\leq \int_{\Lambda\times \R}\left(\int_{\Omega}\beta\left(\frac{\zeta_\ve^t}{|\zeta_\ve|}(y,t)\right)\vr_\ve(y-x)\dy\right)\d \widehat{f}(x,t,Dv)
\eqqcolon \mathrm{J}_\ve+\mathrm{J}_{\ve}'\,,
\end{equation}
 with
\[
\mathrm{J}_\ve\coloneqq\int_{\Lambda\times \R}\left(\int_{\Omega}\beta\left(\frac{\ell^t}{|\ell|}(x)\right)\vr_\ve(y-x)\dy\right)\d \widehat{f}(x,t,Dv)\,,
\]
\[
\mathrm{J}_{\ve}'\coloneqq\int_{\Lambda\times \R}\left(\int_{\Omega}\left(\beta\left(\frac{\zeta_\ve^t}{|\zeta_\ve|}(y,t)\right)-\beta\left(\frac{\ell^t}{|\ell|}(x)\right)\right)\vr_\ve(y-x)\dy\right)\d \widehat{f}(x,t,Dv)\,.
\]
We observe that
\[
\mathrm{J}_\ve=\int_{\Lambda\times \R} \beta\left(\frac{\ell^t}{|\ell|}(x)\right) \chi_{\Omega}* \widetilde{\vr}_\ve(x)\d\,\widehat{f}(x,t,Dv)\,,
\]
where \(\chi_{\Omega}\) is the indicator function of \(\Omega\) and \(\widetilde{\vr}(x)\coloneqq\vr(-x)\) for every \(x\in \R^N\). Using that \(\chi_{\Omega}*\widetilde{\vr}_{\ve}(x)\) converges to \(\chi_{\Omega}(x)\) \(\cH^{N}\) for a.e. \(x\in \Omega\) together with Remark~\ref{lemma-Negligibility}, we get that \(\beta\left(\frac{\ell^t}{|\ell|}(x)\right) \chi_{\Omega}* \widetilde{\vr}_{\ve}(x)\) converges outside a \(|Dv|\)-negligible set, and thus for  \(\wh{f}(x,t,Dv)\) a.e. \((x,t)\in \Lambda\times \R\). Since the integrand of \({\rm J}_\varepsilon\) is uniformly bounded by \(\|\beta\|_{L^{\infty}(\R)}\) on \(\Lambda\times \R\) and the measure \(\wh{f}(x,t,Dv)\) is finite, the dominated convergence theorem implies that
\[
\lim_{\ve\to 0}\mathrm{J}_{\ve}= \int_{\Omega\times \R} \beta\left(\frac{\ell^t}{|\ell|}(x)\right) \d\widehat{f}(x,t,Dv)\,.
\]
By the definition \eqref{def:ell} of \(\ell\), the area formula and the fact that \(f(x,u(x),\nabla u (x))=0\) on \(\Lambda_0\), this can be written as:
\begin{equation}\label{eq1831}
\lim_{\ve\to 0}\mathrm{J}_\ve = \int_{\Omega\cap \Lambda_+} \beta\left(\frac{-1}{\sqrt{1+|\nabla u(x)|^2}}\right)f(x,u(x),\nabla u(x))\dx\,.
\end{equation}
As for \(\mathrm{J}_{\ve}'\), using that \(\widehat{f}(x,t,Dv)(\Lambda_0\times \R)=0\) we have the following estimate:
\begin{flalign}
\mathrm{J}_{\ve}'
&=\int_{\Lambda\times \R}\left(\int_{\frac{\Omega-x}{\ve}}\left(\beta\left(\frac{\zeta_\ve^t}{|\zeta_\ve|}(x+\ve y',t)\right)-\beta\left(\frac{\ell^t}{|\ell|}(x)\right)\right)\vr(y')\dy'\right)\d\,\widehat{f}(x,t,Dv) \nonumber\\
&\leq \int_{\Lambda_+\times \R}\left(\int_{[\varrho>0]}\left|\beta\left(\frac{\zeta_\ve^t}{|\zeta_\ve|}(x+\ve y',t)\right)-\beta\left(\frac{\ell^t}{|\ell|}(x)\right)\right|\vr(y')\dy'\right)\d\,\widehat{f}(x,t,Dv)\, \label{eq1837}.
\end{flalign}
We next observe that by \eqref{def:zeta-eps}, \eqref{eqE},  and~\eqref{def-vr-ve}, for \(|Dv|\)-a.e.  \((x,t)\in \Lambda_+\times \R\) and for every \(y'\in  [\vr>0]\), we have
\[
\ve \zeta_\ve(x+\ve y',t)=\ve^{1-N} \int_{u^{-1}(t)}\vr\left(\frac{x-z}\ve+y'\right)\ell (z)\d\cH^{N-1}(z)\xrightarrow[]{\ve\to 0}\ell(x)\int_{T_{x}u^{-1}(t)}\vr(y'-z)\d\cH^{N-1}(z)\,. 
\]
 The last integral is strictly positive since \(\vr\) is strictly positive on a neighbourhood of \(y'\). 
 Thus,
\begin{equation}\label{eq825}
\lim_{\ve\to 0}\tfrac{\zeta_\ve^t}{|\zeta_\ve|}(x+\ve y',t)= \tfrac{\ell^t}{|\ell|}(x)\,.
\end{equation}
By the dominated convergence theorem applied in \eqref{eq1837}, this gives
\(\lim_{\ve\to 0}\mathrm{J}_{\ve}'=0\). In view of \eqref{eq1814} and \eqref{eq1831}, we have thus proved that
\[
\limsup_{\ve\to 0}\int_{{\Omega}\times \R}\beta(q^{t}_{\ve, \delta}(x,t))\d \widehat{f}(x,t,Dv)*_x\vr_{\varepsilon}\leq \int_{\Omega\cap \Lambda_+} \beta\left(\frac{-1}{\sqrt{1+|\nabla u(x)|^2}}\right)f(x,u(x),\nabla u(x))\dx\,.
\]
Inserting this estimate in \eqref{eq1737} and taking into account \eqref{eq-addendum-adelta-2}, the conclusion of Proposition~\ref{prop-limsup-hard} follows.
\end{proof}

\subsection{Completion of the  proof of the limsup estimate}\label{ssec:limsup-proof}
\begin{proof}[Proof of Proposition~\ref{prop-limsup}]
As in the proof of the liminf estimate, we make use of a  sequence of continuous decreasing functions \(\theta_k\) from \eqref{theta_k}. We write
\[
\limsup_{\ve\to 0}\int_{\Omega\times \R}\d \widehat{f}(x,t,Dv_{\ve,\delta})
\leq \limsup_{\ve\to 0}\int_{\Omega\times \R}\theta_k(q^{t}_{\ve,\delta})\d \widehat{f}(x,t,Dv_{\ve,\delta})
+\limsup_{\ve\to 0}\int_{\Omega\times \R}(1-\theta_k(q^{t}_{\ve,\delta}))\d \widehat{f}(x,t,Dv_{\ve,\delta}).
\]
By Lemma~\ref{lm-Reschetnyak} with \(\theta=\theta_k\) and the fact that \(\theta_k\leq 1\) and then~\eqref{eq1717prim}, we can estimate the first term as follows 
\begin{equation}\label{eq-limsup-easy}
 \limsup_{\ve\to 0}\int_{\Omega\times \R}\theta_k(q^{t}_{\ve,\delta})\d \widehat{f}(x,t,Dv_{\ve,\delta})\leq \int_{\Omega\times \R}\d \widehat{f}(x,t,Dv_{0,\delta}) = \int_{\Omega\times \R}\d \widehat{f}(x,t,Dv)\,.    
\end{equation}
For the second term, we rely on Proposition~\ref{prop-limsup-hard} with \(\beta=1-\theta_k\) on \(\Omega\) to get
\begin{multline*}
\limsup_{\ve\to 0} \int_{\Omega\times \R} (1-\theta_k(q^{t}_{\ve,\delta}))\wh{f}(x,t, q_{\ve,\delta})\d |Dv_{\ve,\delta}|\\
   \leq    \wcC\int_{\Omega}\left[1-\theta_k\left(\frac{-1}{\sqrt{|\nabla u(x)|^2+1}}\right)\right]f(x,u(x), \nabla u(x))\dx + \wcC\int_{\Omega\times \R} (1-\theta_k(q^{t}_{0, \delta}(x,t)))\dho (Dv_{0, \delta})\eqqcolon \mathrm{R}_k\,.
\end{multline*}
Let us show that $\mathrm{R}_k\to 0$ for $k\to\infty$. This can be proven by the dominated convergence theorem since \(\theta_k\) converges pointwise to \(\chi_{(-\infty,0)}\), $(1-\theta_k)\leq 1$, $u\in\cE(\Omega)$,
\[
\int_{\Omega\times \R} \dho (Dv_{0, \delta})= \int_{\Omega\times \R}\d\wh{1}(Dv)+\int_{\Omega\times \R}\wh{1}(0,-\delta \alpha')\dx\dt =  {|\Omega|}-\delta \int_{\Omega\times \R}\alpha'(t)\dx \dt = (1+\delta \|\alpha\|_{L^{\infty}(\R)}) {|\Omega|} <\infty\,,
\]
and also the fact that  \(q^{t}_{0,\delta}(x,t)<0\)  for \(|Dv_{0,\delta}|\)-a.e. \((x,t)\). In turn, by the above observation  together with \eqref{eq-limsup-easy},  we get the desired result.

\end{proof}

\section{Approximation on subdomains}\label{sec:localization}

\subsection{The inner approximation}\label{ssec:us-to-u}

In the three previous sections, starting from a map \(u\) defined on \(\Lambda\), we have constructed a family of maps \(v_{\ve,\delta}\) defined on a smaller set \(\Omega\Subset \Lambda\). We observe however that the definition of the maps \(v_{\ve,\delta}\) does not depend on the particular choice of \(\Omega\) (provided that \(\Omega+B(0,\ve)\subset \Lambda\)). Hence, we can deduce from Propositions~\ref{prop-liminf} and~\ref{prop-limsup} the following consequence:  
\begin{corollary} \label{coro-lim-hat-E} Suppose  $f$ satisfies Assumptions~\ref{f-red} and \eqref{eqHZconv}. Let  $v_{\ve, \delta}$ be given by~\eqref{def-v-eps-delt} with $\alpha$ satisfying~\eqref{alpha'prop} and~\eqref{eq1168}, for $u \in \cE(\Lambda)\cap L^\infty(\Lambda)$. Then for any open $\Lambda'\Subset\Lambda$  and any $\delta>0$, it holds
\[
\lim_{\ve\to 0} \int_{\Lambda'\times \R}\dhf(x,t, Dv_{\ve,\delta}) = \int_{\Lambda'\times \R}\dhf(x,t,Dv)\,.
\]    
\end{corollary}

The ultimate goal of this section is to derive from Corollary~\ref{coro-lim-hat-E} a suitable approximation of \(u\), see Proposition \ref{prop-inner-approximation} below.
To this aim, we rely on some technical lemmas from \cite{Bousquet-Pisa}. Let \(\Lambda'\)  be an open subset of \(\Lambda\) and \(w:\Lambda'\times \R\to \R\) be a bounded Borel map such that \(Dw\) is a finite \(\R^{N+1}\)-valued measure and for every \(x\in \Lambda'\), $t\mapsto w(x,t)$ is non-increasing and left-continuous.
We further assume that for every \(s\in (0,1)\), there exists \(M_s>0\) such that for all $t\geq M_s$ it holds
\begin{equation}\label{eq-Ms}
\esssup_{x\in\Lambda'}w(x, t)<s<\essinf_{x\in\Lambda'}w(x,-t).
\end{equation}
Recall the definition of the generalized inverse with respect to the second variable $w^{-1}(\cdot,s)$ introduced in \eqref{def-gen-inv}. For a.e. \(s\in (0,1)\), the function $w^{-1}(\cdot,s)$ belongs to \(BV(\Lambda')\cap L^{\infty}(\Lambda')\) and satisfies \(\|w^{-1}(\cdot,s)\|_{L^{\infty}(\Lambda')}\leq M_s\) while the indicator function \(\chi_{[w>s]}\) of the set \([w>s]=\{(x,t)\in \Lambda'\times \R : w(x,t)>s\}\) agrees with \(1_{w^{-1}(\cdot,s)}\) a.e. on \(\Lambda'\times \R\), see~\cite[Lemmas~2.4 and~2.5]{Bousquet-Pisa}. Let \(f:\Lambda'\times \R\times \R^N \to \rp\) satisfy \ref{f-red} on \(\Lambda'\). If \(\wh{E}_{\Lambda' \times \R} (w)<\infty\), then for a.e. \(s\in (0,1)\), the map \(w^{-1}(\cdot,s)\) belongs to \(W^{1,1}(\Lambda')\) while \(s\mapsto \wh{E}_{\Lambda'\times \R}(\chi_{[w>s]})\) is measurable, with
\begin{equation}
    \label{hat-E-no-hat-E}
\wh{E}_{\Lambda' \times \R} (w)=\int_{\R} \wh{E}_{\Lambda'\times \R}(\chi_{[w>s]})\d s\geq \int_{0}^{1} \wh{E}_{\Lambda'\times \R}(1_{w^{-1}(\cdot,s)})\d s= \int_{0}^{1}E_{\Lambda'}(w^{-1}(\cdot,s))\d s.
\end{equation}
The last assertion follows from \cite[Lemmas~2.6 and~8.5]{Bousquet-Pisa}. In spite of the fact that in the latter reference, those results are stated for a Lagrangian which does not depend on \(x\), the extension to the \(x\) dependent case is straightforward.\newline

The following abstract fact on the convergence follows from \cite[Lemmas~2.9 and~2.10]{Bousquet-Pisa}.
\begin{lemma}\label{lem:cvgce-usi}
Let $u\in\cE(\Lambda')\cap L^\infty(\Lambda')$. 
For every \(n\in\N\) we consider a bounded Borel map \(v_n:\Lambda'\times \R \to \R\) which is non-increasing and left-continuous with respect to the second variable and satisfies \eqref{eq-Ms} on \(\Lambda'\) with constants \(M_s\) which do not depend on \(n\). Assume further that
\(Dv_n\) is a finite \(\R^{N+1}\)-valued measure and that
\begin{equation*}
    v_n\to 1_u\qquad\text{a.e. in }\ \Lambda'\times\R\,.
\end{equation*}
Then, setting for every \(s\in (0,1)\)
\begin{equation}\label{def-usi}
  \bun(x)\coloneqq v_n^{-1}(x,s)\,   , 
\end{equation}
it holds for a.e. \(s\in (0,1)\)
\begin{equation*}
    \bun\to u\qquad\text{a.e. in }\ \Lambda'\,.
\end{equation*}
If \(\lim_{n\to\infty}\wh{E}_{\Lambda'\times\R}(v_n)=\wh{E}_{\Lambda'\times\R}(1_u)\), then for almost every  \(s\in (0,1)\) up to a subsequence (which may depend on \(s\)),  \((\bun)_{n\in\N}\subset W^{1,1}(\Lambda')\) converges to \(u\) in \(L^{1}(\Lambda')\) and satisfies
\[
\lim_{n\to \infty}E_{\Lambda'}(\bun)=E_{\Lambda'}(u)\,.
\]
\end{lemma}

For the rest of  this subsection, we require that  \(\Lambda'\Subset \Lambda\) and we fix \(\ve_0\) such that \(\Lambda'+B{(0, \ve_0)}\subset \Lambda\). We consider a map \(u\in \cE(\Lambda)\cap L^\infty(\Lambda)\).
Recall that $v_{\ve,\delta}$ is defined for every \(\ve\in (0,\ve_0)\) by~\eqref{def-v-eps-delt} and denote its generalized inverses on \(\Lambda'\) as
\begin{equation}\label{def-u-ve-del}
    u_{\ve,\delta}^s(\cdot)\coloneqq v_{\ve,\delta}^{-1}(\cdot,s)\,.
\end{equation}

\begin{lemma}
\label{lm-ved-Ms}
For every \(0<\ve<\ve_{0}\) and every \(0<\delta<1\), the  map \(v_{\ve,\delta}\) from~\eqref{def-v-eps-delt} satisfies \eqref{eq-Ms}. Moreover, \begin{equation}
\label{eq-binfty-uved}
 \|u_{\ve,\delta}^s\|_{L^{\infty}(\Lambda')}\leq M_s\,,
\end{equation}
where \(M_s\) does not depend on \(\ve\in (0,\ve_0), \delta\in (0,1)\).
\end{lemma}
\begin{proof}
Observe that \(\alpha\) is bounded as any absolutely continuous function,  and remember that  \(\lim_{t\to+\infty}\alpha(t)=0\). 
  Fix \(s\in (0,1)\). There exists \(M_s\geq M=\|u\|_{L^{\infty}(\Lambda)}\) such that for every \(t\geq M_s\), one has \(\alpha(t)<s\). Since \(1_u(x,t)=1\) for every \(t\leq -M\) and a.e. \(x\in \Lambda\), it follows that 
 $ v_{\ve, \delta}(x,t)=1 +\delta\alpha(t)>1$ for all $(x,t)\in \Lambda'\times (-\infty,-M]$. 
Using next that \(1_u(x,t)=0\) for every \(t\geq M\), we get $v_{\ve, \delta}(x,t)=\delta\alpha(t)\leq \alpha(t)<s $ for all $(x,t)\in \Lambda'\times [M_s,\infty)$. This proves that \(v_{\ve,\delta}\) satisfies \eqref{eq-Ms} with this value of \(M_s\), from which~\eqref{eq-binfty-uved} follows. 
\end{proof}

To prove  that $u_{\ve,\delta}^s$ is Lipschitz, we need the following fact which readily follows from \cite[Lemma~2.8]{Bousquet-Pisa}.

\begin{lemma}\rm \label{rk-lips-us} Suppose that $U$ is a bounded open set, $w:U\times\R\to\R$ is a bounded Borel map, which is non-increasing and left-continuous with respect to the second variable and satisfies~\eqref{eq-Ms} for $t\geq M_s$. Assume further that there exists \(\wt C>0\) such that for all $x,y\in U$ 
and all $t\in \R$, it holds
\begin{equation}\label{eq2012} w(x,t+\wt C|x-y|)\leq w(y,t)\,.
\end{equation}
Then \(w^{-1}(\cdot,s)\) is Lipschitz continuous on \(U\) and its Lipschitz rank is not larger than $\wt C$.
\end{lemma}

Let us prove that $v_{\ve, \delta}$ satisfies \eqref{eq2012}.

\begin{lemma}\label{lm-estimate-ved} Let $\alpha$ from \eqref{def:alpha} satisfy \eqref{alpha'prop}. Let $v_{\ve, \delta}$ be given by~\eqref{def-v-eps-delt} with  $u \in W^{1,1}(\Lambda)$. For every \(\delta>0\), there exists \(\overline{C}_\delta>0\) such that for every \(\ve>0\) and every  $x, y \in \Lambda'$, $t \in \R$, it holds that
\begin{equation}\label{eq-lm-estimate-ved}
    v_{\ve, \delta}\left(x, t + \tfrac{\overline{C}_\delta}{\ve}|x-y|\right) - v_{\ve, \delta}(y, t) \leq 0\,.
\end{equation} 
\end{lemma}
\begin{proof}
Remember that \(M=\|u\|_{L^{\infty}(\Lambda)}\). Let \(c_M>0\) be given by \eqref{alpha'prop} with \(\mathcal{k}_1=M\). Set
\[
\overline{C}_\delta\coloneqq\|\nabla \vr\|_{L^{1}(\R^N)}(\delta c_M)^{-1}.
\]
By definition of \(v_{\ve, \delta}\) we have
\begin{equation}\label{eq260}
v_{\ve, \delta}\left(x, t + \tfrac{\overline{C}_\delta}{\ve}|x-y|\right) - v_{\ve, \delta}(y, t)= \left(v_{\ve,0}\left(x, t + \tfrac{\overline{C}_\delta}{\ve}|x-y|\right) - v_{\ve,0}(y, t)\right) + \delta \left( \alpha\left(t + \tfrac{\overline{C}_\delta}{\ve}|x-y|\right) - \alpha(t) \right)\,.
\end{equation}
 If \(t+\frac{\overline{C}_\delta}{\ve}|x-y| > M\), then
\[
v_{\ve,0}\left(x,t+\tfrac{\overline{C}_\delta}{\ve} |x-y|\right)=\int_{\{u\geq t+\frac{\overline{C}_\delta}{\ve}|x-y|\}} \vr_{\ve}(x-z)\d z =0\,.
\]
Since \(v_{\ve,0}(y, t) \geq 0\) and \(\alpha\) is decreasing, this implies \eqref{eq-lm-estimate-ved} in this case. Similarly, if \(t<-M\), then 
\[
v_{\ve,0}(y,t)= \int_{\rn}\vr_{\ve}(y-z)\d z =1 \geq v_{\ve,0}\left(x, t + \tfrac{\overline{C}_\delta}{\ve}|x-y|\right)\,,
\]
and \eqref{eq-lm-estimate-ved} holds in this case as well. In the rest of the proof, we thus assume that
\begin{equation}\label{eq273}
-M\leq t \leq t+\tfrac{\overline{C}_\delta}{\ve}|x-y|\leq M\,.
\end{equation}
Since \(v_{\ve,0}(x,\cdot)\) is non-increasing, we infer that
\begin{flalign}
  \nonumber  v_{\ve,0}\left(x, t + \tfrac{\overline{C}_\delta}{\ve}|x-y|\right) - v_{\ve,0}(y, t)
&\leq v_{\ve,0}(x, t) - v_{\ve,0}(y, t)=\int_{\{u\geq t\}}(\vr_{\ve}(x-z)-\vr_{\ve}(y-z))\d z\\
\nonumber &\leq \left\langle \int_{\{u\geq t\}}\left(\int_{0}^{1} \nabla \vr_{\ve}(sx+(1-s)y-z) \d s \right)\d z , x-y\right\rangle\\
\label{eq278}
&\leq |x-y|\int_{0}^{1}\left(\int_{\{u\geq t\}}|\nabla \vr_{\ve}(sx+(1-s)y-z)|\d z \right)\d s\,,\end{flalign}
where the last line holds due to the Fubini theorem and the Cauchy--Schwarz inequality. We then estimate the innermost integral as follows
\begin{align*}
\int_{\{u\geq t\}}|\nabla \vr_{\ve}(sx+(1-s)y-z)|\d z &\leq \int_{\rn}|\nabla \vr_{\ve}(sx+(1-s)y-z)|\d z
 = \int_{\rn}|\nabla \vr_{\ve}(z)|\d z  =\tfrac{1}{\ve}\|\nabla \vr\|_{L^{1}(\rn)}\,.\end{align*}
Inserting the above inequality into \eqref{eq278}, we get
\begin{equation}\label{eq287}
v_{\ve,0}\left(x, t + \tfrac{\overline{C}_\delta}{\ve}|x-y|\right) - v_{\ve,0}(y, t)
\leq \tfrac{1}{\ve}\|\nabla \vr\|_{L^{1}(\rn)}|x-y|\,.
\end{equation}
As for the second term in \eqref{eq260}, taking into account \eqref{eq273}, we simply write
\begin{equation}\label{eq292}
\delta \left( \alpha\left(t + \tfrac{\overline{C}_\delta}{\ve}|x-y|\right) - \alpha(t) \right)\leq \tfrac{\overline{C}_\delta}{\ve}\delta |x-y|(-c_M)=-\tfrac{1}{\ve}\|\nabla \vr\|_{L^{1}(\R^N)}|x-y|\,.
\end{equation}
By adding \eqref{eq287} and \eqref{eq292}, we can conclude that \eqref{eq-lm-estimate-ved} holds true.
\end{proof}
The above lemma, together with Lemma~\ref{rk-lips-us}, implies that \(u_{\ve,\delta}^s\) is Lipschitz continuous of rank \(\tfrac{\overline{C}_\delta}{\ve}\).

\begin{proposition}\label{prop-inner-approximation} Suppose  $f$ satisfies Assumptions~\ref{f-red} and \eqref{eqHZconv} on \(\Lambda\). Let $u \in \cE(\Lambda)\cap L^\infty(\Lambda),$ \(v_{\ve, \delta} \) be given by~\eqref{def-v-eps-delt} with $\alpha$ from \eqref{def:alpha} satisfying~\eqref{alpha'prop} and ~\eqref{eq1168}. Then, for every open set \(\Lambda'\Subset \Lambda\),  for almost every $s\in (0,1)$, there exist sequences $(\ve_n,\delta_n)_{n\in\N}$ converging to $(0,0)$ and \((u^s_{\ve_n,\delta_n})_{n\in\N}\subset W^{1,\infty}({\Lambda'})\) as in~\eqref{def-u-ve-del},
which is bounded in \(L^{\infty}({\Lambda'})\) and such that  $u^s_{\ve_n,\delta_n} \to u$ in $L^{1}({\Lambda'})$ as $n \to \infty$ and
\begin{equation}\label{eq-main-inner}
\lim_{n\to \infty}E_{{\Lambda'}}(u^s_{\ve_n,\delta_n})=E_{{\Lambda'}}(u) \,.
\end{equation} 
\end{proposition}

\begin{proof} 
By Corollary~\ref{coro-lim-hat-E}, we have
\(\lim_{\ve\to 0}\wh{E}_{{\Lambda'}\times\R}(v_{\ve, \delta})=\wh{E}_{{\Lambda'}\times\R}(1_u)\), and thus
\[
\lim_{\delta\to 0}\lim_{\ve\to 0}\wh{E}_{{\Lambda'}\times\R}(v_{\ve, \delta})=\wh{E}_{{\Lambda'}\times\R}(1_u)\,.
\]
Given any \((\delta_n)_{n\in\N}\) decreasing to \(0\), one can find a sequence  \((\ve_n)_{n\in\N}\) also decreasing to \(0\) and such that
\[
\lim_{n\to \infty}\wh{E}_{{\Lambda'}\times\R}(v_{\ve_n, \delta_n})= \wh{E}_{{\Lambda'}\times\R}(1_u)\,.
\]
We also have \(v_{\ve_n, \delta_n}\to 1_u\) in \(L^{1}_{loc}({\Lambda'}\times \R)\) and thus, up to a subsequence, also a.e. in \({\Lambda'}\times \R\). By Lemma~\ref{lem:cvgce-usi} for a.e. $s\in(0,1)$ there exists a subsequence of $((\ve_n,\delta_n))_{n\in\N}$, such that  $u^s_{\ve_n,\delta_n} \to u$ in $L^{1}({\Lambda'})$ as $n \to \infty$ and~\eqref{eq-main-inner} holds true.
 It follows from \eqref{eq-binfty-uved} that the sequence \((u^s_{\ve_n,\delta_n})_{n\in \N}\) is bounded in \(L^{\infty}(\Lambda')\). 
 By Lemma~\ref{lm-estimate-ved} and Lemma~\ref{rk-lips-us}, each  \(u^s_{\ve_n,\delta_n}\) is Lipschitz continuous on \({\Lambda'}\). 
\end{proof}

The above proposition provides an \emph{inner} approximation, in the sense that the sequence \((u^s_{\ve_n,\delta_n})_{n\in \N}\) which converges to \(u\) on \(\Lambda'\) does not necessarily
agree with \(u\) on \(\partial \Omega\). In the next section, we modify this approximating sequence to comply with this additional boundary constraint, but only for special domains \(\Omega\) and just on one part of the boundary. We completely solve this issue in Section~\ref{ssec:proof-reduced}.

\subsection{The boundary approximation on a special open set}\label{ssec:boundary-approx}
In this paragraph, we assume that \(N\geq 2\) and we explain in Remark~\ref{rk-approx-bd-N=1} below how to adapt the arguments to the  (simpler) case \(N=1\).  Denote by \(B'(x',r)\) the ball in $\R^{N-1}$, of centre \(x'\in \R^{N-1}\) and radius \(r>0\).  
Let \(r,b>0\) and \(\psi:\R^{N-1}\to (-b/2, b/2)\) a Lipschitz function. 
We define the set $\Sigma\subset\rn$ by the following formula
\[ 
\Sigma\coloneqq\{(x',x_N)\in B'(0,r) \times (-b,b) :\ \psi(x')<x_N\}\,.
\]
We also define the lower boundary of \(\Sigma\) as: 
\[
\Gamma\coloneqq\{(x',\psi(x')) :\ x'\in B'(0,r) \}\,.
\]

\begin{lemma}\label{lm-approx-epigraph} Given   a bounded open set  \(\Lambda\Supset \Sigma\),  let \(f:\Lambda\times \R\times \rn \to \rp\) satisfy Assumptions~\ref{f-red} and \eqref{eqHZconv} on $\Lambda$,  \(\vp:\Lambda\to \R\) a Lipschitz continuous function  and $u\in \cE(\Lambda)\cap L^\infty(\Lambda)$. We assume that \(u=\varphi\) on \(\{(x',x_N)\in \Lambda:x_N\leq \psi(x')\}\). 
Then for almost every $s\in(\frac{1}{4},1)$ there exists  \((\bunp) _{n\in\N}\subset W^{1,\infty}(\Sigma)\) which is bounded in \(L^{\infty}(\Sigma)\),  converges to \(u\) in \(L^{1}(\Sigma)\), coincides with \(\vp\) on \(\Gamma\), and such that
\[
\lim_{n\to \infty}\int_{\Sigma}f(x,\bunp,\nabla \bunp)\dx=  \int_{\Sigma}f(x,u,\nabla u)\dx\,.
\]
\end{lemma}
\begin{proof} 
We start with defining for every \(\vk \in (0, b/2)\), the subset
\[
\Sigma_\vk: = \{(x',x_N)\in B'(0,r)  \times (-b,b) : \ \psi(x')<x_N<\psi(x')+\vk\}\,.
\]
Let \(\alpha\in W^{1,1}(\R)\) be a bounded positive function as in~\eqref{def:alpha} which satisfies  ~\eqref{alpha'prop} and ~\eqref{eq1168}. 
We apply Proposition~\ref{prop-inner-approximation} with \(\Sigma\) playing the role of \(\Lambda'\) (here, we need that \(f\) and \(u\) be defined on the  larger set \(\Lambda\)). For a.e. \(s\in (0,1)\), we get sequences \(((\ve_n, \delta_n))_{n\in \N}\) converging to \((0,0)\) which leads to define
\[
\bun=u^s_{\ve_n,\delta_n}(\cdot)=v^{-1}_{\ve_n,\delta_n}(\cdot,s).
\]
Note that each $\bun$  is Lipschitz continuous on \(\Sigma\), $\bun \to u$ in $L^{1}(\Sigma)$ as $n \to \infty$, $\sup_{n \in \N} \|\bun\|_{L^{\infty}(\Sigma)} < \infty$, and  
\begin{equation}
\label{eq2330}
\lim_{n\to \infty}\int_{\Sigma}f(x,\bun, \nabla \bun)\dx =\int_{\Sigma}f(x,u, \nabla u)\dx \,.
\end{equation}

\noindent {\bf{Step 1. De-centered convolution. }}
Our aim is to show that by using the trick of de-centered convolution, one can ensure the smallness  condition: $ \|\bun-\vp\|_{L^{\infty}(\Sigma_{\vk_n})}\leq c\vk_n$ for every $n\in \N$ sufficiently large, where \(\vk_n\coloneqq\ve_n/8\).

Actually, when defining \(v_{\ve_n,\delta_n}\), we choose the regularization kernel with the additional requirement that
\[
\supp  \vr \Subset B'(0,\tfrac{1}{8L_{\psi}})\times (\tfrac{1}{4}, \tfrac{3}{4})\subset\rn\,,
\] 
where \(L_{\psi}\geq 1\) is any fixed number not lower than the Lipschitz rank of \(\psi\). There exists \(n_0\in \N\) such that for every \(n\geq n_0\), the set \(\Sigma_{\vk_n}+(B'(0,\frac{\ve_n}{8L_{\psi}})\times (-\frac{3\ve_n}{4}, -\frac{\ve_n}{4}))\) is contained in \(\Lambda\). Let \(n\geq n_0\) and  \(x=(x',x_N)\in \Sigma_{\vk_n}\).
 Then for every \(y=(y',y_N)\in  B'(0,\frac{\ve_n}{8L_{\psi}})\times (\frac{\ve_n}{4}, \frac{3\ve_n}{4})\), one has \(x-y\in \Lambda\) and 
\[
\psi(x'-y')\geq \psi(x')-L_{\psi}\tfrac{\ve_n}{8L_{\psi}}{\geq} \psi(x')- \tfrac{\ve_n}{8} \geq  
x_N-y_N \,,  
\]
where the last inequality relies on the fact that \(x_N\leq \psi(x')+\vk_n\) and \(y_N\geq \ve_n/4\). This implies that  $x-y$ belongs to the intersection of \(\Lambda\) with the hypograph of \(\psi\) where \(u\) coincides with \(\varphi\). This yields for every \(t\in \R\),
\begin{equation}\label{eq3029}
{1}_u*_x \vr_{\ve_n}(x,t)=\int_{B'(0,{\frac{\ve_n}{8L_{\psi}}})\times (\frac{\ve_n}{4}, \frac{3\ve_n}{4})} {1}_u(x-y, t)\vr_{\ve_n}(y)\dy = \int_{\R^{N}} {1}_\vp(x-y,t)\vr_{\ve_n}(y)\dy  =  {1}_\vp*_x \vr_{\ve_n}(x,t)\,.
\end{equation}
In particular,  denoting by \(L_{\vp}\) the Lipschitz constant of \(\vp\) on \(\Lambda\), one has for every \(x, \wt{x}\in \Sigma_{\vk_n}\),
\begin{align*}
v_{\ve_n, \delta_n}(x,t+L_\vp|x-\wt{x}|)&=
 {1}_u*_x \vr_{\ve_n}(x,t+L_\vp|x-\wt{x}|)+\delta_n\alpha(t+L_\vp|x-\wt{x}|)\\
 &= {1}_\vp*_x \vr_{\ve_n} (x,t+L_\vp|x-\wt{x}|)+\delta_n\alpha(t+L_\vp|x-\wt{x}|)\,.
\end{align*}
Using that \(L_{\vp}\) is the Lipschitz constant of \(\vp\),  for every \(y \in B'(0,\frac{\ve_n}{8L_{\psi}})\times (\frac{\ve_n}{4}, \frac{3\ve_n}{4})\) and \(t\in \R\) it holds ${1}_\vp (x-y,t+L_\vp|x-\wt{x}|)\leq {1}_\vp (x-y,t)$. Together with the fact that \(\alpha\) is non-increasing, one obtains
\[
v_{\ve_n, \delta_n}(x,t+L_\vp|x-\wt{x}|)\leq 
{1}_\vp*_x \vr_{\ve_n} (\wt{x},t)+\delta_n\alpha(t)
= v_{\ve_n, \delta_n}(\wt{x},t)\,.
\]

In  view of Lemma~\ref{rk-lips-us}, this  ensures that 
\begin{equation}\label{grad-bun-Lip}
    \|\nabla \bun\|_{L^{\infty}(\Sigma_{\vk_n})}\leq L_\vp\,.
\end{equation}
Finally, using \eqref{eq3029} again, for every \(x\in \Sigma_{\vk_n}\), we have
\[ {1}_u *_x\vr_{\ve_n}(x, \vp(x)+L_\vp\ve_n)=  {1}_\vp *_x\vr_{\ve_n}(x, \vp(x)+L_\vp\ve_n)
=0\ \,,
\]
where the last equality holds true by the very definition of $1_\vp$ and since \(\vp(x)+L_\vp\ve_n > \vp(x-y)\) for every \(y\in B'(0,\frac{\ve_n}{8L_{\psi}})\times (\frac{\ve_n}{4}, \frac{3\ve_n}{4})\).
It follows that 
\[
v_{\ve_n, \delta_n} (x, \vp(x)+L_\vp\ve_n) =\delta_n\alpha(\varphi(x)+L_\varphi \varepsilon_n).
\]
In particular, for every  \(s\in (\frac 14, 1)\) and every \(n\) so large that \(\delta_n\leq \frac {1}{4\|\alpha\|_{L^{\infty}}}\), the right-hand side of the above equality is not larger than \(s\) and this implies that 
\(
\bun(x)\leq \vp(x)+L_\vp\ve_n\,.
\) Similarly we can prove that
\(
\bun(x)\geq \vp(x)-L_\vp\ve_n\,.
\)
We can conclude that 
\begin{equation}
\label{eq2045}
 \|\bun-\vp\|_{L^{\infty}(\Sigma_{\vk_n})}\leq L_\vp\ve_n=8 L_\vp \vk_n\,.   
\end{equation}

\noindent {\bf{Step 2. Truncation argument. }} Let \(\phi_n \in C^{\infty}(\overline{\Sigma}, [0,1])\) be  
converging pointwise to \(\chi_{\overline{\Sigma}\setminus \overline{\Gamma}}\) when \(n\to \infty\). We further require that 
\begin{equation}\label{eq1125}
\phi_n\equiv 1 \textrm{ on } \overline{\Sigma}\setminus \overline{\Sigma_{\vk_n}}, \qquad \phi_n\equiv 0 \textrm{ on } \Gamma \qquad\text{and} \qquad \|\nabla \phi_n\|_{L^{\infty}(\Sigma)}\leq \tfrac{C}{\vk_n}\,,
\end{equation}
for some universal constant \(C>0\).
We then define
\[
\bunp\coloneqq\phi_n \bun + (1-\phi_n)\vp\,.
\]
By construction \(\bunp \in W^{1,\infty}(\Sigma)\)  and \(\bunp=\vp\) on \(\Gamma\).
Observe also that the sequence \((\bunp)_n\) is bounded in \(L^{\infty}(\Sigma)\).
Moreover, \(\bunp-\bun=(1-\phi_n)(\vp-\bun)\), so that
\[
\|\bunp-\bun\|_{L^{1}(\Sigma)}\leq  \|\vp-\bun\|_{L^{\infty}(\Sigma_{\vk_n})}|\Sigma_{\vk_n}| \,,
\]
and by \eqref{eq2045}, the left-hand side converges to \(0\). Since \((\bun)_{n\in\N}\) converges to \(u\) in \(L^{1}(\Sigma)\), this implies that
\[
\lim_{n\to +\infty}\|\bunp-u\|_{L^{1}(\Sigma)}=0.
\]
Let us show that
\begin{equation}
    \label{conv-of-Phi}
\lim_{n\to \infty}\int_{\Sigma}f(x, \bunp, \nabla \bunp)\dx =\int_{\Sigma}f(x,u,\nabla u)\dx \,.
\end{equation}
  
We notice that
\begin{multline*}
\int_{\Sigma}|f(x, \bunp, \nabla \bunp)-f(x,\bun,\nabla \bun)|\dx  
\\=\int_{\Sigma_{\vk_n}}\left|f\Big(x, \phi_n \bun + (1-\phi_n)\vp, \phi_n \nabla \bun+ (1-\phi_n)\nabla \vp + (\bun-\vp)\nabla \phi_n\Big)-f(x,\bun,\nabla \bun)\right|\dx \,.
\end{multline*}
By \eqref{grad-bun-Lip} we have \(\|\nabla \bun\|_{L^{\infty}(\Sigma_{\vk_n})}\leq L_\vp\) and \eqref{eq2045} together with \eqref{eq1125} ensure that \(\|(\bun-\vp)\nabla \phi_n\|_{L^{\infty}(\Sigma_{\vk_n})}\leq c\vk_n \frac{C}{\vk_n}\leq cC\). Since \(f\) is bounded on bounded sets (see Lemma \ref{lem-bdd-f}) and \(\lim_{n\to +\infty}|\Sigma_{\vk_n}|=0\), one gets
\[
\lim_{n\to \infty}\int_{\Sigma}|f(x, \bunp, \nabla \bunp)-f(x,\bun,\nabla \bun)|\dx  =0\,,
\]
which, together with \eqref{eq2330}, implies \eqref{conv-of-Phi}.
\end{proof}

\begin{remark}[Boundary approximation in the one dimensional case]\rm\label{rk-approx-bd-N=1}
In the case $N=1$, a special open set is simply an interval \((a,b)\) and a neighbourhood of the boundary can be chosen as a (small) interval near the extremities \(a\) or \(b\). Let us focus for instance on the approximation near \(a\). We replace the set \(\Sigma\) by \((a,b)\). In analogy to the  higher dimensional case considered above, this amounts to taking \(\psi\) equal to the constant \(a\) and \(\Gamma=\{a\}\). We also replace \(\R^{N-1}\times \R\) by \(\R\). Accordingly, in Step 1, we remove all the \(N-1\) dimensional balls.  The rest of the proof is essentially the same, so we omit the details.
\end{remark}

\subsection{Gluing the local approximations}\label{ssec:proof-reduced}

\begin{proof}[Proof of Proposition~\ref{prop-main-reduced}] 
We proceed step by step.\newline
    
\noindent{\bf Step 1. Covering and partition of unity. } We consider a finite family of open cubes $\{U_j\}_{0\leq j< J}$ such that 
\[
U_0\Subset \Omega\Subset  \bigcup_{j=0}^JU_j\Subset\Lambda, \qquad \partial\Omega \subset \bigcup_{j=1}^JU_j \,,
\] and  for every \(1\leq j\leq J\), the set \(U_j\cap \Omega\) is isometric to the epigraph of a Lipschitz function. Namely, we assume that there exist an affine isometry \(\sigma_j:\rn\to \rn\), \(r_j, b_j>0\) and a Lipschitz map \(\psi_j:\ \R^{N-1}\to (-b_j/2, b_j/2)\) such that
\(
\sigma_j(U_j)= B(0,{r_j})\times (-b_j, b_j), 
\)
and 
\[
\sigma_j(\Omega) \cap  \left( B(0,2{r_j})\times (-2b_j, 2b_j)\right) = \{(y,t) :\  t>\psi_j(y)\}\cap \left(  B(0,2{r_j})\times (-2b_j, 2b_j)\right)\,,
\]
\[
\sigma_j(\partial\Omega) \cap   \left(B(0,2{r_j})\times (-2b_j, 2b_j)\right) = \{(y,t) : t=\psi_j(y)\}\cap  \left( B(0,2{r_j})\times (-2b_j, 2b_j)\right)\,.
\]
 We also introduce a partition of unity  
 \(\{\Psi_j\}_{0\leq j \leq J}\) on \(\Lambda\) subordinate to the covering \((U_j)_{0\leq j \leq J}\), that is, \[\Psi_j\in C^{\infty}_c(U_j,[0,1])\qquad\text{and}\qquad \sum_{j=0}^{J}\Psi_j\equiv 1\  \text{ on }\   \overline{\Omega}\,.\]

\noindent{\bf Step 2. Penultimate approximation. } 
We extend \(u\in\cEp(\Omega)\cap L^{\infty}(\Omega)\) by \(\vp\) on \(\rn\), recall that \(M=\|u\|_{L^{\infty}(\rn)}\), and denote \[\Sigma_j\coloneqq\Omega\cap U_j\,.\] 
Note that \(f\) satisfies Assumptions~\ref{f-red} and \eqref{eqHZconv} on \(\Lambda\times \R\times \R^N\) and that \(u\in \cE(\Lambda)\cap L^\infty(\Lambda)\).
Moreover, for every \(1\leq j\leq J\), \(u=\varphi\) on the set \(\Lambda_j\setminus \Omega\) where
\[
\Lambda_j\coloneqq \Lambda\cap \sigma_{j}^{-1}\left(  B(0,2{r_j})\times (-2b_j, 2b_j)\right)\,.
\]
Observe that \(\Sigma_j\Subset \Lambda_j\) and \(\Lambda_j\setminus \Omega\) is the intersection of \(\Lambda_j\) with the hypograph of \(\psi_j\).
Hence, we can rely on Proposition~\ref{prop-inner-approximation}  on \(U_0\) and Lemma~\ref{lm-approx-epigraph} on each \(\Sigma_j\), for \(1\leq j \leq J\) (applied with \(\Lambda_j\) instead of \(\Lambda\)) to deduce that there exist \(M'>0\) and maps  \(\bunpj\in W^{1,\infty}(\Sigma_j)\) for \(0\leq j \leq J\)  such that
\begin{equation}\label{eq1900}
 \|\bunpj\|_{L^{\infty}(\Sigma_j)}\leq M'\,,\quad \bunpj\xrightarrow[n\to\infty]{} u\ \text{ in }\ L^{1}(\Sigma_j)\,,\quad \text{ and }\quad \lim_{n\to\infty} E_{\Sigma_j}(\bunpj)=E_{\Sigma_j}(u)\,,
\end{equation}
and for \(1\leq j\leq J\), each \(\bunpj\) agrees with \(\vp\) on \(U_j\cap \partial \Omega\). We are thus entitled to extend each \(\bunpj\) by \(\vp\) on \(U_j\setminus \Omega\), so that each property in \eqref{eq1900} holds true on \(U_j\) instead of \(\Sigma_j\).

Define $\overline{\alpha}:\R\to (0,\infty)$ via
\begin{equation}
    \label{def-o-alpha}
\overline{\alpha}(t)\coloneqq
\begin{cases}
    1 & \textrm{ if } t\leq -M'\,,\\
    \frac{M'-t}{2M'} & \textrm{ if } -M'\leq t\leq M'\,,\\
    0 & \textrm{ if } t\geq M'\,.
\end{cases}
\end{equation}
Finally, for every \(n\in\N,\ \odel>0\), one defines for every \((x,t)\in \rn\times \R\),
\[
\Pnd(x,t)\coloneqq \sum_{j=0}^{J}\Psi_j(x)1_{\bunpj}(x,t) + \odel \overline{\alpha}(t)\,.
\]
Note that even if \(\bunpj\) is not defined on the whole \(\rn\), the product \(\Psi_j(x)1_{\bunpj}(x,t)\) is well-defined on \(\rn\), since \(\Psi_j\) is compactly supported in \(U_j\).
For every \(\odel>0\), we also define the map
\[
\Phi_{\infty,\odel}(x,t)\coloneqq\sum_{j=0}^{J}\Psi_j(x) {1}_{u}(x,t)+\odel \overline{\alpha}(t)\,.
\]
On \(\overline{\Omega}\times \R\), one has \(\Phi_{{\infty},\odel}(x,t)=1_u(x,t)+\overline{\delta}\overline{\alpha}(t)\).
 For every \(x\in \rn\), the map  \(t\mapsto \Pnd(x,t)\) is non-increasing and left-continuous. Moreover, \(\Pnd\in L^{\infty}(\rn\times \R)\) and since \(\max_{0\leq j \leq J}\|\bunpj\|_{L^{\infty}(U_j)}\leq M'\), for every \(x\in \overline{\Omega}\) and every \(t\geq M'\), we have
\begin{equation}\label{eq1904}
\Pnd(x,-t)=1+\odel \qquad\text{and}\qquad \Pnd(x,t)=0\,. 
\end{equation}

The same arguments leading to \cite[Lemma 6.3]{Bousquet-Pisa} imply that  \(D\Pnd = \lambda_n+h_{n, \odel}\dxodt \), where
\begin{flalign*}
\lambda_n(x,t) &\coloneqq  \sum_{j=0}^{J}\Psi_j(x)D1_{\bunpj}(x,t)\qquad\text{and}\qquad
h_{n,\odel}(x,t)\coloneqq\left( \sum_{j=0}^{J}\nabla \Psi_j(x) 1_{\bunpj}(x,t),\odel \overline{\alpha}'(t)\right)\,
\end{flalign*}
and $h_{n,\odel}(x,t)=(0,0)$ for every $x\in\Omega$ and $|t|>M'$. Furthermore, for every \(\odel_0>0\) it holds
\begin{equation}\label{eq2645}
\sup_{n\in\N}|\lambda_n|(\rn\times \R)<\infty \qquad\text{and}\qquad
\sup_{n\in\N,\, \odel<\odel_0} \|h_{n, \odel}\|_{L^{1}(\Omega \times \R)}<\infty\,.
\end{equation}

Let us establish some additional properties of $\Pnd$.
\begin{enumerate}[{\it ($v$-a)}]
\item {\it (Limits of \(\Pnd\)).} We observe that for every \(w_1, w_2 \in L^{1}(\Omega)\),
\[
\int_{\Omega\times \R}|1_{w_1}-1_{w_2}|\dx\dt = \int_{\Omega}|w_1-w_2|\dx.
\]
Using that \(\lim_{n\to \infty}\|\bunpj-u\|_{L^{1}(U_j)}=0\),  for every \(\overline{\delta}>0\), we have  that
\(\lim_{n \to \infty}\|\Pnd - \Phi_{ {\infty},\odel}\|_{L^{1}(\Omega\times \R)} = 0\). 
%
%
%
%
\item {\it (Limit of $\wh{E}_{\Omega\times\R}(\Phi_{n, \overline{\delta}})$).}  We show that 
for every \(\overline{\delta}>0\), $
\lim_{n\to \infty}\wh{E}_{\Omega\times\R}(\Phi_{n, \overline{\delta}})
=E_{\Omega}(u)\, $.

By using that the measures \(\lambda_n\) and \(h_{n, \odel}\dxodt \) are mutually singular, \eqref{def-o-alpha}, and the fact  that \(h_{n,\odel}\)   vanishes on \(\Omega\times (\R\setminus (-M',M'))\), one gets
\begin{align}
    \int_{\Omega\times \R}&\dhf(x,t,D\Pnd)=\int_{\Omega\times \R}\d  \wh{f}(x,t,\lambda_n)+\int_{\Omega\times \R}\dhf(x,t,h_{n,\odel}\dxodt )\nonumber\\
    =&\int_{\Omega\times \R}\dhf\left(x,t, \sum_{j=0}^{J}\Psi_j D1_{\bunpj}\right)+\frac{\odel}{2M'}\int_{-M'}^{M'}\int_{U_j}f\left(x,t,\frac{2M'}{\odel}\left[ \sum_{j=0}^{J}\nabla \Psi_j 1_{\bunpj}\right]\right)\dx \dt \nonumber\\
    \leq &\sum_{j=0}^{J}\int_{\Omega}\Psi_j f (x,\bunpj, \nabla \bunpj)\dx  +\frac{\odel}{2M'}\int_{-M'}^{M'}\int_{\Omega}f\left(x,t,\frac{2M'}{\odel}\left[ \sum_{j=0}^{J}\nabla \Psi_j 1_{\bunpj}\right]\right)\dx \dt
    \eqqcolon \mathcal{K}_{1,n}^\Omega+\mathcal{K}_{2,n}^\Omega\, ,\label{eq1925}
\end{align}
where the last inequality relies on the convexity of \(\wh{f}\) with respect to the last variable and~\eqref{eq:f-f-hat-general}.

Since \(\lim_{n\to \infty}\|\bunpj-u\|_{L^{1}(U_j)}=0\),  we have  
\[ 
\lim_{n\to \infty}\left\|\sum_{j=0}^{J}\nabla \Psi_j 1_{\bunpj} -\sum_{j=0}^{J}\nabla \Psi_j 1_{u}\right\|_{L^{1}(\Omega)}=0.
\]
Since \(\sum_{j=0}^{J}\nabla \Psi_j 1_{u}=1_{u}\nabla \left(\sum_{j=0}^{J}\Psi_j\right)=0\) on \(\Omega\times \R\), it follows that $\lim_{n\to \infty}\left\|\sum_{j=0}^{J}\nabla \Psi_j 1_{\bunpj}\right\|_{L^{1}(\Omega)}=0$. By the dominated convergence theorem, we get
\begin{equation}\label{eq:limK2n}
\mathcal{K}_{2,n}^\Omega\xrightarrow[n\to \infty]{}\frac{\odel}{2M'}\int_{-M'}^{M'}\int_{\Omega}f(x,t,0)\dx \dt =0\,.
\end{equation}
On the other hand, to pass to the limit in \(\mathcal{K}_{1,n}^\Omega\), we start from the observation that for every \(0\leq j \leq J\), we have
\(\lim_{n\to +\infty}E_{U_j}(\bunpj)=E_{U_j}(u)\) and  \(\lim_{n\to \infty}\|\bunpj-u\|_{L^{1}(U_j)}=0\). Using that \(f\) is superlinear with respect to the last variable, we deduce from the Dunford-Pettis theorem that 
\((\bunpj)_{n\in \N}\) weakly converges to \(u\) in \(W^{1,1}(U_j)\). By weak lower semicontinuity, it follows that for every open subset \(A\subset U_j\), one has
\begin{equation}\label{eq3234}
\liminf_{n\to \infty} \int_A f(x,\bunpj, \nabla \bunpj)\dx \geq \int_A f(x,u, \nabla u)\dx. 
\end{equation}
Applying \eqref{eq3234} with \(A=U_j\setminus \overline{\Omega}\), we get
\[
\limsup_{n\to \infty}E_{U_j\cap \Omega}(\bunpj)=\lim_{n\to \infty}E_{U_j}(\bunpj)-\liminf_{n\to \infty}E_{U_j\setminus \overline{\Omega}}(\bunpj)\leq E_{U_j}(u)-E_{U_j\setminus \overline{\Omega}}(u)=E_{U_j\cap \Omega}(u).
\]
Relying on \eqref{eq3234}  with \(A=U_j\cap \overline{\Omega}\), this implies that \(\lim_{n\to \infty}E_{U_j\cap \Omega}(\bunpj)=E_{U_j\cap \Omega}(u)\). 

Hence, one can apply \cite[Proposition 1.80]{AmFuPa} to the sequence of finite Borel measures on \(U_j\cap \Omega\) defined by
\(
\mu_n(A)=\int_{A}f(x,\bunpj, \nabla \bunpj)\dx\),
to conclude that
\[
\lim_{n\to \infty}\int_{U_j\cap \Omega}\Psi_j(x)f(x, \bunpj(x), \nabla \bunpj(x))\dx = \int_{U_j\cap \Omega}\Psi_j(x)f(x, u(x), \nabla u(x))\dx. 
\]
Thus, using~\eqref{eq1925} and~\eqref{eq:limK2n}, we get
\begin{equation*}
    \limsup_{n \to \infty} \wh{E}_{\Omega\times\R}(\Pnd) \leq  \sum_{j=0}^{J} \int_{U_j\cap \Omega}\Psi_j(x)f(x, u(x), \nabla u(x))\dx = E_{\Omega}(u)\,,
\end{equation*}
where in the last equality we used that $\sum_{j=0}^J\Psi_j(x)=1$ and that $\Psi_j$ equals $0$ outside $U_j$.
%
On the other hand, using that \(\Pnd\) converges to \(\Phi_{ {\infty},\odel}\) in \(L^{1}(\Omega\times \R)\) and the fact that \(\sup_n|D\Pnd|(\Omega\times \R)<\infty\) by \eqref{eq2645}, we deduce that \(D\Pnd\) weakly-$*$ converges to \(D\Phi_{ {\infty},\odel}\). The Reschetnyak
lower semicontinuity theorem thus implies that
\[
\liminf_{n\to \infty}\wh{E}_{\Omega\times\R}( \Pnd)\geq \wh{E}_{\Omega\times\R}(\Phi_{ {\infty},\odel})=E_{\Omega}(u). 
\]
It follows that \(\lim_{n\to \infty}\wh{E}_{\Omega\times\R}( \Pnd)=E_{\Omega}(u)\).

\end{enumerate}

\noindent{\bf Step 3. Construction of the final approximation. }  For every \(\overline{\delta}\in (0,\infty)\),  property {\it ($v$-b)} implies that \(\lim_{n\to \infty}\wh{E}_{\Omega\times\R}(\Pnd)=E_{\Omega}(u)\) while  by property {\it ($v$-a)}, \((\Pnd)_{n\in\N}\) converges to \(1_u+\overline{\delta}\overline{\alpha}\) in \(L^{1}(\Omega\times (-M',M'))\). Hence, 
\[
\lim_{\overline{\delta}\to 0}\lim_{n\to \infty}\wh{E}_{\Omega\times\R}(\Pnd)=E_{\Omega}(u) \textrm{ and } \lim_{\overline{\delta}\to 0}\lim_{n\to \infty} \|\Pnd-(1_u+\overline{\delta}\overline{\alpha})\|_{L^{1}(\Omega\times (-M',M'))}=0.
\]
Given any sequence \((\overline{\delta}_i)_{i\in \N}\)  {tending} to \(0\), we can thus find a subsequence \((n_i)_{i\in \N}\) such that
\begin{equation}\label{eq3262}
 {\sup_{i \in \mathbb{N}}\wh{E}_{\Omega\times\R}(\Phi_{n_i,\overline{\delta}_i}) < \infty},\quad \lim_{i\to +\infty} \wh{E}_{\Omega\times\R}(\Phi_{n_i,\overline{\delta}_i})=E_{\Omega}(u)\,,\ \textrm{ and }\ \lim_{i\to +\infty}\|\Phi_{n_i,\overline{\delta}_i}- 1_u\|_{L^{1}(\Omega\times (-M',M'))}=0 \, .
\end{equation}
Up to an extraction, one can assume that \(\lim_{i\to +\infty}\Phi_{n_i,\overline{\delta}_i}(x,t)=1_u(x,t)\) for a.e.  \((x,t)\in \Omega\times (-M',M')\). 
For every \(s\in (0,1)\), we introduce the maps $\unp$ with the use of the generalized inverse with respect to the second variable of $\Phi_{n_i, \overline{\delta}_i}$ constructed in Step~2, i.e., 
\[
\unp(x)\coloneqq\Phi_{n_i, \overline{\delta}_i}^{-1}(x,s) \,.
\]
We establish some properties of $\unp$ following the same steps as in Section~\ref{ssec:us-to-u}. 
\begin{enumerate}[{\it ($u$-a)}]
\item {\it (Boundedness of (\(\unp\)))}.
By \eqref{eq1904} and~\eqref{def-gen-inv} we infer that \(\unp\) belongs to \(L^{\infty}(\Omega)\), with \(\|\unp\|_{L^{\infty}(\Omega)}\leq M'\).
By the fact that \(\wh{E}_{\Omega\times\R}(\Phi_{n_i, \overline{\delta}_i})<\infty\) and~\eqref{hat-E-no-hat-E}, one gets that for a.e. $s$ it holds \(\unp\in W^{1,1}(\Omega)\) and \(E_\Omega(\unp)<\infty\).

\item {\it (Limits of (\(\unp\)))}. From Lemma~\ref{lem:cvgce-usi} with \(\Lambda'=\Omega\) and \eqref{eq3262}, we deduce that for a.e. \(s\in (0,1)\), there exists a subsequence (we do not relabel) such that \((\unp)_{i}\to u\) in \(L^{1}(\Omega)\) and  \(E_\Omega(\unp)\to E_\Omega(u)\) as  $ i\to \infty$.

\item {\it (\(\unp\in{\rm Lip}(\Omega)\))}.  
We only need to prove that for every \(n\in\N,\ \odel \in (0,1)\), there exists \(C>0\) (which may depend on \(n, \odel\)) such that for every \(t\in \R\) and every \(x, y\in \Omega\), it holds
\begin{equation}
    \label{for-Lip}
\Pnd(x, t+C|x-y|) \leq\Pnd(y,t)\,,
\end{equation}
which in  view of Lemma~\ref{rk-lips-us} ensures that each $(\Phi_{n_i,\overline{\delta}_i})^{-1}(\cdot,s)=\unp$ is Lipschitz continuous.

We start from the fact that for all \(j=0, \dots, J\) functions  \(\bunpj\) are Lipschitz continuous on    \(U_j\)  with constant \(C_n\). Then for every \(x,y\in U_j\), every \(t\in \R\), and $c\geq C_n$ we have
\(
1_{\bunpj}(x, t+c|x-y|)\leq 1_{\bunpj}(y,t)\,
\)
so that 
\(\Psi_j(x)1_{\bunpj}(x, t+c|x-y|)\leq \Psi_j(x)1_{\bunpj}(y,t)\)
and thus
\begin{equation}\label{eq3295}
\Psi_j(x)1_{\bunpj}(x, t+c|x-y|)\leq 
\Psi_j(y)1_{\bunpj}(y,t)+\|\nabla \Psi_j\|_{L^{\infty}}|x-y|\,.
\end{equation}
Inequality \eqref{eq3295} remains true when \(x\not\in U_j\) (because in that case, the left-hand side vanishes) and when \(y\not\in U_j\) (because in that case, it follows from the inequality \(\Psi_j(x)\leq \Psi_j(y)+\|\nabla \Psi_j\|_{L^{\infty}}|x-y|=\|\nabla \Psi_j\|_{L^{\infty}}|x-y|\)).

Let \(C\coloneqq C_n+\frac{2M'}{\odel}\sum_{j=0}^{J}\|\nabla \Psi_j\|_{L^{\infty}}\).
Then for every \(x,y\in \Omega\),
\[
\sum_{j=0}^{J}\Psi_j(x)1_{\bunpj}(x, t+C|x-y|)\leq  \sum_{j=0}^{J}\left(\Psi_j(y)1_{\bunpj}(y, t)+\|\nabla \Psi_j\|_{L^{\infty}}|x-y|\right)\,.
\]
Hence, using the definition of \(\Pnd\),
\begin{equation}\label{eq2031}
\Pnd(x,t+C|x-y|)\leq\Pnd(y,t) + \sum_{j=0}^{J}\|\nabla \Psi_j\|_{L^{\infty}}|x-y|+\odel (\overline{\alpha} (t+C|x-y|)-\overline{\alpha}(t))\,. 
\end{equation}
We are now in position to prove \eqref{for-Lip} for we are in one of the following three cases. If \(t+C|x-y|\geq M'\), then by \eqref{eq1904}, $\Pnd(x,t+C|x-y|)=0\leq \Pnd(y,t)\,.$
If \(t\leq -M'\), 
then by \eqref{eq1904} again, $\Pnd(x,t+C|x-y|)\leq 1+\odel= \Pnd(y,t)\,.$
Otherwise \(-M'\leq t \leq t+C|x-y|\leq M'\) and in that case $\overline{\alpha} (t+C|x-y|)-\overline{\alpha}(t)=-\frac{C|x-y|}{2M'}. $
In view of \eqref{eq2031} and the definition of \(C\), this yields~\eqref{for-Lip}, which implies that $\unp$ is Lipschitz continuous on~$\Omega$.

    \item {\it (\(\unp=\vp\) on \(\partial \Omega\))}.  
For every \(n\in\N\) and every \(1\leq j\leq J\),  one has \(\bunpj=\vp\) on \(U_j\cap \partial \Omega\). Hence, $\Pnd= {1}_{\vp}+\odel\overline{\alpha} \text{ on } \partial \Omega \times \R\,$. Thus, for every \(i\in \N\) such that \(s >\overline{\delta}_i\), it holds \(\unp=\vp\) on \(\partial \Omega\).

\end{enumerate}

\noindent{\bf Step 4. Conclusion. } 
For a.e. \(s\in (0,1)\), we  have constructed a sequence \((u^{s,\varphi}_i)_{i\in \N}\subset W^{1,\infty}(\Omega)\), which is bounded in \(L^{\infty}(\Omega)\), converges to \(u\) in \(L^{1}(\Omega)\) and satisfies \(\lim_{i\to +\infty}E_{\Omega}(u^{s,\varphi}_i)=E_{\Omega}(u)\). Moreover, for every \(i\) sufficiently large, one has \(\overline{\delta}_i<s\), which implies that \(u^{s,\varphi}_i\) coincides with \(\varphi\) on \(\partial\Omega\). 

\end{proof}

 



\section{Proofs of the main results}\label{sec:proofs}

\subsection{Proof of Theorem~\ref{th-main-conv}}
\label{ssec:proof_conv}
Before giving the proof of Theorem~\ref{th-main-conv}, we present a couple of auxiliary facts. In the next lemma, we extend the integrand \(f\) to a slightly larger domain.

\begin{lemma}\label{lem:reduction-extension}
Let $\Omega \subseteq \rn$ be bounded and Lipschitz, and let $f:\Omega\times \R\times \rn \to \rp$ be a Carathéodory function convex with respect to the last variable which satisfies ~\eqref{eqHZconv} on \(\Omega\). Then there exists an open set \(\Lambda\Supset \Omega\) and a Carath\'{e}odory function \(g:\Lambda\times \R\times \rn\to \R\) convex with respect to the last variable which satisfies~\eqref{eqHZconv} on $\Lambda$ and  coincides with $f$ on $\Omega\times \R \times \R^N$. 
\end{lemma}
\begin{proof}
We first extend \(f\) on \(\overline{\Omega}\times \R\times \R^N\) by setting \(f(x,t,\xi)=0\) for every \(x\in \partial \Omega\) and every \((t,\xi)\in \R\times \R^N\).
Since \(\Omega\) is bounded and Lipschitz, there exists an open set \(\Omega_0\subset \R^N\) such that \(\partial \Omega \subset \Omega_0\) and a biLipschitz map \(\Pi_0:\Omega_0\to \Omega_0\) which agrees with the identity on \(\partial\Omega\), and such that \(\Pi_0(\Omega_0\setminus \overline{\Omega}) = \Omega_0\cap \Omega \), see e.g. \cite[Theorem 3.1.]{Conti-Focardi-Iurlano}.
We then set \(\Lambda\coloneqq \Omega\cup \Omega_0\) and 
\[
\Pi:x\in \Lambda \mapsto 
\begin{cases}
x & \textrm{ if } x\in \overline{\Omega},\\
\Pi_0(x) & \textrm{ if } x\in \Omega_0\setminus \overline{\Omega}.
\end{cases}
\]
Since \(\Pi_0\) is Lipschitz continuous on \(\Omega_0\setminus \Omega\) and coincides with the identity on \(\partial\Omega\), we deduce that \(\Pi\) is Lispchitz continuous on \(\Lambda\). 
We also observe that \(\Pi\) and \((\Pi|_{\Lambda\setminus \overline{\Omega}})^{-1}\) map negligible sets on negligible sets (here, we use that \(\Pi_0\) is a biLipschitz homeomorphism).

We then define \(g(x,t,\xi)=f(\Pi(x),t, \xi)\). This is a Carath\'{e}odory function on \(\Lambda\times \R\times \R^N\) which is convex with respect to the last variable.
Moreover, one has  for every \(x\in \Lambda\) and every \(\ve>0\),
\begin{equation}\label{eq1054}
g_{B(x,\ve)}^{-}(t,\xi)= \essinf_{y\in B(x,\ve)\cap \Lambda} f(\Pi(y),t,\xi)= \essinf_{z\in \Pi(B(x,\ve)\cap \Lambda)} f(z,t,\xi)\,.   
\end{equation}

Denoting by \(\ca\) the Lipschitz constant of \(\Pi\),  one has $\Pi(B(x,\ve)\cap \Lambda)\subset B(\Pi(x), \ca\ve)\cap \overline{\Omega}\,.$ Since \(\partial \Omega\) is negligible, this implies that
\[
g_{B(x,\ve)}^{-}(t,\xi)\geq \essinf_{z\in B(\Pi(x), \ca\ve)\cap \Omega} f(z,t,\xi)=f_{B(\Pi(x),\ca\ve)}^{-}(t,\xi)
\]
and thus
\[
\left(f_{B(\Pi(x),\ca\ve)}^{-}(t,\cdot)\right)^{**}(\xi)\leq \left(g_{B(x,\ve)}^{-}(t,\cdot)\right)^{**}(\xi)\,.
\]
Let \(\mathcal{k}=(\mathcal{k}_1,\mathcal{k}_2)\in (0,\infty)^2\). If \(t\in (-\mathcal{k}_1, \mathcal{k}_1)\) and \((g_{B(x,\ve)}^{-}(t,\cdot))^{**}(\xi) {+ |\xi|^{\max(p, N)}}\leq \mathcal{k}_2\ve^{-N} \), then
\[
\left(f_{B(\Pi(x),\ca\ve)}^{-}(t,\cdot)\right)^{**}(\xi) {+ |\xi|^{\max(p, N)}}\leq {\mathcal{k}_2}{\ve^{-N}}=\frac{\mathcal{k}_2\ca^N}{(\ca\ve)^N}\,.
\]
Since \(f\) satisfies \eqref{eqHZconv}, there exists \(C_{\widetilde{\mathcal{k}}}>0\) with \(\widetilde{\mathcal{k}}=(\mathcal{k}_1, \mathcal{k}_2\ca^N)\) such that for a.e. \(x\in \Pi^{-1}(\Omega)=\Lambda \setminus \partial \Omega\), it holds
\begin{equation*}
f(\Pi(x),t,\xi) \leq C_{\widetilde{\mathcal{k}}}\left(\left(f_{B(\Pi(x),\ca\ve)}^{-}(t,\cdot)\right)^{**}(\xi)+1\right) \leq C_{\widetilde{\mathcal{k}}}\left(\left(g_{B(x,\ve)}^{-}(t,\cdot)\right)^{**}(\xi)+1\right)\,.
\end{equation*}
Since \({|\partial \Omega|}=0\), this proves that for a.e. \(x\in \Lambda\), 
\[
g(x,t,\xi)\leq  C_{\widetilde{\mathcal{k}}}\left(\left(g_{B(x,\ve)}^{-}(t,\cdot)\right)^{**}(\xi)+1\right)\,.
\]
We can conclude that \(g\) satisfies \eqref{eqHZconv} on \(\Lambda\times \R\times \rn\). 
\end{proof}

In the view of the next lemma, we can assume without loss of generality that the function \(u\in W^{1,p}_{\varphi}(\Omega)\) to be approximated, is bounded on \(\Omega\). Remember that this automatically holds true when \(N=1\), due to the embedding \(W^{1,1}(\Omega)\subset L^{\infty}(\Omega)\) in that case.

\begin{lemma}\label{lem:reduction-bounded}
    Let $\Omega \subseteq \rn$ be a bounded open set, and let $f:\Omega\times \R\times \rn \to \rp$ be a Carath\'{e}odory function which is convex with respect to the last variable. We further assume that there exist \(\vt\in [1,\infty]\), \(a\in L^{\vt}(\Omega)\) and \(t_0>0\) which satisfy~\eqref{eq:f-origin} for some $p \geq 1$. 
 If $u\in \cEp(\Omega)\cap W^{1,p}(\Omega)$,  
     then there exists a sequence $(u_n)_{n\in\N}\subset \cEp(\Omega)\cap W^{1,p}(\Omega)\cap L^\infty(\Omega)$, which converges to $u$ in $W^{1,p}(\Omega)$, and such that $\lim\limits_{n\to \infty}E_\Omega(u_n)=E_\Omega(u)$. 
\end{lemma}

\begin{proof}
    We set $u_n\coloneqq\min(n,\max(-n,u)).$
    Then, $\|u_n-u\|_{W^{1,p}(\Omega)}\to 0$ when $n\to\infty$, and for $n\geq\Vert \vp\Vert_{L^\infty}$ we have $u_n\in\Wphi(\Omega)$.
    Moreover,
    \begin{equation}\label{eq:121435}
      E_\Omega(u_n)=  \int_\Omega f(x,u_n,\nabla u_n)\dx=\int_{\{|u| < n\}}f(x,u,\nabla u)\dx + \int_{\{u\geq n\}}f(x,n,0)\dx+\int_{\{u\leq  -n\}}f(x,-n,0)\dx\,.
    \end{equation}
    By the monotone convergence theorem, 
    \begin{equation*}
        \int_{\{|u| < n\}}f(x,u,\nabla u)\dx \xrightarrow[n \to \infty]{} \int_\Omega f(x,u,\nabla u)\dx\,.
    \end{equation*}
Using the assumption \eqref{eq:f-origin} and  the H\"older inequality, for every \(n\geq t_0\),
\[
\int_{\{u\geq  n\}}f(x,n,0)\dx \leq n^{p^*/\vt'}\int_{\{u\geq  n\}}a(x)\dx  \leq n^{p^*/\vt'}|\{u \geq  n\}|^{1/\vt'} \left(\int_{\{u\geq n\}}a(x)^\vt\dx\right)^{1/\vt}\,.
\]
By the Chebyshev inequality  $|\{u \geq n\}|\leq  {n^{-p^*}}\|u\|_{L^{p^*}(\Omega)}^{p^*},$ so the right-hand side is finite since \(u\in W^{1,p}(\Omega)\subset L^{p^*}(\Omega)\).
Hence,
\[
\int_{\{u\geq  n\}}f(x,n,0)\dx \leq \|u\|_{L^{p^*}(\Omega)}^{p^*/\vt'}\left(\int_{\{u\geq  n\}}a(x)^\vt\dx\right)^{1/\vt}\,.
\]
By the monotone convergence theorem again, we get
    \begin{equation*}
        \int_{\{u \geq  n\}}f(x,n,0)\dx \xrightarrow[n \to \infty]{} 0\,.
    \end{equation*}
A similar analysis holds on the sublevel sets \(\{u \leq -n\}\). 
   In conjunction with~\eqref{eq:121435}, this yields  $\lim\limits_{n\to \infty}E_\Omega(u_n)=E_\Omega(u)$.
\end{proof}

The assumption \eqref{eqHZconv} and~\eqref{eqHZiso} enjoy the following stability property.
\begin{lemma}
\label{lm-stab-hzconv}
    Let  \(f:\Omega\times \R\times \R^{N}\to \rp\) be a Carath\'{e}odory function convex with respect to the last variable and \(h:\R^N\to \rp\) {be} convex. If \(f\) satisfies~\eqref{eqHZconv}, then so does \(f+h\). If \(f\) satisfies~\eqref{eqHZiso}, then so does \(f+h\).
\end{lemma}
\begin{proof}
Let \(\mathcal{k}=(\mathcal{k}_1, \mathcal{k}_2)\in (0,\infty)^2\) and \(\ve>0\).
Since \(f\leq f+h\), for every \(x\in \Omega\), \(t\in [-\mathcal{k}_1, \mathcal{k}_1]\) and \(\xi\in \R^N\), 
\begin{equation}\label{eq1123}
\left(f_{B(x,\ve)}^{-}(t,\cdot)\right)^{**}(\xi)\leq \left((f+h)_{B(x,\ve)}^{-}(t,\cdot)\right)^{**}(\xi)\,.
\end{equation}
Hence, if \(\xi\) satisfies
\[
\left((f+h)_{B(x,\ve)}^{-}(t,\cdot)\right)^{**}(\xi) + |\xi|^{\max(p, N)}\leq \frac{\mathcal{k}_2}{\ve^{N}}\,,
\]
then \((f_{B(x,\ve)}^{-}(t,\cdot))^{**}(\xi)+ |\xi|^{\max(p, N)}\leq \mathcal{k}_2\ve^{-N}\). If \(f\) satisfies \eqref{eqHZconv}, this implies that for a.e. \(x\in \Omega\), 
\[
f(x,t,\xi)\leq \cC_{\mathcal{k}}((f_{B(x,\ve)}^{-}(t,\cdot))^{**}(\xi)+1)\,,
\]
for some \( \cC_{\mathcal{k}}\geq 1\).
It follows that
\[
f(x,t,\xi)+h(\xi)\leq \cC_{\mathcal{k}}((f_{B(x,\ve)}^{-}(t,\cdot))^{**}(\xi)+1)+h(\xi)\leq \cC_{\mathcal{k}}((f_{B(x,\ve)}^{-}(t,\cdot))^{**}(\xi)+h(\xi)+1)\,.
\]
We can complete the proof by observing that \((f_{B}^-)^{**}+h\leq f_{B}^{-}+h=(f+h)_{B}^{-}\) and thus
\[
(f_{B(x,\ve)}^{-}(t,\cdot))^{**}(\xi)+h(\xi)\leq ((f+h)_{B(x,\ve)}^{-}(t,\cdot))^{**}(\xi)\,.
\]
The conclusion for condition~\eqref{eqHZiso} can be obtained in  essentially the same way.
\end{proof}

\begin{proof}[Proof of Theorem~\ref{th-main-conv}]
We first consider the case \(p=1\).
Let \(f\) and \(u\) as in the statement. By Lemma~\ref{lem:reduction-bounded} and a diagonal argument, we may assume that \(u\in L^{\infty}(\Omega)\).
Using Lemma~\ref{lem:reduction-extension}, we introduce an extension \(f_1\) of \(f\) on \(\Lambda\times \R\times \R^N\), which  satisfies  \eqref{eqHZconv} on \(\Lambda\times \R\times \R^N\). 
 Note that by~\cite[Theorem~VIII.1.3 and Lemma~VIII.1.2]{Ekeland-Temam} for $\nabla u \in L^1(\Lambda)$ there exists a superlinear convex function  \(h:\rn\to \rp\) such that \(h(\nabla u)\in L^{1}(\Lambda)\) and \(h(0)=0\). Let $g_0$ and  $g$ be the functions defined for a.e. $x\in\Lambda$ and all $(t,\xi)\in\R\times \rn$ by
\[
g_0(x,t,\xi)=(f_1(x,t,\xi)-f_1(x,t,0))_+\,, \quad g(x,t,\xi)\coloneqq g_0(x,t,\xi)+h(\xi)+\sqrt{1+|\xi|^2}-1\,.
\]
Then \(g\) is a Carath\'{e}odory non-negative function which is convex with respect to the last variable. Moreover, \(g(x,t,\xi)\geq h(\xi)\) and \(g(x,t,0)=0\) for every \((x,t,\xi)\in \Lambda\times \R\times \rn\). 
From  Lemma~\ref{lm-reduction-2-0}, we deduce that the function \(g_0\) satisfies \eqref{eqHZconv}, while Lemma~\ref{lm-stab-hzconv} allows to conclude that the same property holds for \(g\). Hence, \(g\) satisfies \ref{f-red} and \eqref{eqHZconv}.
Since \(E_{\Omega}(u) <\infty\), one also has  \(\int_{\Omega}g(x,u,\nabla u)\dx <\infty\). By Proposition~\ref{prop-main-reduced} applied to \(g\), there exists a sequence \((u_n)_{n\in \N}\subset W^{1,\infty}_{\varphi}(\Omega)\)  such that $ u_n \to u$   as $n \to \infty$ in $L^{1}(\Omega)$, $\sup_n\|u_n\|_{L^\infty}<\infty$ and
\begin{equation}\label{eq-main-reduced}
\lim_{n\to \infty}\int_{\Omega} g(x,u_n, \nabla u_n)\dx=\int_{\Omega} g(x,u, \nabla u)\dx \,.
\end{equation} 
Since \(g(x,u_n(x), \nabla u_n(x))\geq h(\nabla u_n(x))\), 
the Dunford--Pettis theorem implies that, up to a subsequence, \((u_n)_{n\in \N}\) weakly converges to \(u\) in \(W^{1,1}(\Omega)\). By the weak lower semicontinuity of the functionals in \(W^{1,1}(\Omega)\), we know that
\[
\liminf_{n\to \infty}\int_{\Omega}g_0(x,u_n,\nabla u_n)\dx  \geq \int_{\Omega}g_0(x,u,\nabla u)\dx \,, \quad \liminf_{n\to \infty}\int_{\Omega}h(\nabla u_n)\dx  \geq \int_{\Omega}h(\nabla u)\dx
\]
and 
\begin{equation}\label{eq2607}
\liminf_{n\to \infty}\int_{\Omega}\sqrt{1+|\nabla u_n|^2}\dx  \geq \int_{\Omega}\sqrt{1+|\nabla u|^2}\dx \,.
\end{equation}
In view of \eqref{eq-main-reduced}, this implies that
\[
\lim_{n\to \infty}\int_{\Omega}g_0(x,u_n,\nabla u_n)\dx  = \int_{\Omega}g_0(x,u,\nabla u)\dx\, , \quad \lim_{n\to \infty}\int_{\Omega}\sqrt{1+|\nabla u_n|^2}\dx  = \int_{\Omega}\sqrt{1+|\nabla u|^2}\dx\, .
\]
The last equality above and the  weak convergence of \((u_n)_{n\in \N}\) in \(W^{1,1}(\Omega)\) then imply the strong convergence in \(W^{1,1}(\Omega)\), see e.g. \cite[Section~1.3.4, Proposition~1]{GMS-2}. Since \(g_0\) is non-negative, it follows from the first equality that \((g_0(\cdot,u_n, \nabla u_n))_{n\in\N}\) converges to \(g_0(\cdot, u, \nabla u)\) in \(L^{1}(\Omega)\).
 We note that
\[
f(x,u_n(x), \nabla u_n(x))\leq g_0(x,u_n(x),\nabla u_n(x))+f(x,u_n(x),0)\leq g_0(x,u_n(x), \nabla u_n(x))+\max_{|t|\leq \sup_n\|u_n\|_{L^\infty}}f(x,t,0)\,, 
\]
and that the last term in the right-hand side is bounded by  Lemma~\ref{lem-bdd-f}. The Vitali convergence theorem thus implies
\[
\lim_{n\to \infty}\int_{\Omega}f(x,u_n(x), \nabla u_n(x))\dx =\int_{\Omega}f(x,u(x), \nabla u(x))\dx \,.
\]
This completes the proof when \(p=1\). 

When \(p>1\), we replace the term \(h(\xi)+\sqrt{1+|\xi|^2}-1\) by \(|\xi|^p\); that is, we introduce the function:
\[
\overline{g}(x,t,\xi)=f(x,t,\xi)+|\xi|^p\,,
\]
and apply the above arguments to \(\overline{g}\) instead of \(g\). Observe in particular that since \(u\in W^{1,p}(\Omega)\), we still have \(\int_{\Omega}\overline{g}(x,u,\nabla u)\dx<\infty\). We thus obtain an approximating sequence \((u_n)_{n\in \N}\) which satisfies \eqref{eq-main-reduced} with \(\overline{g}\) instead of \(g\).  By \(p\)-coercivity of \(\overline{g}\), this implies that  \((u_n)_{n\in\N}\) weakly converges to \(u\) in \(W^{1,p}(\Omega)\). By the  above lower semicontinuity argument which leads to \eqref{eq2607}, we deduce that 
\[
\lim_{n\to +\infty}\int_{\Omega}|\nabla u_n|^p\dx = \int_{\Omega}|\nabla u|^p\dx. 
\] 
By uniform convexity of the \(L^{p}\) norm, it follows that \((u_n)_{n\in\N}\) strongly converges to \(u\) in \(W^{1,p}(\Omega)\).
The end of  the proof can now be repeated without changes, which completes the case \(p>1\).
\end{proof} 

\subsection{Proofs of Theorems~\ref{theo:doubling} and~\ref{theo:iso-ortho}}
\label{ssec:proofs_rest}

Having proven Theorem~\ref{th-main-conv}, we may now infer Theorems~\ref{theo:doubling} and~\ref{theo:iso-ortho}.

\begin{proof}[Proof of Theorem~\ref{theo:doubling}] 
It is enough to show that our regime implies \eqref{eqHZconv}. Indeed, once we show \eqref{eqHZconv} the result holds due to Theorem~\ref{th-main-conv}. 

Fix $L=(L_1,L_2)\in (0,\infty)^2$ and $t\in [-L_1,L_1]$. Without loss of generality, we can assume that the corresponding constant \(C_L\) in \eqref{eqHZD2} is not lower than $1$. Since \(f(x,t,0)=0\), the convexity of $f$ implies that 
    \begin{equation}\label{eq:delta2_conv}f_{B(x,\ve)}^{-}(t,\xi)\leq {L_2}{\ve^{-N}} \implies
f(x,t,\tfrac{1}{C_L}\xi) \leq f_{B(x,\ve)}^{-}(t,\xi)+1\,.\end{equation}
    We apply \cite[Theorem~1.2]{ha-aniso} to get \[(f_{B(x,\ve)}^{-}(t,\cdot))^{**}(\xi)
\leq {L_2}{\ve^{-N}} \implies
f(x,t,\tfrac{1}{\wt C_L}\xi) \leq  (f_{B(x,\ve)}^{-}(t,\cdot))^{**}(\xi)+1\,,\] for some $\wt C_L>1$. Since $f(x,t,\cdot)$ satisfies the $\Delta_2$ condition, there exists $\overline{C}_L>1$ such that
\begin{equation}\label{eq:delta2_delta2}(f_{B(x,\ve)}^{-}(t,\cdot))^{**}(\xi)
\leq {L_2}{\ve^{-N}} \implies
f(x,t,\xi) \leq \overline{C}_L[  (f_{B(x,\ve)}^{-}(t,\cdot))^{**}(\xi)+1]\,,\end{equation} which implies \eqref{eqHZconv}.

\end{proof}

\begin{proof}[Proof of Theorem~\ref{theo:iso-ortho}]
We simply observe that by Theorem~\ref{theo:final}, $f$ satisfies condition~\eqref{eqHZconv}, in the isotropic as well as in the orthotropic case. Therefore, Theorem~\ref{th-main-conv} gives the conclusion.

\end{proof}

\section*{Appendix}
\renewcommand\thesection{\Alph{section}}
\renewcommand{\theHsection}{\Alph{section}}
\setcounter{section}{0}
\setcounter{proposition}{0}

\section{Auxiliaries}\label{sec:ex}
\begin{proposition}\label{prop:comp}
    Let  \(\Omega\) be a bounded Lipschitz open set in \(\rn\). Let $f : \Omega \times \R \times \rn \to [0, \infty)$ be a Carath\'eodory function such that there exists a function $g : \Omega \times \R \times \rn \to [0, \infty)$ satisfying all the assumptions of Theorem~\ref{th-main-conv} with some $p \geq 1$. Assume that there exist a constant $C > 0$ and a non-negative function $\vartheta \in L^1(\Omega;\R)$ such that~\eqref{eq:ass} holds true.
    %
    %
    Let \(\vp:\rn\to \R\) be Lipschitz continuous. Then, for every \(u\in W^{1,p}_{\vp}(\Omega)\) such that $\int_{\Omega} f(x, u, \nabla u)\dx < \infty$, there exists a sequence \((u_{n})_{n\in\N}\subset W^{1,\infty}_\vp(\Omega)\)  such that  $u_n \to u$ in $W^{1, p}(\Omega)$ as $n \to \infty$ and
\begin{equation*}
    \lim_{n\to \infty} \int_{\Omega} f(x, u_n, \nabla u_n)\dx = \int_{\Omega} f(x, u, \nabla u)\dx \,.
\end{equation*}
\end{proposition}
\begin{proof}
     Let us take any $u \in W^{1, p}_{\vp}(\Omega)$ such that $\int_{\Omega} f(x, u, \nabla u)\dx < \infty$. Then by~\eqref{eq:ass} we know that $\int_{\Omega} g(x, u, \nabla u)\dx \leq  \int_{\Omega} C(f(x, u, \nabla u) + \vartheta(x))\dx < \infty$. As $g$ satisfies assumptions of Theorem~\ref{th-main-conv}, we infer that there exists $(u_n)_n \subset W^{1, \infty}_{\vp}(\Omega)$ converging to $u$ in $W^{1, p}(\Omega)$, and such that
    \begin{equation}\label{eq:1}
         \lim_{n\to \infty} \int_{\Omega} g(x, u_n, \nabla u_n)\dx = \int_{\Omega} g(x, u, \nabla u)\dx \,.
    \end{equation}
    Since $(u_n)_n$ converges in $W^{1, p}(\Omega)$, we pick its subsequence $(u_{n_k})_{k}$ such that $u_{n_k}\to u$ and $\nabla u_{n_k}\to \nabla u$ for a.e. $x \in \Omega$. As $g$ is continuous with respect to its last two variables, we get that $g(x, u_{n_k}(x), \nabla u_{n_k}(x)) \to g(x, u(x), \nabla u(x))$ for a.e. $x \in \Omega$. This with~\eqref{eq:1} and the Brezis-Lieb lemma implies that $(g(\cdot, u_{n_k}, \nabla u_{n_k}))_{k}$ converges in $L^1$ to $g(\cdot, u, \nabla u)$. By the Vitali convergence theorem, $(g(\cdot, u_{n_k}, \nabla u_{n_k}))_k$ is uniformly integrable, and by~\eqref{eq:ass} we also get that $(f(\cdot, u_{n_k}, \nabla u_{n_k}))_k$ is uniformly integrable. As $f$ is continuous with respect to its last two variables, we have that $f(x, u_{n_k}(x), \nabla u_{n_k}(x)) \to f(x, u(x), \nabla u(x))$ for a.e. $x$. We finish the proof by noticing that the Vitali convergence theorem implies
    \begin{equation*}
        \lim_{k \to \infty} \int_{\Omega} f(x, u_{n_k}, \nabla u_{n_k})\dx = \int_{\Omega} f(x, u, \nabla u)\dx\,.
    \end{equation*}
\end{proof}
\begin{lemma}
    \label{lem-bdd-f}
Let $\Omega \subseteq \rn$ be a bounded open set, and let $f:\Omega\times \R\times \rn \to \rp$ be a Carath\'{e}odory function which   satisfies either \eqref{eqHZconv} or \eqref{eqHZiso}. Then
     for every  {$m > 0$}, there exists \(\constK_m>0\) such that for a.e. \(x\in \Omega\),
\begin{equation}\label{eq-lem-bdd-f} 
\max_{|(t,\xi)|\leq m} f(x,t,\xi) \leq \constK_m\,.
\end{equation}
\end{lemma}
\begin{proof}
For a.e. \(y\in \Omega\), the function \((t,\xi)\mapsto f(y,t,\xi)\) is continuous. Let \(A_m\coloneqq\essinf_{y\in \Omega}\max_{|(t,\xi)|\leq m} f(y,t,\xi)\). 
Setting  {$\mathcal{k}=(\mathcal{k}_1, \mathcal{k}_2)$  with $\mathcal{k}_1=m$  and $\mathcal{k}_2=(A_m+m^{\max(p,N)})(\diam\Omega)^N$}, it follows either from \eqref{eqHZconv} or from \eqref{eqHZiso} with $\ve=\diam \Omega $ that there exists \(C_{\mathcal{k}}>0\) such that  for a.e. \(x\in \Omega\) and every \(|(t,\xi)|\leq m\),
\[
f(x,t,\xi)\leq C_{\mathcal{k}}(A_m+1)\,.\qedhere
\] 
\end{proof}
\begin{lemma}\label{lm-reduction-2-0}
Let \(f:\Omega\times \R\times \R^N\to [0,\infty)\) be a Carathéodory function convex with respect to the last variable and \(g(x,t,\xi)\coloneqq(f(x,t,\xi)-f(x,t,0))_+\). 
 Then map \(g\) satisfies~\eqref{eqHZconv}, if and only if \(f\)  satisfies~\eqref{eqHZconv}. Moreover,   \(g\) satisfies~\eqref{eqHZiso}, if and only if \(f\)  satisfies~\eqref{eqHZiso}.
\end{lemma}
\begin{proof} {Let us prove the result for~\eqref{eqHZconv}}. We observe that  \(g\) is a Carath\'{e}odory non-negative function which is convex with respect to the last variable.
Assume first that~\eqref{eqHZconv} holds for \(f\).
 Let  \(\mathcal{k}=(\mathcal{k}_1,\mathcal{k}_2)\in (0,\infty)^2\) and \(\ve\in (0,\diam\Omega]\). For  a.e. \(x\in \Omega\), for  every \(t\in (-\mathcal{k}_1, \mathcal{k}_1)\), \(\xi \in \R^N\),
\begin{equation}\label{eq712}
f(x,t,\xi)\leq g(x,t,\xi)+f(x,t,0)\leq g(x,t,\xi)+\max_{|t|\leq \mathcal{k}_1}f(x,t,0)\leq g(x,t,\xi)+\constK_{\mathcal{k}_1}, 
\end{equation}  
where the constant \(\constK_{\mathcal{k}_1}\) in the last inequality is given by Lemma~\ref{lem-bdd-f}.
Hence,
\begin{equation}\label{eq279}
(f_{B(x,\ve)}^{-}(t,\cdot))^{**}(\xi)\leq (g_{B(x,\ve)}^{-}(t,\cdot))^{**}(\xi)+\constK_{\mathcal{k}_1}\,.
\end{equation}
Assuming further that $(g_{B(x,\ve)}^{-}(t,\cdot))^{**}(\xi)  + |\xi|^{\max(p, N)}\leq \mathcal{k}_2\ve^{-N}$, we can apply \eqref{eqHZconv} to \(f\) with $\wt{\mathcal{k}}=(\mathcal{k}_1,\mathcal{k}_2+(\diam \Omega)^N\constK_{\mathcal{k}_1})$. We infer that there exists \(\cC_{\wt{\mathcal{k}}}\geq 1\)  such that
\begin{equation*}
f(x,t,\xi)\leq \cC_{\wt{\mathcal{k}}} ((f_{B(x,\ve)}^{-}(t,\cdot))^{**}(\xi)+1)\,.
\end{equation*}
Using that \(g(x,t,\xi)\leq f(x,t,\xi) \) in the left-hand side and \eqref{eq279} in the right-hand side,
we obtain
\[
g(x,t,\xi)\leq \cC_{\wt{\mathcal{k}}} ((g_{B(x,\ve)}^{-}(t,\cdot))^{**}(\xi)+\constK_{\mathcal{k}_1}+1)\;,
\]
which shows that $g$ satisfies \eqref{eqHZconv}. 
This completes the proof of the \emph{if} part of the statement.

Conversely, assume that \eqref{eqHZconv} holds for \(g\).
Let \(\mathcal{k}=(\mathcal{k}_1, \mathcal{k}_2)\in (0,\infty)^2\) and \(\ve\in (0, \diam\Omega]\). For a.e. \(x\in \Omega\) and every \(t\in (-\mathcal{k}_1, \mathcal{k}_1)\), if \(\xi \in \R^N\) satisfies \((f_{B(x,\ve)}^-(t,\cdot))^{**}(\xi)  {+ |\xi|^{\max(p, N)}}\leq \mathcal{k}_2 \ve^{-N}\), then since \(g\leq f\), one has
\begin{equation*}
\left(g_{B(x,\ve)}^-(t,\cdot)\right)^{**}  {+ |\xi|^{\max(p, N)}}\leq \left(f_{B(x,\ve)}^-(t,\cdot)\right)^{**}  {+ |\xi|^{\max(p, N)}}\leq \mathcal{k}_2 \ve^{-N}\,.
\end{equation*}
Using that \(g\) satisfies \eqref{eqHZconv}, we deduce that there exists \(\cC_{\mathcal{k}}\geq 1\) such that
\[
g(x,t,\xi)\leq \cC_{\mathcal{k}}\left(\left(g_{B(x,\ve)}^-(t,\cdot)\right)^{**}(\xi)+1\right)\leq  \cC_{\mathcal{k}}\left(\left(f_{B(x,\ve)}^-(t,\cdot)\right)^{**}(\xi)+1\right).
\]
We can conclude by \eqref{eq712} that \(f\) satisfies \eqref{eqHZconv}.  The proof for~\eqref{eqHZiso} is essentially the same and we omit it.
\end{proof}
\section{Examples}
We provide an example illustrating that condition~\eqref{eqHZconv}  and conditions~\eqref{eqHZiso},~\eqref{eqHZ-iso-easy} are essentially different when considered in the fully anisotropic setting.
\begin{example}\label{ex:counter}
Let $N = 2$ and let the function $f : B(0, 2) \times \R^2 \to [0, \infty)$ be defined as
    \begin{equation*}
        f(x, \xi) \coloneqq |\langle x, \xi \rangle|^4 + |\xi|\,.
    \end{equation*}
    Then $f$ is of at least linear growth with respect to the last variable and vanishes for \(\xi=0\). Moreover, for every $L \in (0, \infty)$, there exists a constant \(C_{L}>0\)  such that for every \(\ve>0\), for every $x\in B(0, 2)$ and $\xi \in \R^2$, it holds
    \begin{equation}\label{eq:f-cond}
        |\xi| \leq L\ve^{-1} \Longrightarrow f(x, \xi) \leq C_L\left(f(y, \xi) + 1\right)\,.
    \end{equation}
    However, $f$ does not satisfy the condition $(H^{\text{conv}})$ with $p = 1$.
\begin{proof}
    Note that whenever $|\xi| \leq L\ve^{-1}$, then for every $x, y$ such that $|x - y| < \ve$, we have
    \begin{align*}
        f(x, \xi) 
        \leq \left(|\langle y, \xi \rangle| + |\langle x - y, \xi \rangle| \right)^4 + |\xi|
        \leq \left(|\langle y, \xi \rangle| + L \right)^4 + |\xi| \leq 8|\langle y, \xi \rangle|^4 + 8L^4 + |\xi|
        \leq (8L^4 + 8)(f(y, \xi) + 1)\,.
    \end{align*}
    Hence, \eqref{eq:f-cond} holds true with $C_L = 8L^4 + 8$. Let us now fix any $\ve \in (0, 1)$, and take $x = (0, \ve^{1/2})$. We consider the ball $B(x, \ve)$. Observe that for any $a \in \R$, we have
    \begin{equation}\label{eq:f-inq}
        f\left((-\ve, \ve^{1/2}), (\ve^{-1/2}a, a)\right) = (1 + \ve^{-1})^{1/2}|a| = f\left((\ve, \ve^{1/2}), (-\ve^{-1/2}a, a)\right)\,.
    \end{equation}
    Let us denote $g \coloneqq (f^-_{B(x, \ve)})^{**}$. By  convexity of $g$, we now get for any $a \in \R$ that
    \begin{align}\label{eq:inq-g}
        g\left((0, a)\right) &\leq \tfrac{1}{2}g\left(\left(\ve^{-1/2}a, a\right)\right) + \tfrac{1}{2}g\left(\left(-\ve^{-1/2}a, a\right)\right)\nonumber\\
        &\leq \tfrac{1}{2}f\left((-\ve, \ve^{1/2}), (\ve^{-1/2}a, a)\right) + \tfrac{1}{2}f\left((\ve, \ve^{1/2}), (-\ve^{-1/2}a, a)\right)
        = (1 + \ve^{-1})^{1/2}|a| \leq 2\ve^{-1/2}|a|\,,
    \end{align}
    where in the equality we used~\eqref{eq:f-inq}. Let us now take $\xi = (0, \ve^{-1})$. By~\eqref{eq:inq-g}, we get
    \begin{equation}\label{eq:andecent}
        g(\xi) + |\xi|^2 \leq 2\ve^{-3/2} + \ve^{-2} \leq 3\ve^{-2}\,.
    \end{equation}
    On the other hand, again by~\eqref{eq:inq-g}, we get
    \begin{equation}\label{eq:precedent}
        f(x, \xi) = |\langle x, \xi \rangle|^4 + |\xi| = \ve^{-2} + \ve^{-1} \geq \tfrac{1}{2}\ve^{-1/2}\left( g(\xi) + 1 \right)\,. 
    \end{equation}
    The two inequalities ~\eqref{eq:andecent} and~\eqref{eq:precedent} imply that $f$ cannot satisfy \eqref{eqHZconv}.
\end{proof}
\end{example}
Let us infer the absence of the Lavrentiev gap for a generalized double phase functional of Orlicz non-doubling type.
{ \begin{example}\label{ex:f-o}
Let $ \psi_0,  \psi_1:\rp\to \rp$ be two increasing non-negative convex functions such that \(\psi_1/\psi_0\) is non-decreasing. Let $a \in \SP_{\omega(\cdot, \cdot)}\cap L^{\infty}(\Omega\times\R)$ with $\omega$ satisfying~\eqref{eq:omega-cond}. Then, the Lavrentiev phenomenon does not occur with the integrand $f_{\text{o}}(x, t, \xi) \coloneqq  \psi_0(|\xi|) + a(x, t) \psi_1(|\xi|)$. 
    
    \begin{proof} Let us take any $\cal{k_2} > 0$ and consider $x, t, \xi, \ve$ such that ${(f_{\text{o}})}^-_{B(x,\ve)}(t,\xi)+|\xi|^{N}
\leq {\cal{k}_2}{\ve^{-N}}$. We get that $|\xi| \leq \min\left( \psi_0^{-1}({\cal{k}}_2\ve^{-N}), {\cal{k}}_2^{1/N}\ve^{-1}\right)$. Fix any $y \in B(x, \ve)$ and consider the function $C(\cdot)$ coming from the definition of $\SP_{\omega(\cdot,\cdot)}$. If $\frac{\ve^{-N}}{ \psi_1( \psi_0^{-1}({\cal{k}}_2\ve^{-N}))} \leq \frac{ \psi_0({\cal{k}}_2^{1/N}\ve^{-1})}{ \psi_1({\cal{k}}_2^{1/N}\ve^{-1})}$, we use the inequality ${|\xi| \leq {\cal{k}}_2^{1/N}\ve^{-1}}$ and condition~\eqref{eq:omega-cond} to get
\begin{equation}\label{eq:help-omega-1}
    \frac{f_{\text{o}}(x, t, \xi)}{f_{\text{o}}(y, t, \xi)} \leq 1 + C(t) + {C(t)}\omega(\ve)\frac{ \psi_1({\cal{k}}_2^{1/N}\ve^{-1})}{ \psi_0({\cal{k}}_2^{1/N}\ve^{-1})} \leq 1 + C(t) + {C(t)}\wt C(\cal{k}_2)\,.
\end{equation}
On the other hand, if $\frac{\ve^{-N}}{ \psi_1( \psi_0^{-1}({\cal{k}}_2\ve^{-N}))} \geq \frac{ \psi_0({\cal{k}}_2^{1/N}\ve^{-1})}{ \psi_1({\cal{k}}_2^{1/N}\ve^{-1})}$, we use the inequality $|\xi| \leq  \psi_0^{-1}({\cal{k}}_2\ve^{-N})$ and condition~\eqref{eq:omega-cond} to get
\begin{equation}\label{eq:help-omega-2}
    \frac{f_{\text{o}}(x, t, \xi)}{f_{\text{o}}(y, t, \xi)} \leq 1 + C(t) + {C(t)}\omega(\ve)\frac{ \psi_1( \psi_0^{-1}({\cal{k}}_2\ve^{-N}))}{{\cal{k}}_2\ve^{-N}} \leq 1 + C(t) + {C(t)}\wt C(\cal{k}_2)\,.
\end{equation}
 From~\eqref{eq:help-omega-1} and~\eqref{eq:help-omega-2}, we deduce that $f_{\text{o}}$ satisfies~\eqref{eqHZiso} under the condition $a \in \SP_{\omega(\cdot,\cdot)}$ with $\omega$ satisfying~\eqref{eq:omega-cond}. As $0\leq f(x, t, 0) \leq \psi_0(0)+\|a\|_{L^{\infty}}\psi_2(0)$, Theorem~\ref{theo:iso-ortho} applies.
 \end{proof}
\end{example}
The example below illustrates how to exclude the Lavrentiev phenomenon in case when the Lagrangian is not convex with respect to the last variable, but is comparable to a convex one.
\begin{example}[Koch, Ruf, Sch\"affner~\cite{Lukas23}] \label{ex:Lukas}
    Let us define $f_{\text{D}} \coloneqq |\xi|^p + a(x, t)\exp(|\xi|^q)$ for any $p \geq 1, q > 0$. Assume that $a \in \SP_{\omega(\cdot, \cdot)} \cap L^{\infty}(\Omega\times\R)$ with $\omega(s) \leq \exp(-s^{-\vk})$, $\vk > q\min(1, N/p)$. Then, there is no Lavrentiev gap for the functional~\eqref{def-E-of-u} with the integrand $f_{\text{D}}$.
   
\begin{proof} In  view of Proposition~\ref{prop:comp} and Theorem~\ref{theo:iso-ortho}, we shall find a function $f_{\text{W}}$ satisfying the assumptions of Theorem~\ref{theo:iso-ortho}, and being comparable to $f_{\text{D}}$. Let $r_* \coloneqq q^{-1/q}$. 
We define $f_{\text{W}}$ as
    \begin{equation*}
        f_{\text{W}}(x, t, \xi) \coloneqq \begin{cases}
            |\xi|^p + a(x, t)\exp(r_*^q)r_*^{-1}|\xi|\,,&  \text{  for } |\xi| \in [0, r_*]\,,\\
            |\xi|^p + a(x, t)\exp(|\xi|^q)\,,&  \text{  otherwise}.
        \end{cases}
    \end{equation*}
    The convexity of $f_{\text{W}}$ with respect to $\xi$ follows from the fact that the function $r \mapsto \exp(r^q)$ is convex on \([r_*,\infty)\) and its derivative at \(r=r_*\) is equal to  $\exp(r_*^q)r_*^{-1}$. It is also clear that $f_{\text{W}} \leq f_{\text{D}} \leq f_{\text{W}} + \|a\|_{L^{\infty}}\exp(r^q_*)$
    and also that $f_{\text{W}}$ satisfies~\eqref{eq:f-origin}, as $f_{\text{W}}(x, t, 0) \leq \|a\|_{L^{\infty}}$. We now aim at showing that $f_{\text{W}}$ satisfies~\eqref{eqHZiso}. By similar computations to those in Example~\ref{ex:f-o}, we only have to prove that $\omega$ satisfies~\eqref{eq:omega-cond} with $$
    \psi_0(s)=s^p \qquad\text{and}\qquad \psi_1(s)=\begin{cases}
    \exp(r_*^q)r_{*}^{-1}s & \textrm{ if } s\leq r_*\,,\\ \exp(s^q) & \textrm{ if } s\geq r_*\,.
    \end{cases}
    $$
   We note that for any $L > 0$ and sufficiently small $s$, we have
    \begin{equation*}
        \omega(s)\frac{ \psi_1( \psi_0^{-1}(Ls^{-N}))}{s^{-N}} {\leq} s^N\exp(L^{q/p}s^{-Nq/p} - s^{-\vk})\,, \quad \omega(s)\frac{ \psi_1(L^{1/N}s^{-1})}{\psi_0(L^{1/N}s^{-1})}{\leq} L^{-p/N}s^{p}\exp(L^{q/N}s^{-q} - s^{-\vk})\,. 
    \end{equation*}
    As $\vk > q\min(1, N/p)$, at least one of the expressions above tends to $0$ as $s$ tends to $0$, and therefore,~\eqref{eq:omega-cond} is satisfied for small $s$. For larger $s$,~\eqref{eq:omega-cond} is clear as the left-hand side is bounded from above while the right-hand side is bounded from below by positive constants independent of \(s\).   
    Hence, the function $f_{\text{W}}$ satisfies \eqref{eqHZiso}. From Theorem~\ref{theo:iso-ortho} and Proposition~\ref{prop:comp}, we can now infer the absence of the Lavrentiev phenomenon for the functional $u \mapsto \int_{\Omega} f_{\text{D}}(x, u, \nabla u)\dx$. 
    \end{proof}
\end{example} 
The last example illustrates the exclusion of the Lavrentiev gap in the doubling fully anisotropic setting.
\begin{example}\label{ex:f-a}
    Let $\psi : [0, \infty) \to [0, \infty)$ be a convex and non-decreasing function such that $\psi \in \Delta_2$ and let $\upsilon : \rn \to \rn$ be $C^{0, \gamma}$ for some $\gamma \in (0, 1]$. Then, the Lavrentiev phenomenon does not occur with the integrand $f_{\text{a}}(x, \xi) \coloneqq \psi(|\langle \upsilon(x), \xi \rangle|) + |\xi|^{N/\gamma}$.
\end{example}
\begin{proof} We shall show that $f$ satisfies~\eqref{eqHZD2}. Let $C_{\psi}\geq 1$ be a doubling constant of $\psi$ and $C_\upsilon$ be a H\"older constant of $\upsilon$. Note that whenever $f^-_{B(x, \ve)}(\xi)\leq L\ve^{-N}$, we have $|\xi| \leq L^{\gamma/N}\ve^{-\gamma}$. Therefore, for every $y \in B(x, \ve)$, we have
    \begin{align*}
        f_{\text{a}}(x, \xi) &= \psi(|\langle \upsilon(x), \xi \rangle|) + |\xi|^{N/\gamma} \leq C_{\psi}\left(\psi(|\langle \upsilon(y), \xi\rangle|) + \psi(|\langle \upsilon(x) - \upsilon(y), \xi\rangle|){+1}\right) + |\xi|^{N/\gamma}\\
        &\leq C_{\psi}\left(\psi(|\langle \upsilon(y), \xi\rangle|) + \max_{s \leq C_\upsilon L^{\gamma/N}}\psi(s){+1}\right) + |\xi|^{N/\gamma} \leq C_{\psi} \left(1 + \max_{s \leq C_\upsilon L^{\gamma/N}}\psi(s)\right)(f(y, \xi) + 1)\,,
    \end{align*}
    which is~\eqref{eqHZD2} with $C_L = C_{\psi}\left(1 + \max_{s \leq C_\upsilon L^{\gamma/N}}\psi(s)\right)$. As additionally $f_{\text{a}}(x, 0) ={\psi(0)}$, Theorem~\ref{theo:doubling} applies.

\end{proof}}

\printbibliography
\end{document}